\RequirePackage{fix-cm}
\documentclass[10pt,oneside,english]{amsart}
\usepackage{mathpazo}
\usepackage{eulervm}
\usepackage[T1]{fontenc}
\usepackage[latin9]{inputenc}
\usepackage{color}
\usepackage[margin=.8in]{geometry}
\usepackage{geometry}
\usepackage{babel}
\usepackage{units}
\usepackage{amstext}
\usepackage{amsthm}
\usepackage{amssymb}
\usepackage{setspace}
\usepackage[unicode=true,pdfusetitle,
 bookmarks=true,bookmarksnumbered=false,bookmarksopen=false,
 breaklinks=false,pdfborder={0 0 1},backref=false,colorlinks=true]
 {hyperref}

\makeatletter
\numberwithin{equation}{section}
\numberwithin{figure}{section}
\theoremstyle{plain}
\newtheorem{thm}{\protect\theoremname}[section]
\theoremstyle{definition}
\newtheorem{defn}[thm]{\protect\definitionname}
\theoremstyle{definition}

\theoremstyle{remark}
\newtheorem{rem}[thm]{\protect\remarkname}
\theoremstyle{plain}
\newtheorem{cor}[thm]{\protect\corollaryname}
\theoremstyle{definition}
\newtheorem{example}[thm]{\protect\examplename}
\theoremstyle{plain}
\newtheorem{prop}[thm]{\protect\propositionname}
\theoremstyle{plain}
\newtheorem{lem}[thm]{\protect\lemmaname}
\theoremstyle{plain}
\newtheorem{fact}[thm]{\protect\factname}
\theoremstyle{plain}
\newtheorem*{thm*}{\protect\theoremname}
\theoremstyle{remark}
\newtheorem{claim}[thm]{\protect\claimname}

\usepackage[psamsfonts]{eucal}
\DeclareMathOperator{\VFA}{VFA}
\DeclareMathOperator{\ACFA}{ACFA}
\DeclareMathOperator{\ACVF}{ACVF}
\DeclareMathOperator{\OGA}{OGA}
\DeclareMathOperator{\FA}{FA}
\DeclareMathOperator{\Spec}{Spec}
\DeclareMathOperator{\VF}{VF}

\DeclareMathOperator{\tp}{tp}

\DeclareMathOperator{\tdeg}{tdeg}
\DeclareMathOperator{\acl}{acl}
\DeclareMathOperator{\AKE}{AKE}
\DeclareMathOperator{\dimm}{dim}

\makeatother

\providecommand{\claimname}{Claim}
\providecommand{\corollaryname}{Corollary}
\providecommand{\definitionname}{Definition}
\providecommand{\examplename}{Example}
\providecommand{\factname}{Fact}
\providecommand{\lemmaname}{Lemma}
\providecommand{\problemname}{Problem}
\providecommand{\propositionname}{Proposition}
\providecommand{\remarkname}{Remark}
\providecommand{\theoremname}{Theorem}

\begin{document}
\title{Specialization of Difference Equations and High Frobenius Powers} 
\author{Yuval Dor}  
\author{Ehud Hrushovski}
\address{yuval.dor4@gmail.com}
\address{hrushovski@maths.ox.ac.uk, Mathematical Institute, Oxford}

\begin{abstract}

 We study valued fields equipped with an automorphism $\sigma$ which is
locally infinitely contracting in the sense that $\alpha\ll\sigma\alpha$ for
all $0<\alpha\in\Gamma$. We show that various notions of valuation theory, such
as Henselian and strictly Henselian hulls, admit meaningful transformal analogues. We prove canonical amalgamation results, and exhibit the way  that transformal wild ramification is
controlled by torsors over generalized vector groups. Model theoretically, we
determine the model companion:   it is decidable, admits a simple axiomatization, and enjoys elimination of quantifiers up to algebraically
bounded quantifiers.

  The model companion is shown to agree with the limit theory of the Frobenius
action on an algebraically closed and nontrivially valued field. This opens the
way to a motivic intersection theory for difference varieties that was
previously available only in characteristic zero. As a first consequence, the
class of algebraically closed valued fields equipped with a distinguished
Frobenius $x\mapsto x^{q}$ is decidable, uniformly in $q$.
\end{abstract}

\maketitle
\tableofcontents{}

\newpage{}

\section*{Introduction}

\subsection{Motivation and Background} By a \emph{difference field} we mean a field $K$ equipped with a distinguished endomorphism $\sigma \colon K \to K$ of fields. Difference fields encode finite difference equations; these are analogous to differential equations, but involving a discrete difference operator $\sigma$ rather than a derivation. Here $\sigma$ is treated as the generic endomorphism: the only rules are $(x+y)^\sigma = x^\sigma + y^\sigma$ and $(xy)^\sigma = x^\sigma y^\sigma$.

In \cite{hrushovski2004elementary} a study of the geometry of difference equations was initiated. The notion of a difference variety cut out by difference equations makes sense intrinsically, and various notions of algebraic geometry, such as smoothness and dimension, admit meaningful analogues in these more generalized settings.

From a model theoretic point of view, the ambient geometry for the study of difference equations is provided by the theory $\ACFA$, the model companion of the class of difference fields. The models of $\ACFA$ obey a generalized form of the Nullstellensatz; their role in the category of difference fields is analogous to the role of algebraically closed fields in the category of ordinary fields.

One of the main motivations for the present work is the intriguing connection between $\ACFA$ and the model theory of the Frobenius endomorphism. We say that $K$ is a \emph{Frobenius difference field} if it is of positive characteristic $p$ and $\sigma = p^n$ coincides with a power of the Frobenius endomorphism of $K$. With this terminology, we have the following:

\begin{thm}
\label{thm:frobacfa}
\begin{enumerate}
\item The class of difference fields admits a model companion, denoted by $\ACFA$.
\item The theory $\ACFA$ is precisely the limit theory of the Frobenius action on an algebraically closed field, i.e, it is the set of sentences true over all algebraically closed Frobenius difference fields, outside a finite set of exceptional prime powers.
\item The elementary theory of the class of algebraically closed Frobenius difference fields is decidable
\end{enumerate}
\end{thm}

This was proved in \cite{hrushovski2004elementary}.  Later, 
the key algebro-geometric estimate was given a new and purely algebro-geometric proof in  \cite{shuddhodan2021hrushovski}. 

Theorem \ref{thm:frobacfa} is a part of a far reaching connection between transformal geometry and asymptotic algebraic geometry, mediated by the Frobenius. Loosely speaking, the Frobenius specialization of a difference equation is obtained by moving to characteristic $p$ and formally setting the generic endomorphism $\sigma$ to agree with the Frobenius endomorphism $x \mapsto x^q$. The idea advanced in \cite{hrushovski2004elementary} is that for $q$ large, properties of the original equation are faithfully reflected in the Frobenius specialization. Most importantly, under the specialization $\sigma \mapsto q$, the model theoretic notion of total dimension $n$ is intertwined with the combinatorial notion of cardinality $c \cdot q^n$.

\subsection{Main Results} The main goal of the present work is to show that the transfer principle of Theorem \ref{thm:frobacfa} remains valid over valuation rings. This is the first step in establishing the framework for a theory of specialization of difference equations initiated in \cite{hrushovski2004elementary}.

Characteristic of the Frobenius action is its highly contracting nature on the open unit ball. This is captured by the following definition, which is the central object of study in this work:

\begin{defn}
\begin{enumerate}
\item By a \emph{transformal valued field} we mean a valued field $K$ equipped with a distinguished endomorphism $\sigma \colon K \to K$ of fields such that $\sigma^{-1}\left(\mathcal{O}\right) = \mathcal{O}$ and $\sigma^{-1}\left(\mathcal{M}\right) = \mathcal{M}$
\item Let $K$ be a transformal valued field. We say that $K$ is $\omega$-\textit{increasing}
if the induced action of $\sigma$ on $\Gamma$ is rapidly increasing:
to wit, we have the inequality $n\alpha<\sigma\alpha$ for all $0<\alpha\in\Gamma$
and $n\in\mathbf{N}$. 
\item Write $\VFA$\footnote{See Remark \ref{rem-terminology} for a note on terminology} for the theory of $\omega$-increasing transformal valued fields which are perfect and inversive; that is, the Frobenius endomorphism and $\sigma$ are both onto.\end{enumerate}
\end{defn}

The theory $\VFA$ was introduced in \cite{hrushovski2004elementary}, but strong structure theory was only obtained under finite generation assumptions. Loosely speaking, the uniformization theorem of \cite{hrushovski2004elementary} can be viewed as a transformal analogue of the classification of local fields of positive characteristic. In characteristic zero, Durhan (formerly Azgin \cite{azgin2010valued}) proved the existence of a relative model companion of $\VFA$.

Here we work in the absolute settings and with no assumption on characteristic. In view of Theorem \ref{thm:frobacfa}, one is led to conjecture that the asymptotic theory of the Frobenius action on an algebraically closed and nontrivially valued field is the model companion of $\VFA$, and this indeed turns out to be the case:

\begin{defn}
\label{def-wvfa-tilde}Let $\widetilde{\VFA}$ be the following
first order theory of $\omega$-increasing transformal valued fields
$K$:

\begin{enumerate}

\item The field $K$ is \textit{transformally Henselian}\footnote{See Remark \ref{rem-terminology} for a note on terminology}: residual simple
roots lift to integral roots (see Definition \ref{def:transformally-henselian}
for a precise statement)

\item The residue field $k$ of $K$ is an existentially closed difference
field; that is, it is a model of $\ACFA$

\item The valuation group $\Gamma$ of $K$ is nonzero and transformally
divisible; that is, we have $\nu\Gamma=\Gamma$ for all $0\neq\nu\in\mathbf{Z}\left[\sigma\right]$

\item Let $\tau x$ be a nonconstant additive difference operator given
by a linear combination of difference monomials of the form $x^{\sigma^{n}\cdot p^{m}}$
where $n,m\in\mathbf{N}$; then $\tau$ is onto on $K$-points.
\end{enumerate}
\end{defn}

We prove the following:

\begin{thm}
\label{main-thm}

\begin{enumerate}

\item Every model of $\VFA$ embeds in a model of $\widetilde{\VFA}$;
the theory $\widetilde{\VFA}$ is model complete.

\item Let $F$ be an algebraically closed model of $\VFA$. Then
the theory $\widetilde{\VFA_{F}}$ of models of $\widetilde{\VFA}$
over $F$ is complete; hence the theory $\widetilde{\VFA}$
eliminates quantifiers to the level of the field theoretic algebraic
closure.

\item In $\widetilde{\VFA}$, the residue field $k$ and the valuation
group $\Gamma$ are stably embedded and model theoretically fully
orthogonal; the induced structure is the abstract difference field
and the abstract ordered module structure, respectively.

\item Let $F\subseteq K$ be an embedding of models of $\VFA$
with $K$ a model of $\widetilde{\VFA}$. Then $F$ is model
theoretically algebraically closed in $K$ if and only if it is field
theoretically algebraically closed and transformally Henselian.
\end{enumerate}
\end{thm}

Using Theorem \ref{thm:frobacfa}, we deduce:

\begin{cor}
\label{frob:non}
Let $\widetilde{T}$ be the limit theory of the Frobenius action on an algebraically closed and nontrivially valued field; that is, the set of sentences true over all algebraically closed and nontrivially valued Frobenius transformal valued fields, outside a finite set of exceptional prime powers. Then the theory $\widetilde{T}$ is model complete and decidable; indeed, it coincides precisely with the model companion $\widetilde{\VFA}$ of $\VFA$.
\end{cor}

Some remarks on Theorem \ref{main-thm} are in order.

\textbf{Model Theoretic Tameness}. The first clause of Theorem \ref{main-thm} implies that
the models of $\widetilde{\VFA}$ are existentially closed. 
This leads to the 
second clause, that places $\widetilde{\VFA}$ in the model-theoretic `tame' world:
the category of definable sets is of directly geometric nature. In particular,
deep diophantine questions or G\"{o}delian phenomena do not enter the
picture. The stable embeddedness of the residue field has the highly
nontrivial consequence that the image of a difference variety under
the residue map is already definable in the language of difference
fields, namely it is the image of a difference variety under a finite
morphism, or a boolean combination of definable sets of this form.
The same holds for a more sophisticated `derived' residual image, implicit in the statement
that transformal dimension zero sets in $\VFA$ are residually analyzable; see Remark \ref{rem-rationality} for further discussion and an application.

\textbf{Quantifier Elimination.}
In Proposition \ref{qe-wvfa} we give a precise (though somewhat inexplicit) quantifier elimination result for the theory $\widetilde{\VFA}$, strengthening the second clause of Theorem \ref{main-thm}. In Theorem \ref{TheAmalgamationProperty} we give a partial characterization of the amalgamation bases; this characterization is complete if higher powers of $\sigma$ are taken into account. Example \ref{example-no-qe-charp-wildly-ramified} and Corollary \ref{cor-nonsplit-h-finite} show that our quantifier elimination results are in some sense optimal. For a comparison with the quantifier elimination results of Durhan (requiring the choice of an angular coefficient map) see Remark \ref{rem-ake-splitting-qe}.

\textbf{The Ramification Theory of Valued Fields.} Let $K$ be a valued field which is strictly Henselian; that is, the valued field $K$ is Henselian, and the residue field is separably closed. Then $K$ is algebraically closed precisely in the event that its valuation group is divisible and additive operators given by linear combinations of the Frobenius are onto. The axiomatization given in Theorem \ref{main-thm} and Definition \ref{def-wvfa-tilde} can be seen as a transformal analogue of this fact.

Wildly ramified extensions of Henselian valued fields are controlled by additive equations involving powers of the Frobenius; the situation in $\VFA$ is similar (see Section \ref{sec:Wild-Ramification}).

Going further, if $K$ is strictly Henselian then the absolute Galois group of $K$ is solvable, with subquotients non-canonically isomorphic to finite linear algebraic groups over the residue field. As already suggested in \cite{hrushovski2004elementary}, this phenomenom is related to the fact that in $\widetilde{\VFA}$, transformally algebraic extensions are analyzable over the residue field.

\textbf{The Transformal Henselization.} The notion of \textit{transformal Henselianity} was introduced
``tentatively'' in \cite{hrushovski2004elementary}. The definition is 
a direct transformal analogue of Hensel's
lifting lemma, see Definition \ref{def:transformally-henselian}.
We believe that the tentative status of this Definition can now be
safely removed. If $K$ is a model of $\VFA$, then $K$ admits a
 transformally Henselian hull, which is unique up to a
unique isomorphism; we also give a relative version (a strict transformal Henselization) which is canonical relative to a given extension of the residue field, and whose finite covers are controlled; see Theorem \ref{transformalhenselization} for a precise statement.

If $K$ is a Henselian valued field, then the valuation lifts uniquely to any algebraic extension; as a matter of fact, Henselian valued fields are characterized by this property. In the $\omega$-increasing regime, this restriction leads to a vacuous theory, as Example \ref{ex:hens-talg-extensions-not-unique} shows. Nevertheless, working over a transformally Henselian base, once the valuation of a transformally algebraic extension is fixed, then it extends uniquely to any larger base; see Corollary \ref{cor-talg-unique-basechange} for a precise statement and further discussion.

\textbf{The work of Durhan-Van-den-Dries} In characteristic zero, the existence of a model companion of $\VFA$ is due to Durhan (formerly Azgin, \cite{azgin2010valued}). His work involves a Newton-Hensel style approximation scheme applied to the coefficients of a transformal operator. Earlier work using this idea includes Scanlon's thesis (see \cite{scanlon2000model}), work of Belair-Macintyre-Scanlon \cite{belair2007model}, and the work of Durhan-Van den Dries \cite{azgin2011elementary}

Let $K$ be a model of $\VFA$ of characteristic zero. We say that $K$ is \textit{transformally Kaplansky} if nonconstant additive operators given by linear combinations of difference monomials of the form $x^{\sigma^n}$ are onto on the residue field of $K$. Durhan shows that if $K$ is transformally Kaplansky, then $K$ admits a spherically complete immediate extension which is unique up to isomorphism. Furthermore, he proves that the class of models of $\VFA$ which are transformally Kaplansky and transformally algebraically maximal is first order, namely they are characterized by the fact that they obey the Newton-Hensel style approximation scheme mentioned above. Finally, fields of this form enjoy elimination of field quantifiers in the three sorted language, after choosing a splitting of the valuation map.

\textbf{The Difficulties in Finite Characteristic.} The key to Durhan's proof (as well as any other model completeness theorem for enriched valued fields we are aware of) is a uniqueness statement for spherically complete immediate extensions. Example \ref{exa:non-unique-spherically} and Corollary \ref{immediate-non-unique} show that no statement of this form can hold in positive characteristic, no matter the assumption on the residue field. 

Durhan's work gives a quantifier elimination result relative to the residue field and the valuation group, after splitting the valuation map. In positive characteristic, one must deal with the twisted Artin-Schreier extensions of Example \ref{example-no-qe-charp-wildly-ramified}, which are typical; Corollary \ref{cor-no-qe-relative-tores} shows that elimination of field quantifiers in the three sorted language is impossible in general.

Heuristically, the major difficulty in finite characteristic boils down to the following phenomenon. Loosely speaking, in $\VFA_{\mathbf{F}_{p}}$ one considers \textit{partial}
difference equations, involving $\sigma$ and the Frobenius; this
is in contrast with the study of $\VFA_{\mathbf{Q}}$, or the
study of Frobenius transformal valued fields, where only one of these
automorphism is present.

\textbf{Previous Work.} Valued fields of equal characteristic zero equipped with an automorphism were studied by various authors. In his thesis \cite{pal2012multiplicative} Pal considers the case where the action of $\sigma$ on $\Gamma$ is via multiplication by a positive rational, or more generally via multiplication by a positive real (appropriately defined). Durhan-Van-den-Dries study the isometric case \cite{azgin2011elementary} where the induced action of $\sigma$ on $\Gamma$ is the identity. In \cite{durhan2015quantifier} Durhan and Onay establish a uniqueness result for spherically complete immediate extensions with no assumptions whatsoever on the induced action on the valuation group.

Model theoretic properties of $\widetilde{\VFA_{\mathbf{Q}}}$ were studied in the work of Chernikov-Hils and Hils-Rideau (\cite{chernikov2014valued}, \cite{hils2021ei}). Chernikov and Hils show that in $\widetilde{\VFA_{\mathbf{Q}}}$, forking independence is governed by quantifier free formulas, which are \text{NIP}; in particular, in Shelah's classification, the theory $\widetilde{\VFA_{\mathbf{Q}}}$ enjoys the $\text{NTP}_2$ property. It seems likely that the results hold in finite characteristic as well; to make this precise, one would need to refine the amalgamation theorem of \ref{TheAmalgamationProperty} to give control of finite $\sigma$-invariant Galois extensions of the amalgamation, as in the construction of the strict transformal Henselization (Theorem \ref{transformalhenselization}), though we do not pursue this at the moment. The work of Hils-Rideau \cite{hils2021ei} gives elimination of imaginaries in the geometric sorts of \cite{haskell2005stable}; the result follows from a more general Ax-Kochen-Ershov principle which reduces the imaginary sorts of an enriched valued field of equal characteristic zero to to those of the residue field and the valuation group.

\begin{rem}
    In this work all transformal valued fields are assumed perfect and inversive. This assumption has been dispensed with in subsequent work of the first author and Halevi (see \cite{dor2023contracting}). In particular, the asymptotic theory of the class of separably closed Frobenius transformal valued fields (which need not be perfect) is decidable.
\end{rem}

\begin{rem}
\label{rem-terminology} Some remarks on terminology and choices made in this paper.
\begin{enumerate}
\item In this work all transformal valued fields are assumed perfect and inversive unless explicitly stated to the contrary. The assumption that the Frobenius is onto is needed to make sense of the notion of a Frobenius twist (see Section \ref{sec:Difference-Algebra}) and the assumption that $\sigma$ is onto for the notion of a $\sigma$-invariant Galois extension to be meaningful (see the criterion of Proposition \ref{criterion-unique-algclosure}). Since a perfect inversive hull exists and is unique up to a unique isomorphism, for the purpose of establishing the existence of a model companion, these assumptions are harmless. In \cite{hrushovski2004elementary} finite generation constraints play an essential role; in order for ramification to be visible at the level of the value group, it is necessary not to assume that the fields are inversive.
    \item In this work we use the acronym $\VFA$ to describe the class of $\omega$-increasing transformal valued fields which are perfect and inversive. This terminology is potentially misleading, since one can consider valued fields equipped with an automorphism subject to various other constraints (see e.g Scanlon's thesis \cite{scanlon2000model}). It would perhaps be more accurate to use the terminology $\omega\VFA$ or $\VFA^{\omega}$ instead. For typographical reasons, we will stick with $\VFA$; since the $\omega$-increasing constraint is imposed throughout the paper, this convention is not likely to cause confusion.
\item In this work the notion of transformal Henselianity (Definition \ref{transformalhenselization}) plays a central role. This notion should not be confused with the closely related notion of $\sigma$-Henselianity studied in \cite{azgin2010valued} and elsewhere in the literature. The ultraproduct of Henselian Frobenius transformal valued fields is transformally Henselian, but not in general $\sigma$-Henselian. On the other hand, a model of $\VFA$ of characteristic zero is $\sigma$-Henselian precisely when it has no immediate transformally algebraic extensions, and transformal additive linear operators are onto on the residue field; in view of Corollary \ref{frob:non} it seems natural to stick with our terminology, whereas the class of $\sigma$-Henselian fields can be rightfully called "transformally Kaplansky and transformally algebraically maximal". The $\sigma$-Henselian approximation scheme can then either be viewed as an interesting notion in its own right or else as a first order characterization of the class of transformally Kaplansky, transformally algebraically maximal models of $\VFA$.

\end{enumerate}
\end{rem}

\begin{rem}
\label{rem:transformal-analogues}

    One of the central goals of this paper is to define and study analogues of classical notions of valuation theory in the presence of an $\omega$-increasing automorphism; for example, transformally Henselian and strictly transformally Henselian models of $\VFA$ are in many ways analogous to Henselian and strictly Henselian valued fields in the category of ordinary valued fields.

    On occasion however we will want to refer to the the underlying valued field of a model of $\VFA$, ignoring $\sigma$. Thus, for example, we use the term "algebraically Henselian" to simply mean "Henselian" (as opposed to transformally Henselian).

\end{rem}

\begin{rem}
    The transformal analogues of Remark \ref{rem:transformal-analogues} are motivated by the case of Frobenius transformal valued fields. For example, a transformally unramified extension of models of $\VFA$ (see Definition \ref{def:tunramified}) is analogous to an unramified extension of ordinary valued fields, viewed as transformal valued fields via the Frobenius endomorphism. This general heuristic is supported by Corollary \ref{frob:non}.
\end{rem}

\begin{rem}
\label{rem-rationality}
In earlier unpublished work, one of us (EH) constructed \textit{Weil-Cavalieri
rings} that generalize the Grothendieck rings of pseudo-finite fields.
The Weil-Cavalieri rings are built out of definable sets in ACFA of
transformal dimension zero (generalizing definable sets over pseudo-finite
fields). However, while equality of classes in the Grothendieck ring
takes into account only definable bijections, the Weil-Cavalieri equivalence
is able to incorporate a \textit{principle of invariance of number}
a la Leibniz, Poncelet and Schubert: in a continuously moving family
of definable sets of transformal dimension zero, the classes remain
constant. Here continuity must be understood to refer to the generic
type of the difference field; whereas pseudofinite fields are in themselves
discrete. Using this depends essentially on the consequences
of Theorem \ref{main-thm} (3) mentioned above. See the account in
\cite{giabicani2011theorie}; the restriction to characteristic zero
there can now, thanks to the present work, be removed.

We mention without proof a corollary that we hope to come back to
later, combining ideas from \cite{hrushovski2004elementary} with
the Weil-Cavalieri ring. For large characteristic compared to the
data, it was the content of the unpublished MA thesis of one of us
(YD) in the Hebrew University. This large characteristic assumption
can now be removed. When the definable set $X$ does not utilize the
valuation, we understand this can also be proved using the methods
of \cite{shuddhodan2021hrushovski}.

Fix a prime $p$. Let $K$ be an algebraically closed and nontrivially
valued field of characteristic $p>0$. For each $n\in\mathbf{N}$
let $K_{n}=\left(K,x\mapsto x^{p^{n}}\right)$ be the displayed Frobenius
transformal valued field. Let $F=F(x,x^{\sigma},\cdots,x^{\sigma^{N}})$
be a nonzero difference polynomial over $K$, and denote $Z_{F}=\{x:F(x)=0\}$.
Let $X$ be a definable subset of $Z_{F}$ in the language of transformal
valued fields. Then 
\[
|X(K_{n})|=\sum_{i=1}^{m}\alpha_{i}t_{i}^{n}
\]
for some $t_{1},\ldots,t_{m},\alpha_{1},\ldots,\alpha_{m}\in\mathbf{Q}^{alg}$,
and large enough $n\in\mathbf{N}$. The cardinality of $X\left(K_{n}\right)$
is $c\cdot p^{nd}+O\left(p^{n\cdot\left(d-\frac{1}{2}\right)}\right)$
for some natural number $d$ and some rational number $0<c\in\mathbf{Q}$.
The natural number $d$ is the \textit{inertial dimension} of $X$,
studied in \cite{hrushovski2004elementary}.
\end{rem}

\subsection{Scatteredness and Rationality of Cuts}

One of the main themes of this paper is that the $\omega$-increasing
nature of $\sigma$ severely restricts the ultrametric structure on
transformally algebraic extensions of models of $\VFA$. The
purpose of this section is to make this precise and explain the algebraic
and model theoretic ideas surrounding this principle. We also give
examples illustrating the proof of the amalgamation theorem.

\subsubsection{\label{subsec:Functorial-Amalgamation-of}Functorial Amalgamation
of $\mathbf{R}$-valued Fields}

Let $\mathcal{C}$ be the category of complete algebraically closed
$\mathbf{R}$-valued fields. A theorem of Poineau \cite{poineau2013espaces}
shows that $\mathcal{C}$ admits \textit{functorial}, linearly disjoint
amalgamation. Here is the idea in the case of a single variable; see
\cite{ben2015tensor} and \cite{hrushovski2010non}, 14.6.3 for more.
Let $K$ be a complete $\mathbf{R}$-valued field which is algebraically
closed. Let $L$ be a spherically complete immediate extension, and
$E=K\left(x\right)$ a simple transcendental immediate extension;
we want to equip the field of fractions of $E\otimes_{K}L$ with a
valuation in a functorial manner.

For this we argue as follows. Let $\mathcal{B}$ be the nested family
of all closed $K$-balls containing $x$. Since $K$ is complete,
the set $I$ of valuative radii of the balls in $\mathcal{B}$ is
bounded; hence it admits a least upper bound $\alpha\in\mathbf{R}$.
There is in $L$ a \textit{unique} closed ball $B$ of valuative radius
$\alpha$ contained in the intersection. The amalgamation is given
by declaring $x$ to be generic in $B$ over $L$, i.e, to avoid any
finite union of open $L$-definable subballs.

\subsubsection{Rationality of Cuts}

One would like to use a similar argument in $\VFA$. But the
argument of \ref{subsec:Functorial-Amalgamation-of} used, in a crucial
way, the \textit{least upper bound property} of $\mathbf{R}$. 

To be more precise, let $\nicefrac{L}{K}$ be an extension of valued
fields which does not enlarge the valuation group. Given $a\in L$
let us write
\[
C\left(a\right)=\left\{ v\left(a-b\right)\colon b\in K\right\} 
\]
for the \textit{cut} determined by the element $a$ over $K$. The
key question becomes whether this cut is \textit{rational}, i.e, whether
it is the cut just above some element of $\Gamma$. For complete $\mathbf{R}$-valued
fields this is automatic. But in the $\omega$-increasing setting,
over a base which is nontrivially valued, the valuation group is never Archimedean (see \ref{defs:oag}); one must prove that the cut is in fact rational!
\begin{example}
Let 
\[
x=t^{\sigma^{-1}}+t^{\sigma^{-1}+\sigma^{-2}}+\ldots
\]
 be the displayed formal power series. One checks that $x^{\sigma}-tx=t$. The sequence of truncations of $x$ is a pseudo-Cauchy sequence which converges to $x$; the associated sequence of valuative radii of balls is given by the series of exponents $\sigma^{-1}+\dots+\sigma^{-n}$, which  converge
in $\mathbf{Q}\left(\sigma\right)$ to the quantity $\frac{1}{\sigma-1}$. Thus as expected the cut associated to $x$ is rational.
\end{example}

The Example is typical:
\begin{prop}
\label{prop:rationality-of-cuts-1}Let $\nicefrac{L}{K}$ be an immediate
transformally algebraic extension of models of $\VFA$. Then
under suitable assumptions on $K$, the cut determined by an element
of $L$ over $K$ is rational; see Corollary \ref{cor:rat-cut} for
a precise statement.
\end{prop}

We note also that in   \cite{kedlaya2001algebraic}, 
Kedlaya has identified the algebraic closure of $\mathbf{F}_{p}\left(\left(\mathbf{Z}\right)\right)$
within $\mathbf{F}_p\left(\left(\mathbf{Q}\right)\right)$ the generalized power series. It follows from his analysis
that it is sufficient to consider series supported on a well ordered
subset of $\mathbf{Q}$, all of whose accumulation points are \textit{rational} (resonating with our rationality results.)

\subsubsection{\label{subsec:Working-Example-of}Working Example of Amalgamation}

Let us give an example of amalgamation in a model case which will
help illustrate the key ideas. Let $K$ be a completed algebraic closure
of $\mathbf{Q}\left(t^{\mathbf{Z}\left[\sigma^{\pm1}\right]}\right)$,
so the valuation group of $K$ is $\mathbf{Q}\left[\sigma^{\pm1}\right]$.
Let $L=K\left(x\right)=K\left(x\right)_{\sigma}$ be a transformal
valued field extension of $K$ generated by an element $x$ with $x^{\sigma}-x=c$
for some $c\in K$.
\begin{lem}
\label{example-lemma}(1) The cut determined by the element $x$ in
$\mathbf{Q}\left[\sigma^{\pm1}\right]$ is either $\left(-\infty,0\right)$
or $\left(-\infty,0\right]$

(2) The valuation we are given on $L$ is in fact the unique one lifting
the valuation on $K$
\end{lem}

\begin{rem}
Let $K$ be a valued field. Chatzidakis and Perera show in \cite{chatzidakis2017criterion}
that for the valuation to lift uniquely to every Artin-Schreier extension
of $K$, it is a necessary and sufficient that the operator $X\mapsto X^{p}-X$
is surjective on the maximal ideal; this is automatic when $K$ is complete and the value group is Archimedean. In this situation, one shows that the
cut determined by an Artin-Schreier extension of $K$ is again either
$\left(-\infty,0\right)$ or $\left(-\infty,0\right]$, corresponding
to the ramified and unramified cases, respectively.
\end{rem}

\begin{proof}
For the sake of brevity we only sketch the proof.

(1) We have the equality
\[
C\left(x^{\sigma}\right)=C\left(x+c\right)
\]
but one readily verifies that $C\left(x\right)=C\left(x+c\right)$
and $C\left(x^{\sigma}\right)=\sigma\cdot C\left(x\right)$, which
gives
\[
C\left(x\right)=\sigma\cdot C\left(x\right)
\]
so there are only three possibilities, namely $\left(-\infty,0\right),\left(-\infty,0\right]$
and the whole group $\mathbf{Q}\left[\sigma^{\pm1}\right]$; the last
possibility is excluded since $K=\widehat{K}$ is complete. 

(2) In view of $x^{\sigma}=x+c$ it is enough to prove that the quantity
$vfx$ is determined, for a polynomial $f\in K\left[x\right]$; since
$K$ is algebraically closed, by splitting $f$ into linear factors,
it is enough to prove that the quantity $v\left(x-a\right)$ is determined,
for $a\in K$. Now the element $x-a$ obeys an equation over $K$
of the same form, for a different choice of $c$; so by allowing $c$
to vary it is enough to show that the quantity $vx$ is determined.
By (1) we cannot have $vx>0$; so a straightforward computation gives
$vx=\sigma^{-1}\cdot vc$.
\end{proof}
Let us focus on the case where $L$ determines an immediate extension
of $K$, i.e, with the case where the cut determined by $x$ has no
maximal element; this is the case for example when
\[
x=t^{-1}+t^{\sigma^{-1}}+\ldots
\]

Let $\mathcal{B}$ be the nested family of all closed balls with center
and radii in $K$ containing $x$; the valuative radii of the balls
in $\mathcal{B}$ tends to $0$. There is therefore in $L$ a \textit{unique}
closed ball $B$ of valuative radius zero which is contained in the
intersection; this is the closed ball of valuative radius zero around
some (or any) solution of the equation $X^{\sigma}-X=t^{-1}$. We
can then functorially amalgamate, exactly as in \ref{subsec:Functorial-Amalgamation-of}.

\subsection*{Acknowledgments}

We would like to thank Zo{\'e} Chatzidakis and James Tao for useful discussions.
Some of our results were obtained independently by Martin Hils and
G{\"o}nen{\c{c}} Onay.  We are indebted to the anonymous referee for his or her thorough reading.
\newline
This text is based on the Harvard PhD thesis of Yuval Dor, which is in turn a fundamental reworking and extension of a previous joint manuscript.

\newpage{}

\subsection{A Word on Prerequisites}

The present work is in the intersection of model theory, valuation
theory, and difference algebra; we made however an attempt to keep
it as self contained as possible. Only basic concepts of difference algebra and the model theory of difference fields will be used; see Section \ref{sec:Difference-Algebra} for an overview. We assume that the reader is familiar with the basic
theory of valued fields and their ultrametric structure. We will also
refer to the basic ramification filtration of the absolute Galois
group and its functorial properties; an overview is given in the book
by Neukirch \cite{neukirch2013algebraic}. For standard facts on valuation
theory and the model theory of valued fields we refer the reader to \cite{haskell2005stable}.

\subsection{Organization of the Paper}

In Section 1 we introduce the notion of the Artin-Schreier ideals
of a Henselian valued field, a prime ideal of the valuation ring canonically
attached to the data of a wildly ramified Galois extension of degree
$p$. It is shown that these ideals enjoy a certain tameness property
in the Abhyankar settings.

In Section 2 and 3 we rapidly review basic notions from difference
algebra and transformal modules to be used in the rest of the work.

In Section 4 we introduce the central object of interest, namely the
category of models of $\VFA$. We show that basic constructions
in valuation theory can be lifted to the transformal settings and
study algebraic extensions in detail.  
 It will turn out that  a    sufficient condition for a model of $\VFA$ to be an amalgamation basis
 is that it have no nontrivial $\text{h}$-finite extensions, meaning $\sigma$-invariant algebraic extensions that become finite but nontrivial upon taking   
passing to the Henselization.  (This is  actually a necessary condition too, if one considers
 finite powers of the automorphism as well).  
We give a    criterion
for the uniqueness of a separable closure in terms of $\text{h}$-finite
$\sigma$-invariant separable extensions and demonstrate by means
of example that $\text{h}$-finiteness cannot be relaxed to finiteness
in general. We then formulate a criterion which is easier to verify
in practice, in terms of the action of $\sigma$ on certain valuation
theoretic data, namely the Artin-Schreier ideals of Section 1, the
finite index subgroups of the valuation group, and the absolute Galois
group of the residue field.

In Section 5 we study the transformal analogues of the algebraic notions
of Henselization, strict Henselization, and unramified extensions.
In particular, we demonstrate the existence of a transformal
Henselization which is unique up to a unique isomorphism and obeys
formal properties similar to its algebraic counterpart. We then use
the results of Section 1 and Section 4 to show that the strict Henselization
has no unexpected finite $\sigma$-invariant separable extensions,
a statement essentially equivalent to the stable embededness of the
residue field in the model companion $\widetilde{\VFA}$ of
$\VFA$.

In Section 6 we lift the notion of the generic of a ball from $\ACVF$
to $\VFA$ and use it to define a transformal analogue of the
Herbrand-Hasse transition function of local class field theory.

In Section 7 we study immediate extensions, and show that they are
controlled by torsors for the additive group.

In Section 8 we introduce the notion of transformally Archimedean
ultrapowers. We show that the transformally Archimedean ultrapower of an algebrically closed
field has no immediate transformally algebraic extensions.

In Section 9 we assemble the various pieces together; we prove the
amalgamation property and establish the existence of a model companion.

\subsection{Conventions and Notation}

\subsubsection{~}

Recall that the \textit{characteristic exponent} of a field is equal
to $p$ if the field is positive characteristic $p$ and is equal
to $1$ if the field is of characteristic zero. By convention, a natural
number $p$ which is either prime or equal to $1$ is fixed through
this text, and all valued fields are assumed of equal characteristic,
of characteristic exponent $p$.

\subsubsection{~}

In this work, we will work almost exclusively with fields, as opposed
to rings, and for this reason, the following convention is adopted.
Let $L\hookleftarrow K\hookrightarrow M$ be embeddings of fields.
Then the notation $L\otimes_{K}M$ is overridden: it denotes the field
of fractions of the displayed algebra over $K$, which is in turn
implicitly assumed to be an integral domain, an assumption which will
always hold whenever the notation is used.

\subsubsection{~}

Let $K$ be a field. We write $K^{p^{-\infty}},K^{\text{sep}}$ and
$K^{\text{a}}$ for the perfect hull, a separable closure and an algebraic
closure of $K$ respectively. The absolute Galois group of $K$ with
respect to a fixed choice of a separable closure is denoted by $G_{K}$.
The Frobenius endomorphism is denoted by $\phi x=x^{p}$, thus it
is the identity map in exponential characteristic $p=1$.

\subsubsection{~}\label{defs:oag}

Let $\Gamma$ be an ordered abelian group, written additively. We
say that $\Gamma$ is \textit{perfect }if the equality $p\Gamma=\Gamma$
holds, thus in characteristic exponent $p=1$ the group $\Gamma$
is always perfect. 

We shall write $\Gamma_{\infty}$ for the ordered abelian monoid whose
underlying set is $\Gamma$ together with a distinguished element
$\infty$. The ordering of $\Gamma_{\infty}$ extends the ordering
of $\Gamma$ by declaring $\infty$ to be the maximal element and
to be absorbing under addition. If there is no possible room for confusion,
we will fail to distinguish between $\Gamma$ and $\Gamma_{\infty}$,
suppressing $\infty$ from the notation. By a \textit{cut} in $\Gamma$
we mean a subset of $\Gamma$ which is downwards closed.

By an \textit{ideal} of $\Gamma_{\infty}$ we mean a nonempty subset $I$
of $\Gamma_{\infty}$ which is upwards closed and whose elements are
all nonnegative. We say that $I$ is a \textit{prime ideal} if whenever
the sum $\sum_{s\in S}\alpha_{s}$ of finitely many strictly positive
elements of $\Gamma$ lies in $I$ then already one of them does.
We write $\mathrm{Spec}\Gamma=\mathrm{Spec}\Gamma_{\infty}$ for the
space of prime ideals of $\Gamma_{\infty}$. If $V$ is a convex subgroup
of $\Gamma$, then the set $\gamma\in\Gamma_{\infty}$ with $\alpha<\gamma$
for all $\alpha\in V$ is a prime ideal of $\Gamma_{\infty}$. Moreover,
this association induces a bijection between prime ideals of $\Gamma_{\infty}$
and convex subgroups of $\Gamma$. 

If $0<\alpha,\beta\in\Gamma$ then we write $\alpha\ll\beta$ to indicate
that the inequality $n\alpha<\beta$ holds for all $0<n\in\mathbf{N}$.
Equivalently, the convex subgroup generated by $\beta$ strictly contains
the convex subgroup generated by $\alpha$. If the relations $\alpha\ll\beta$
and $\beta\ll\alpha$ both fail to hold then $\alpha$ and $\beta$
are said to lie in the same \textit{Archimedean equivalence class}
of $\Gamma$. We say that $\Gamma$ is \emph{Archimedean} if every pair of positive elements lie in the same Archimedean equivalence class; this is equivalent to the requirement that $\Gamma$ embeds in $\mathbf{R}$.

\subsubsection{\label{embedding-of-valued-fields}}

Let $K$ be a valued field, per our conventions, always assumed equicharacteristic
of characteristic exponent equal to $p$. We write $\mathcal{O}_{K},k_{K},\Gamma_{K}$
and $\mathcal{M}_{K}$ for the valuation ring, the residue field,
the valuation group and the maximal ideal, respectively, and we drop
the subscript if there is no possible room for ambiguity. The valuation
map is denoted by $v:K\to\Gamma_{\infty}$ and our notation is additive.
The map $I\mapsto v\left(I\right)$ is a bijection between ideals
of the valuation ring $\mathcal{O}$ and ideals of $\Gamma_{\infty}$;
it restricts to a bijection $\mathrm{Spec}\mathcal{O}=\mathrm{Spec}\Gamma$
between prime ideals of $\mathcal{O}$ and prime ideals of $\Gamma_{\infty}$.

If $K$ and $\widetilde{K}$ are valued fields then by an \textit{embedding}
of $K$ in $\widetilde{K}$ we mean an embedding $i:K\hookrightarrow\widetilde{K}$
of fields with the property that $i^{-1}\left(\widetilde{\mathcal{O}}\right)=\mathcal{O}$
and $i^{-1}\left(\widetilde{\mathcal{M}}\right)=\mathcal{M}$. If
we view $K$ as a subfield of $\widetilde{K}$ by means of the map
$i$ then we will also say that $\widetilde{K}$ is a valued field
over $K$ and that $\nicefrac{\widetilde{K}}{K}$ is an extension
of valued fields. If $\nicefrac{\widetilde{K}}{K}$ is a valued field
extension then $\widetilde{K}$ is said to be an \textit{immediate}
extension of $K$ if $k=\widetilde{k}$ and $\Gamma=\widetilde{\Gamma}$.

\subsubsection{~}

Let $\mathcal{L}_{\VF}$ be the language of valued fields. The language
$\mathcal{L}_{\VF}$ is one sorted with a sort $K$ for the field,
unary predicates $K^{\times},\mathcal{O}$ and $\mathcal{M}$, constant
symbols $0$ and $1$ and a unary function symbol $\left(\bullet\right)^{-1}\colon K^{\times}\to K^{\times}$.
We let $\VF$ denote the theory of valued fields in the language $\mathcal{L}_{\VF}$;
the axioms and the interpretations of the symbols are the obvious
ones. If $K,\widetilde{K}$ are valued fields then an embedding of
$K$ in $\widetilde{K}$ in the category of structures for the language
$\mathcal{L}_{\VF}$ is an embedding of valued fields. We let $\ACVF=\ACVF_{p,p}$
denote the theory of nontrivially valued algebraically closed fields
of equal characteristic exponent $p$; then $\ACVF_{p,p}$ admits
elimination of quantifiers in the language $\mathcal{L}_{\VF}$, see
\textit{\cite{haskell2005stable}.}

\subsubsection{~}
\label{subsec-ramificationtheory}
Let $K$ be a valued field. Recall that $K$ is said to be \textit{Henselian}
if whenever $fx\in\mathcal{O}\left[x\right]$ is a polynomial with
coefficients in the valuation ring and $a\in\mathcal{O}$ is such
that $vfa>0$ and $vf'a=0$, then there is a unique $b\in\mathcal{O}$
with $fb=0$ and whose residue class coincides with that of $a$.
This is equivalent to the requirement that the integral closure of $\mathcal{O}$
within any algebraic field extension of $K$ be a valuation ring,
and also the requirement that every algebraic field extension of $K$
can be uniquely expanded to a valued field extension of $K$. If $K$
is a valued field then there is an embedding $i:K\hookrightarrow K^{h}$
of $K$ in a Henselian valued field $K^{h}$ which embeds uniquely
over $K$ in any other Henselian extension of $K$; we say that $K^{h}$
is the \textit{Henselization} of $K$. The Henselization of $K$ is
unique up to a unique isomorphism of valued fields over $K$. It is
moreover an immediate extension of $K$ which is separably algebraic
over $K$.

Let us assume that we are given a separably algebraic extension $\widetilde{k}$
of the residue field $k$ of $K$. Then there exists a Henselian valued
field extension $\widetilde{K}$ of $K$ reproducing the embedding
$k\hookrightarrow\widetilde{k}$ residually, and enjoying the following
universal mapping property. Let $L$ be a Henselian valued field extension
of $K$; then every embedding of $\widetilde{k}$ in $l$ over $k$
lifts to an embedding of $\widetilde{K}$ in $L$, and the lifting
is unique. If $\widetilde{k}$ is a separable algebraic closure of
$k$, then $\widetilde{K}$ is said to be a \textit{strict Henselization
}of $K$; the field $K$ is said to be \textit{strictly Henselian}
if it is Henselian and its residue field is separably
closed. We denote a strict Henselization of $K$ by $K^{\text{sh}}$;
it is unique up to isomorphism of valued fields over $K$. If $\nicefrac{\widetilde{K}}{K}$
is an extension of valued fields then $\widetilde{K}$ is \textit{unramified}
over $K$ if it embeds over $K$ in a strict Henselization of $K$.

The valued field $K$ is said to be \textit{tamely closed} if it is
strictly Henselian every finite extension of $K$ is of degree divisible
by $p$; this is equivalent to the requirement that $K$ is strictly
Henselian and $m\Gamma=\Gamma$ for every $0<m\in\mathbf{N}$ not
divisible by $p$. If $K$ is a valued field then there is a valued
field extension $K^{t}$ of $K$ which is tamely closed and which
embeds over $K$ in every other tamely closed extension; we say that
$K^{t}$ is a \textit{tame closure} of $K$, and an algebraic extension
of $K$ is tamely ramified if it embeds over $K$ in a tame closure.

We say that an algebraic extension $\nicefrac{\widetilde{K}}{K}$
of $K$ is \textit{purely wildly ramified} if it is linearly disjoint
over $K$ from $K^{\text{t}}$ within a fixed algebraic closure. Equivalently,
the integral closure of $\mathcal{O}_{K}$ in $\widetilde{K}$ is
a valuation ring, the induced extension $\nicefrac{\widetilde{k}}{k}$
of residue fields is purely inseparably algebraic and $\nicefrac{\widetilde{\Gamma}}{\Gamma}$
is $p^{\infty}$-torsion. If $\nicefrac{\widetilde{K}}{K}$ is finite and purely
wildly ramified then it is of degree a power of $p$.

In summary, let $K$ be a Henselian valued field and let $L$ be a finite Galois extension of $K$. There is a unique maximal extension $L_1$ of $K$ in $L$ which is unramified; this extension is Galois. Moreover, there is a unique maximal extension $L_2$ of $L_1$ in $L$ which is tamely ramified; it is again Galois. The field $L$ is then purely wildly ramified over $L_2$ and the induced factorization $K \subseteq L_1 \subseteq L_2 \subseteq L$ is functorial. The Galois extension $L$ is said to be \emph{totally ramified} if $L_1 = L$; it is said to be tamely (but totally) ramified if $L = L_2$; and it is said to be \emph{wildly ramified} if $L$ properly extends $L_2$.

\subsubsection{~}

Let $K$ be a valued field. By a \textit{nested family of balls in
$K$ }we mean a family $\mathcal{B}$ of balls in $K$ which is linearly
ordered under reverse inclusion. We say that $K$ is \textit{spherically
complete} (or $\emph{maximally complete}$) if the intersection of every nested family of closed balls in $K$ admits a point over $K$; this is equivalent to the requirement that $K$ admits
no proper immediate extensions. If $K$ is algebraically closed, then
by a theorem of Kaplansky \cite{kaplansky1942maximal} there is up
to isomorphism of valued fields over $K$ a unique immediate valued
field extension $\widetilde{K}$ of $K$ which is spherically complete.

We say that $K$ is \textit{complete} if every nested family of closed
balls in $K$ whose valuative radii are unbounded has a nonempty intersection;
if $K$ is spherically complete, then it is complete, but not conversely.
If $K$ is a valued field then there is a complete valued field extension
$\widehat{K}$ of $K$ characterized by the following universal mapping
property. Let $\widetilde{K}$ be a complete valued field extension
of $K$ and assume that $\Gamma$ is cofinal in $\widetilde{\Gamma}$;
then the embedding of $K$ in $\widetilde{K}$ lifts uniquely to an
embedding of $\widehat{K}$ in $\widetilde{K}$. The extension $\nicefrac{\widehat{K}}{K}$
is immediate. If $K$ is Henselian (perfect) then $\widehat{K}$ is
Henselian (perfect). Moreover, if $K$ is Henselian then $\widehat{K}$
is a primary\footnote{This means that $K$ is relatively separably algebraically closed
in $\widehat{K}$} extension of $K$ which and the canonical restriction map $G_{\widehat{K}}\to G_{K}$
of absolute Galois groups is an isomorphism. In other words, the assignment
$L\mapsto L\otimes_{K}\widehat{K}=\widehat{L}$ is a bijection between
finite separable extensions of $K$ and finite separable extensions
of $\widehat{K}$. We will refer to this result as \emph{Krasner's Lemma}.

\subsubsection{~}

Let $\nicefrac{\widetilde{K}}{K}$ be a valued field extension with
$\tdeg_{K}\widetilde{K}$ finite. We say that $\nicefrac{\widetilde{K}}{K}$
is an \textit{Abhyankar extension}, or \textit{without transcendence
defect}, if the equality $\text{tdeg}_{K}\widetilde{K}=\text{tdeg}_{k}\widetilde{k}+\dim_{\mathbf{Q}}\left(\nicefrac{\widetilde{\Gamma}}{\Gamma}\otimes\mathbf{Q}\right)$
holds; the right hand side is, in general, smaller.

Let $K$ be a Henselian valued field and $\widetilde{K}$ a finite
valued field extension of $K$. Then we have the equality $\left[\widetilde{K}:K\right]=\left[\widetilde{\Gamma}:\Gamma\right]\cdot\left[\widetilde{k}:k\right]\cdot p^{d}$
for some natural number $d$. We say that $\widetilde{K}$ is a \textit{defectless
extension} of $K$ if $d=0$ and $K$ is said to be \textit{defectless}
if every finite extension of $K$ is a defectless extension; an arbitrary
valued field is defectless if its Henselization is. Thus if $p=1$
then $K$ is automatically defectless.

\subsubsection{\label{Spherically Complete Domination}} Let $A$ be an algebraically closed valued field which is spherically complete. We will make use of the following spherically complete domination statement; see \cite{haskell2005stable}, Proposition 12.11. Let $B$ and $C$ be valued fields over $A$, jointly embedded over $A$ in an ambient extension. Let us assume that $\Gamma_B = \Gamma_A$. If the fields $k_B$ and $k_C$ are linearly disjoint over $k_A$, then the valued fields $B$ and $C$ are themselves linearly disjoint over $A$. Moreover, this requirement determines the valued field structure on the amalgamation. More precisely, if $b$ is a tuple enumerating $B$, then working in $\ACVF$, the type $\tp \left(\nicefrac{b}{C}\right)$ is completely axiomatized by $\tp\left(\nicefrac{b}{A}\right)$, together with a the set of formulas asserting that every tuple of elements of $k_B$ linearly independent over $k_A$ remains linearly independent over $k_C$.

\newpage{}

\section{\label{sec:Artin-Schreier-Ideals}Artin-Schreier Ideals}

Let $K$ be a Henselian valued field of equal characteristic. In characteristic
zero, Galois extensions of $K$ can be understood concretely in terms
of the induced extension of the residue field and the valuation group.
Indeed, the celebrated theorem of Ax-Kochen-Ershov asserts that not
only the Galois extension of $K$ but in fact every elementary statement
about $K$ whatsoever can be reduced to an elementary statement regarding
its valuation group and its residue field. 

In finite characteristic, the situation is considerably more complicated,
due to the existence of wildly ramified extensions. The purpose of
this short section is to attach certain valuation theoretic data to
such extensions, and show that it obeys certain tameness properties
in the Abhyankar settings. 

The proof makes use of a the generalized stability theorem due to
Kuhlmann \cite{kuhlmann2010elimination} and Temkin \cite{temkin2017tame},
and similar constructions appear in \cite{kuhlmann2010classification}.
These theorems only make sense in the settings of finite generation;
our basic observation is that while the notion of defect is ill behaved
over perfect fields, the knowledge that a perfect field began its
life as the perfect hull of a defectless field gives nontrivial information,
which will be sufficient for our applications.

\subsection{~\label{subsec:artinschrierclassification}}

We will make use of the following standard fact. Let $K$ be a field
of characteristic $p$ and let $\wp X=X^{p}-X$ denote the Artin-Schreier
operator. If $x\in K$ is an element which lacks an Artin-Schreier
root in $K$, then a splitting field of the polynomial $X^{p}-X-x$
provides us with a cyclic Galois extension of $K$ of degree $p$,
and conversely all cyclic Galois extensions of $K$ of degree $p$
are obtained in this way. The projective space $\mathbf{P}\left(\text{coker}\wp\right)$ over $\mathbf{F}_p$,
whose points are the one dimensional subspaces of $\text{coker}\wp$,
therefore classify Galois extensions of $K$ of degree $p$: if $x$
and $y$ are nonzero elements of $K$ lacking an Artin-Schreier root
in $K$, then the splitting fields of the polynomials $X^{p}-X-x$
and $X^{p}-X-y$ are isomorphic over $K$ precisely in the event that
$x$ and $y$ generate the same one dimensional subspace of the coset
space $\text{coker}\wp$.

\subsection{\label{subsec:-2}}

Suppose next that our field $K$ comes equipped with a Henselian valuation.
By \ref{subsec:artinschrierclassification}, the Galois extensions
of $K$ of degree $p$ are classified by the points of the projective
space $\mathbf{P}\left(\text{coker}\wp\right)$. To each point of
this projective space, we will now attach a prime ideal of $\Gamma$
in an invariant manner.

Let therefore an Artin-Schreier extension $M$ of $K$ be given, say
generated by the adjunction of an Artin-Schreier root $b$ of an element
$a\in K$, which does not already have one in $K$. Introduce the
following subset of $p^{-\infty}\cdot\Gamma$:
\[
\delta_{K}^{\ast}\left(a\right)=\left\{ v\left(b-c\right):c\in K\right\} \subset p^{-\infty}\cdot\Gamma
\]

and let $\delta_{K}\left(a\right)$ be the downwards closure of $\delta_{K}^{\star}\left(a\right)$
in $p^{-\infty}\cdot\Gamma$.

Implicit here is the assumption that the quantity $\delta_{K}\left(a\right)$
depends only on $a$, and not on the specific choice of an Artin-Schreier
root of $a$ in the field $M$. However, there is no abuse of notation;
the field $K$ is Henselian, so the Galois action preserves the valuative
distance from an element of $K$.

The following lemma describes the basic properties of the association
$a\mapsto\delta_{K}\left(a\right)$. Note that $\left(1\right)$ can
be restated as saying that it depends, in fact, only on $M$, and
not on its specific presentation as a splitting field of the polynomial
$X^{p}-X-a$.
\begin{lem}
\label{artin-schrier-cuts}Let $K$ be a Henselian valued field.

(1) The association $a\mapsto\delta_{K}\left(a\right)$ descends to
$\mathbf{P}\left(\text{coker}\wp\right)$. In other words, the elements
$a,b$ of $K$ determine the same Artin-Schreier cut provided that
we can write $a=\alpha\cdot b+c$ where $\alpha\in\mathbf{F}_{p}^{\times}$
and $c\in\wp K$.

(2) The cut $\delta_{K}\left(a\right)$ is nonpositive: it cannot
contain any strictly positive elements. If moreover $K$ is strictly
Henselian then $\delta_{K}\left(a\right)$ does not contain $0$.

(3) Let $K^{p^{-\infty}}$ be the perfect hull of $K$. Then the equality
$\delta_{K^{p^{-\infty}}}\left(a\right)=p^{-\infty}\cdot\delta_{K}\left(a\right)$
holds. In particular, if $K$ is perfect, then the cuts $\delta_{K}\left(a\right)$
are left invariant under multiplication by $p$.
\end{lem}

\begin{proof}
(1) Suppose that $a$ and $a'$ differ from each other by an element
of $\wp K$, and let Artin-Schreier roots $b$ and $b'$ of $a$ and
$a'$ be chosen. Then the elements $b$ and $b'$ differ from each
other by an element of $K$, and translation by an element of $K$
evidently does not change cuts. Multiplication by a natural number
which is invertible in $K$ will not change cuts, either, giving the
second point.

(2) The polynomial $X^{p}-X$ splits over the residue field, so Hensel
lifting implies that the polynomial $X^{p}-X-t$ splits in $K$ if
the element $t$ is of strictly positive valuation. Thus a pair of
elements at strictly positive valuative distance from each other simultaneously
lie in $\wp K$. If $K$ is strictly Henselian then $X^{p}-X-t$ splits
in $K$ even in the event that $vt=0$; so $\delta_{K}\left(a\right)$
cannot contain $0$.

(3) First note that the map $\delta_{K}$ is compatible with directed unions as $K$ varies. More precisely, let $\left(F_{i}\right)$ be a directed system of valued fields
with union $F$ and fix an element $a\in F$. Without loss, after
refining the system, the element $a$ lies in all of the $F_{i}$;
one checks that $\delta_{F}\left(a\right)$ then coincides with the
direct limit of the $\delta_{F_{i}}\left(a\right)$. Since $K^{p^{-\infty}}$
is the direct limit of the diagram $K\xrightarrow{\phi}K\xrightarrow{\phi}\ldots$,
this implies that $\delta_{K^{p^{-\infty}}}\left(a\right)$ is the
direct limit of the diagram $\delta_{K}\left(a\right)\xrightarrow{\phi}\delta_{K}\left(a^{p}\right)\xrightarrow{\phi}\ldots$.
Now by (1) the elements $a$ and $a^{p}$ determine the same Artin-Schreier
cut, i.e we have that $\phi\circ\delta_{K}=\delta_{K}$; since the
Frobenius endomorphism induces multiplication by $p$ at the level
of valuation groups, the claim follows.
\end{proof}

\subsection{~}

Let $K$ be a perfect Henselian valued field, and let $\widetilde{K}$
be a Galois extension of degree $p$, corresponding to a point
$\alpha\in\mathbf{P}\left(\text{coker}\wp\right)$. 
The discussion canonically associates to  $\alpha$ 
a cut $X$ of $\Gamma$, consisting entirely of nonpositive elements,
and invariant under multiplication by $p$. Then the set $-X$, along
with the point at $\infty$, is a prime ideal of $\Gamma_{\infty}$.
\begin{defn}
Let $K$ be a valued field which is perfect and Henselian.

(1) The \textit{Artin-Schreier ideal map}
\[
I:\mathbf{P}\left(\text{coker}\wp\right)\to\text{Spec}\Gamma_{\infty}
\]
is the map described in the discussion above.

(2) The prime ideals of $\Gamma$ which are in the image of map $I$ will be referred to as the \emph{Artin-Schreier ideals} of $K$.
\end{defn}

We then have the following:
\begin{cor}
\label{cor:Let--be}Fix a valued field $K$ which is perfect and Henselian. Let us assume that $K$ is the perfect hull of a Henselian
defectless valued field. Then every Artin-Schreier ideal of $K$ is
the radical of a principal ideal, i.e it is form $p^{-\infty}\cdot\left[\gamma,\infty\right]$
where $0\leq\gamma\in\Gamma_{\infty}$. If moreover $K$ is strictly
Henselian then $\gamma\neq 0$.
\end{cor}

\begin{proof}
Let $E$ be the defectless, Henselian valued field implicit in the
assumption, and let $a\in E$. Then the cut $\delta_{E}\left(a\right)$
admits a maximal element outright, since $E$ is defectless and Henselian,
and hence has no immediate algebraic extensions. By Lemma \ref{artin-schrier-cuts},
the cut $\delta_{K}\left(a\right)$ is the perfect hull of $\delta_{E}\left(a\right)$.
The Galois extensions of $K$ are the perfect hulls of the Galois
extension of $E$, so every Artin-Schreier ideal of $K$ is the perfect
hull of a principal ideal. Since $K$ is strictly Henselian, every
Artin-Schreier extension is totally ramified, so we conclude using Lemma
\ref{artin-schrier-cuts} 
\end{proof}
\begin{cor}
\label{artin-schrier-ideal-map} Let $F$ be an algebraically closed
valued field. Let $K_{0}$ be a finitely generated Abhyankar extension
of $F$, and $K$ a perfect strictly Henselian hull of $K_{0}$. Then
every Artin-Schreier ideal of $K$ is the radical of a principal ideal, i.e, it is of the form $p^{-\infty}\cdot\left[\gamma,\infty\right]$
for some $0<\gamma\in\Gamma$.
\end{cor}

\begin{proof}
By a theorem of Kuhlmann \cite{kuhlmann2010elimination} the class
of defectless valued fields is closed under finitely generated Abhyankar
extensions and tamely ramified algebraic extensions, so the result
follows from Corollary \ref{cor:Let--be}.
\end{proof}
\newpage{}

\section{Ordered Transformal Modules}

We develop the basic principles of ordered abelian groups with a rapidly
increasing automorphism, motivated by the case of the Frobenius action
on the valuation group of a valued field in the regime where the prime
power is large. The generalization from the case of ordinary ordered
abelian groups to those more generalized settings is straightforward,
but is included here for the sake of completeness.

\subsection{$\omega\text{OGA}$}

Recall that the ordered abelian group $\Gamma$ is said to be \textit{perfect}
if $p\Gamma=\Gamma$. By a \textit{homomorphism} of ordered abelian
groups we mean a homomorphism of the underlying abelian groups which
carries nonnegative elements to nonnegative elements; an \textit{automorphism}
of an ordered abelian group is a bijective homomorphism from the
group to itself.
\begin{defn}
(1) By a \textit{transformal ordered abelian group} we mean an ordered abelian
group $\Gamma$ equipped with an injective homomorphism $\sigma\colon\Gamma\to\Gamma$
of of ordered abelian groups; a \textit{homomorphism} of ordered transformal
abelian groups is a homomorphism of ordered abelian groups compatible
with $\sigma$ in the evident manner.

(2) Let $\Gamma$ be a transformal ordered abelian group. We say that
$\Gamma$ is \textit{$\omega$-increasing} if one has $n\alpha<\sigma\alpha$
for all $0<\alpha\in\Gamma$ and $0<n\in\mathbf{N}$.

(3) The theory $\omega\text{OGA}$ is the theory of perfect ordered
abelian groups equipped with an $\omega$-increasing automorphism.
The language is one sorted, with unary function symbols $\sigma^{\pm1}$
and $p^{\pm1}$, a binary relation symbol $\leq$ and a binary function symbol $+$ with
the obvious interpretations.

(4) By a \textit{homomorphism} of models of $\omega\text{OGA}$ we
mean a homomorphism of abelian groups which commutes with $\sigma$
and which carries nonnegative elements to nonnegative elements.
\end{defn}

\begin{rem}
Let $\mathbf{Z}\left[\sigma\right]$ be the ring of polynomials over
$\mathbf{Z}$ in the indeterminate $\sigma$, ordered so that $\sigma$
is positive and strictly exceeds every natural number, and $R=\mathbf{Z}\left[\sigma^{\pm1},p^{\pm1}\right]$
the localization of $\mathbf{Z}\left[\sigma\right]$ at the elements
$\sigma$ and $p$; the ordering of $\mathbf{Z}\left[\sigma\right]$
extends uniquely to $R$. If $\Gamma$ is a perfect ordered abelian group,
then the data of an $\omega$-increasing automorphism of $\Gamma$
is equivalent to the data of an action of $R$ on $\Gamma$ with the
property that multiplication by a positive element of $R$ preserves
the positive cone of $\Gamma$, so models of $\omega\text{OGA}$ can
be regarded as ordered modules over $R$. We will refer to them simply
as \textit{ordered modules}, and omit the reference to $R$.
\end{rem}

\begin{defn}
\begin{enumerate}
\item Let $\Gamma$ be a model of $\omega\OGA$. We say that $\Gamma$
is \textit{transformally divisible} if the canonical map $\Gamma\to\Gamma\otimes\mathbf{Q}\left(\sigma\right)$
is an isomorphism, i.e, we have $\nu\Gamma=\Gamma$ for all nonzero
$\nu\in\mathbf{Z}\left[\sigma\right]$.
\item The theory $\widetilde{\omega\OGA}$ is the theory of nonzero
models of $\omega\OGA$ which are transformally divisible.
\end{enumerate}
\end{defn}

\begin{prop}
\label{woga-frobenius}(1) Every model of $\omega\OGA$ embeds in
a model of $\widetilde{\omega\OGA}$; the theory $\widetilde{\omega\OGA}$
is model complete. It is complete, $o$-minimal and admits elimination
of quantifiers.

(2) The theory $\widetilde{\omega\OGA}$ is the theory of an ultraproduct
of structures of the form $\left(\Gamma_{i},<,\alpha\mapsto n_{i}\cdot\alpha\right)$
where the $n_{i}\in\mathbf{N}$ are natural numbers with $n_{i}\to\infty$
and $\Gamma_{i}=\Gamma_{i}\otimes\mathbf{Q}$.
\end{prop}

\begin{proof}
(1) See \cite{chernikov2014valued}, Proposition 5.10.

(2) Clear.
\end{proof}
\begin{lem}
\label{dense-divisible}Let $\Gamma$ be a model of $\omega\OGA$
which is algebraically divisible, thus $\Gamma = \Gamma \otimes \mathbf{Q}$. Then $\Gamma$ lies dense in $\Gamma\otimes\mathbf{Q}\left(\sigma\right)$
for the order topology.
\end{lem}

\begin{proof}
It is enough to prove that $\mathbf{Q}\left[\sigma^{\pm1}\right]$
lies dense in $\mathbf{Q}\left(\sigma\right)$. For, if $\gamma \in \Gamma$ and $\nu_n \in \mathbf{Q}\left[\sigma^{\pm1}\right] \to \nu \in \mathbf{Q}\left(\sigma\right)$ then $\nu_n \cdot \gamma \to \nu \cdot \gamma$. So if $\nu\in\mathbf{Q}\left[\sigma\right]$
and $\alpha\in\mathbf{Q}\left[\sigma^{\pm1}\right]$ we must find,
for every $0<n\in\mathbf{N}$, some $\beta_{n}\in\mathbf{Q}\left[\sigma^{\pm1}\right]$
with $\left|\nu\beta_{n}-\alpha\right|<\sigma^{-n}$. Certainly we
can find such $\beta\in\mathbf{Q}\left(\left(\sigma^{-1}\right)\right)$
as the latter is a field containing $\mathbf{Q}\left[\sigma^{\pm1}\right]$;
so take $\beta_{n}$to be a suitable truncation of $\beta$.
\end{proof}

\subsection{~}

Let $\Gamma$ be a model of $\omega\OGA$. The action of $\sigma$
on $\Gamma$ extends to $\Gamma_{\infty}$ so as to fix $\infty$.
We then say that an ideal $I\subset\Gamma_{\infty}$ is \textit{transformally
prime} if $0\notin I$ and whenever the element $\sigma\alpha$ lies
in $I$ then already $\alpha$ does. The $\omega$-increasing nature
of $\sigma$ then implies that a transformally prime ideal is algebraically
prime. 

We write $\Spec^{\sigma}\Gamma\subset\Spec\Gamma$ for the space of
transformally prime ideals of $\Gamma$; it is covariantly functorial
in a homomorphism of ordered transformal modules, and it is in canonical
bijection with the set of $\sigma$-invariant convex submodules of
$\Gamma$.
\begin{defn}
Let $\Gamma$ be a model of $\omega\OGA$.
\begin{enumerate}
    \item The \textit{transformal height} of $\Gamma$ is the number of
distinct nonzero $\sigma$-invariant convex subgroups of $\Gamma$, if finite, and $\infty$ otherwise.
\item We say that $\Gamma$ is \textit{transformally Archimedean} if
it of transformal height at most one.
\end{enumerate}
\end{defn}

\begin{rem}
\begin{enumerate}
    \item Let $\Gamma$ be a model of $\omega\OGA$. Then $\Gamma$ is transformally
Archimedean if and only if, for every $0<\alpha,\beta\in\Gamma$ we
have $\beta<\sigma^{n}\alpha$ and $\alpha<\sigma^{m}\beta$ for some
$n,m\in\mathbf{N}$.
\item Let $\Gamma$ be an ordered abelian group. Then $\Gamma$ is Archimedean if and only if it embeds in the ordered abelian group $\mathbf{R}$ of real numbers. Thus an Archimedean ordered abelian group is bounded in cardinality. In the transformal settings, this need not be the case. For example, fix a saturated perfect ordered abelian group $V$, and let $\Gamma = \bigoplus_{n\in\mathbf{Z}} V_{n}$ where $V_{n} = V$ and $\sigma$ acts by shifting. 
\end{enumerate}
\end{rem}

\begin{lem}
Let $\Gamma$ be a model of $\omega\OGA$.

(1) The transformal height is additive in short exact sequences of
models of $\omega\OGA$.

(2) The transformal height is preserved under the passage to the transformal
divisible hull $\Gamma\otimes\mathbf{Q}\left(\sigma\right)$ of $\Gamma$.

(3) The transformal height is bounded from above by the dimension
of $\Gamma\otimes\mathbf{Q}\left(\sigma\right)$ as a vector space
over $\mathbf{Q}\left(\sigma\right)$.
\end{lem}

\begin{proof}
Clear.
\end{proof}
\begin{lem}
\label{t-Arch-cuts-rational}Let $\Gamma$ be a model of $\widetilde{\omega\OGA}$
which is transformally Archimedean. Let $X$ be a cut in $\Gamma$.
Let us assume that one has the equality:
\[
X+\alpha=\nu X+\beta
\]
for some $\alpha,\beta\in\Gamma$ and $\nu\in\mathbf{N}\left[\sigma\right]$,
where $\nu$ is such that $\nu\gg1$. Then either $X=\Gamma$ or else
$X$ is a rational cut, i.e we have $X=\left(-\infty,\gamma\right)$
or else $X=\left(-\infty,\gamma\right]$ for some $\gamma\in\Gamma$.
\end{lem}

\begin{proof}
Let $A\subset\Gamma$ be the stabilizer of $X$, i.e, the set of elements
$a\in\Gamma$ with $a+X=X$; then $A$ is a convex subgroup of $\Gamma$.
If a cut is translate by an element of $\Gamma$ then the stabilizer
does not change, so $X$ and $\nu X$ are stabilized by the same subgroup
of $\Gamma$; this means that $a\in A\Rightarrow\nu\cdot a\in A$.
Now assume that there is some $0<a\in A$; then $\left(\nu^{n}\cdot a\right)_{n=1}^{\infty}\subset A$.
Since $\Gamma$ is transformally Archimedean and $\nu\gg1$ the sequence
$\left(\nu^{n}\cdot a\right)_{n=1}^{\infty}$ is unbounded; so $A=\Gamma$
and hence $X=\Gamma$. It follows that $A=0$ and $X$ is the cut
associated to an element $\gamma$ in the Cauchy completion of $\Gamma$.
But $\Gamma=\Gamma\otimes\mathbf{Q}\left(\sigma\right)$ so $\gamma$
must in fact lie in $\Gamma$; it is the unique solution $x$ of the
equation $x+\alpha=\nu\cdot x+\beta$.
\end{proof}

\subsection{\label{UltrapowerOGA}Transformally Archimedean Ultrapowers}

Let $\mathcal{F}$ be a nonprincipal ultrafilter on $\mathbf{N}$
and let $\Gamma$ be a transformally Archimedean model of $\omega\OGA$.
Let $\Gamma^{\star}$ denote the ultrapower of $\Gamma$ with respect
to the ultrafilter $\mathcal{F}$. It is, of course, never transformally
Archimedean, but there is a way out. To the embedding $\Gamma\hookrightarrow\Gamma^{\star}$
one can naturally attach a pair of convex $\sigma$-invariant submodules
of $\Gamma^{\star}$:
\begin{itemize}
\item The \textit{convex hull} $\text{conv}\Gamma$ $\Gamma$ in $\Gamma^{\star}$,
namely the intersection of all convex $\sigma$-invariant submodules
of $\Gamma^{\star}$ containing $\Gamma$
\item The module $\text{inf}\Gamma$ of $\Gamma$-\textit{infinitesimals} in $\Gamma^{\star}$, namely the union of all convex $\sigma$-invariant
submodules of $\Gamma^{\star}$ sharing no element in common with
$\Gamma$ other than the zero element
\end{itemize}
Then evidently one has the strict inclusion $\text{inf}\Gamma\subset\text{conv}\Gamma$,
and so we can form the quotient module:
\begin{defn}
Let $\Gamma$ be a transformally Archimedean model of $\omega\OGA$
and let $\mathcal{F}$ be a nonprincipal ultrafilter on $\mathbf{N}$.
The \textit{transformally Archimedean ultrapower} of $\Gamma$ with
respect to $\mathcal{F}$ is the quotient module $\nicefrac{\text{conv}\Gamma}{\text{inf}\Gamma}$ described above; it is denoted $\Gamma^{\mathcal{F}}$.
\end{defn}

\begin{rem}
In Section \ref{sec:Transformally-Archimedean-Ultrap}, we will lift the construction of the transformally Archimedean ultrapower to the valued field settings.
\end{rem}

The terminology is justified by the following:
\begin{lem}
\label{transformarch}Let $\Gamma$ be a transformally Archimedean
model of $\omega\OGA$, and let $\Gamma^{\mathcal{F}}$ be its transformally
Archimedean ultrapower.

(1) Then $\Gamma^{\mathcal{F}}$ is transformally Archimedean

(2) Let us assume that $\Gamma$ is algebraically divisible; then
$\Gamma^{\mathcal{F}}$ is transformally divisible.
\end{lem}

\begin{proof}
(1) Let $0<\alpha\in\Gamma$ be nonzero. Then the forward orbit $\left(\sigma^{n}\alpha\right)_{n=1}^{\infty}$
is cofinal in $\text{conv}\Gamma$, whence in the homomorphic image
$\Gamma^{\mathcal{F}}$. On the other hand, all nonnegative elements
below the backwards orbit $\left(\sigma^{-n}\alpha\right)_{n=1}^{\infty}$
of $\alpha$ lie in $\text{inf}\Gamma$ and are therefore killed upon
the passage to the transformally Archimedean ultrapower. It follows
that $\Gamma^{\mathcal{F}}$ is transformally Archimedean.

(2) The class of algebraically divisible ordered abelian groups is
closed under ultrapowers, the passage to a convex subgroup and the
passage to a quotient; so $\Gamma^{\mathcal{F}}$ is algebraically
divisible. Now fix $0<\alpha\in\text{conv}\Gamma$ and $0<\nu\in\mathbf{N}\left[\sigma\right]$.
Since $\text{conv}\Gamma$ is algebraically divisible it follows from
Lemma \ref{dense-divisible} and $\omega$-saturation of $\Gamma^{\star}$
that we can find some $\beta\in\Gamma^{\star}$ with $\nu\beta-\alpha$
lying in $\text{inf}\Gamma$. The element $\beta$ lies, in fact,
in $\text{conv}\Gamma$; so the equality $\nu\beta=\alpha$ holds
in $\Gamma^{\mathcal{F}}$ \textit{on the nose}, and we conclude.
\end{proof}

\subsection{~\label{transformallyaffine}}
\begin{defn}\label{def:transformallyaff}
Let $\Gamma$ be a model of $\widetilde{\omega\text{OGA}}$ and let
$\Psi:\Gamma\to\Gamma$ be a function. We say that $\Psi$ is \textit{transformally
affine} if it takes the form $x\mapsto\nu x+x_{0}$ for some $\nu\in\mathbf{Q}\left(\sigma\right)$
and $x_{0}\in\Gamma$.
\end{defn}

It is easy to see that $\nu$ and $x_{0}$ are then uniquely determined;
we refer to $\nu$ as the \textit{slope} of $\Psi$. 

More generally, we say that $\Psi$ is \textit{piecewise transformally
affine} if we can write $\Gamma$ as the union of finitely many closed
intervals or rays $\left(I_{n}\right)_{n=1}^{N}$ , overlapping only
at their endpoints, such that $\Psi_{|I_{n}}$ is the restriction
of a globally defined transformally affine function to $I_{n}$. If
the function $\Psi$ is continuous with respect to the order topology
and $N$ is chosen minimal so as to witness this fact, then this sequence
of intervals is uniquely determined once ordered according to their
left endpoint; the points of overlap of a pair of these intervals
are then called the \textit{singular} points of $\Psi$. If $\lambda$ is a nonsingular point of $\Psi$ then we write $\Psi'\lambda$ for the slope of $\Psi$ at $\lambda$.

\newpage{}

\section{Difference Algebra\label{sec:Difference-Algebra}}

We review rapidly the basic concepts of difference algebra. Of central
importance is Fact \ref{figalois}; a great deal of work is devoted
to rule out the existence of such extensions in various settings.

\subsection{~}

By a \textit{difference ring} we mean a (commutative and unital) ring
$R$ in which a distinguished endomorphism $\sigma\in\text{End}R$
of rings has been singled out; by a \textit{homomorphism} of difference
rings we mean a homomorphism of rings compatible with the action of
$\sigma$ in the evident manner.

We say that $R$ is \textit{inversive} if $\sigma$ is an automorphism
of $R$. If $R$ is a difference ring then there is a homomorphism
\textit{$i:R\to R^{\sigma^{-\infty}}$ }to an inversive difference
ring such that every homomorphism from $R$ to an inversive difference
ring factors uniquely through $i$; the difference ring $R^{\sigma^{-\infty}}$
is then said to be an \textit{inversive hull} of $R$, and it is unique
up to a unique isomorphism of difference rings over $R$. It is given
explicitly by the direct limit of the diagram $R\xrightarrow{\sigma}R\xrightarrow{\sigma}\ldots$
of difference rings.

We say that $R$ is \textit{transformally integral} or that it is
a \textit{difference domain} if the underlying ring of $R$ is an
integral domain and $\sigma$ is an injective endomorphism of rings.
If the underlying ring of $R$ is a field, then $R$ is said to be
a \textit{difference field}; it is then, automatically, a difference
domain.

Let $R$ be a difference ring and let $\mathfrak{p}\in\text{Spec}R$
be a prime ideal. We say that $\mathfrak{p}$ is \textit{transformally
prime} if the equality $\sigma^{-1}\left(\mathfrak{p}\right)=\mathfrak{p}$
holds. In this situation, the localization $R_{\mathfrak{p}}$ of
$R$ away from $\mathfrak{p}$ and the quotient ring $\nicefrac{R}{\mathfrak{p}}$
can both be canonically expanded to difference algebras over $R$.
We write $\text{Spec}^{\sigma}R$ for the space of transformally prime
ideals of $R$. It can be viewed as a subspace of $\text{Spec}R$
when the latter is equipped with the Zariski topology, and its formation
is covariantly functorial in a homomorphism of difference rings.

\begin{rem}
\label{rem-no-transformally-prime-ideals}
Let $R$ be a nonzero ring. Then $R$ admits a prime ideal. In difference algebra this need not be the case; the space $\Spec^{\sigma}R$ of transformally prime ideals of $R$ may be empty even if $R$ is nonzero. 

For example, let $R = \mathbf{Q} \times \mathbf{Q}$ where $\sigma$ acts by swapping the factors: $\left(x,y\right)^{\sigma} = \left(y,x\right)$. From a model theoretic point of view, this fact is closely related to the failure of quantifier elimination described in Fact \ref{figalois}.
\end{rem}

\subsection{~}

Let $K$ be a difference field. It is called a \textit{Frobenius}
\textit{difference field }if it of finite characteristic $p>0$ and
$\sigma$ is a power of the Frobenius endomorphism $\phi x=x^{p}$.
We say that $K$ is a \textit{nonstandard Frobenius difference field}
if is elementary equivalent to an ultraproduct of Frobenius difference
fields but is not itself a Frobenius difference field. The Frobenius
difference fields obviously hold a special place, but these are the
nonstandard Frobenius difference fields which are of the utmost importance,
by virtue of the following result:
\begin{fact}
\label{frobacfa}Let $\FA$ be the theory of perfect inversive difference
fields in a language expanding the language of fields by a unary function
symbol for the action of $\sigma$.

(1) The theory $\FA$ admits a model companion, denoted by $\ACFA$

(2) Let $K$ be a difference field. Then $K$ is existentially closed
as a difference field if and only if it is elementarily equivalent
to an ultraproduct of algebraically closed Frobenius difference fields,
but is not itself a Frobenius difference field. In other words the theory $\ACFA$ is precisely the theory of nonstandard algebraically closed Frobenius difference fields.
\end{fact}

\begin{proof}
See \cite{chatzidakis1999model} for (1); the truth of $\left(2\right)$
was established in \cite{hrushovski2004elementary}.
\end{proof}
\begin{rem}
We will not use the second part of \ref{frobacfa} for the development
of the abstract theory.
\end{rem}

\subsection{~}

Let $K$ be an inversive difference field and let $K^{\text{sep}}$
be an abstract separable closure of $K$ in the category of fields.
Then distinct extensions of $\sigma$ to $K^{\text{sep}}$ are not
in general isomorphic over $K$; consider, for instance, the extreme
scenario where $\sigma$ acts trivially, in which case isomorphism
classes correspond bijectively to conjugacy classes of the absolute
Galois group of $K$.

The obstruction to the uniqueness is accounted for by finite $\sigma$-invariant
separable extensions. More precisely, let $K$ be an inversive difference
field and let $L$ be an (abstract) normal algebraic extension. We
say that $L$ is $\sigma$-\textit{invariant} if it can be expanded
to a difference field over $K$. Equivalently, whenever an irreducible
polynomial $f$ with coefficients in $K$ splits in $L$ then the
polynomial $f^{\sigma}$ obtained by applying $\sigma$ to the coefficients
splits in $L$ as well. If $L$ is finite over $K$ as an abstract
field then the inclusions $K\subset L^{\sigma}\subset L$ hold and
a dimension argument show that such an extension must render $L$
an inversive difference field, too. If $L$ is $\sigma$-invariant
and normal then whenever an abstract isomorphic copy of $L$ can be
embedded over $K$ in some difference field extension, then this copy
must, in fact, be a difference field extension of $K$. 

One then has the following result:
\begin{fact}
\label{figalois}Let $K$ be an inversive difference field and fix
an abstract algebraic closure $\widetilde{K}$ of $K$. Then the following
conditions are equivalent:

(1) The difference field $K$ has no nontrivial finite Galois extensions
left invariant under $\sigma$.

(2) Fix $0<n\in\mathbf{N}$; then all lifts of $\sigma^{n}$ to $\widetilde{K}$
are isomorphic over $K$.
\end{fact}

\begin{proof}
See \cite{chatzidakis1999model}, Lemma 2.9 (the result goes back to Babbit \cite{babbitt1962finitely})
\end{proof}

\subsection{~}

Let $R$ be a difference ring and let $S=R\left[x\right]_{\sigma}$
be the ring of polynomials over $R$ in the indeterminates $\left(x^{\sigma^{n}}\right)_{n=1}^{\infty}$.
The ring $S$ can be expanded to a difference ring over $R$ in the
obvious manner suggested by the names of the variables; it is called
the ring of \textit{difference polynomials} over $R$, and those difference
polynomials lying in the abstract subring of $S$ generated over $R$
by the element $x$ alone are called \textit{algebraic}. If $R$ enjoys
the ascending chain condition on transformally prime ideals, for instance
if $R$ is a difference field, then $S$ enjoys it too. One then also
has the clear generalization of this to any number of variables.

Let $\mathbf{N}\left[\sigma\right]$ be the commutative monoid of
formal finitely supported sums $\sum m_{i}\cdot\sigma^{i}$ where
the $m_{i}$ are natural numbers. It comes equipped with the lexicographic
ordering, which is well founded, indeed isomorphic to $\omega^{\omega}$.
The ring of difference polynomials over $R$ is then the ring of finitely
supported $R$-valued functions on $\mathbf{N}\left[\sigma\right]$
where addition is computed pointwise and multiplication is given by
the convolution product.

One has a \textit{degree} map $\text{deg}:S\to\mathbf{N}\left[\sigma\right]\cup\left\{ -\infty\right\} $.
It obeys similar formal properties to its algebraic counterpart, e.g,
we have $\text{deg}\left(fx\cdot gx\right)=\deg fx+\deg gx$ and $\text{deg}\left(fx+gx\right)\leq\max\left\{ \deg fx,\deg gx\right\} $
whenever $fx,gx\in S$. We will sometimes argue by transfinite induction
on the degree of a difference polynomial. Note, however, that even
if $R$ is a difference field, then difference ideals of $S$ are
rarely principal and there is no transformal analogue of the Euclidean
division algorithm.

\subsection{~\label{subsec:transformder}}

Let $R$ be a difference ring and let $R\left[x\right]_{\sigma}$
be the ring of difference polynomials over $R$. If $f\left(x\right)\in R\left[x\right]_{\sigma}$
is a difference polynomial and $\varepsilon$ is a formal variable,
then we may write
\begin{equation}
f\left(x+\varepsilon\right)=f_{0}\left(x\right)+f_{1}\left(x\right)\cdot\varepsilon+\ldots+f_{\sigma}\left(x\right)\cdot\varepsilon^{\sigma}+\ldots\label{eq:taylor}
\end{equation}
for certain uniquely defined difference polynomials $f_{\nu}\left(x\right)\in R\left[x\right]_{\sigma}$
where $\nu\in\mathbf{N}\left[\sigma\right]$, all but finitely
of which vanish. We refer to \ref{eq:taylor} as the \textit{Taylor
expansion }of $fx$. In the presence of a valuation we imagine $\varepsilon$
to be small in magnitude, but the formalism is difference theoretic.

We refer to the difference polynomial $f_{\nu}\left(x\right)$ as
the $\nu$-th\textit{ transformal derivative} of $fx$. The operators
$f\mapsto f_{\nu}$ are then linear in $R$, descend to the prime
field, and commute with translation of the argument by an element
of $R$. They do not quite commute with a rescaling of the argument
by a nonzero element, but the relationship is easy to describe, namely
if $t\in R$ is nonzero then the $\nu$-th formal transformal derivative
of the difference polynomial $g\left(x\right)=f\left(tx\right)$ is
equal to $t^{\nu}\cdot f_{\nu}\left(x\right)$.

The operator $f\mapsto f_{0}$ is the identity map. The operator $f\mapsto f_{1}$
agrees with the partial derivative of $fx$ with respect to the variable
$x$, when viewed as a multivariable algebraic polynomial. The operator
$f\mapsto f_{1}$ will be referred to as the \textit{derivative},
and we will write $f'=f_{1}$. For algebraic polynomials, there is
no conflict with the usual interpretation of the derivative.

\subsection{~}

Let $\nicefrac{L}{K}$ be an extension of perfect inversive difference
fields. If $a\in L$ then we write $K\left(a\right)_{\sigma}$ for
the smallest perfect inversive difference subfield of $L$ containing
$a$. We say that the element $a$ is \textit{transformally algebraic} over $K$ if it is a root of a nonzero difference polynomial over $K$. The difference field $L$ is \textit{transformally algebraic }over $K$
if every $a \in L$ is transformally algebraic over $K$. For example, the inversive hull of $K$ is transformally algebraic over $K$. Furthermore, the difference field $L$ is transformally algebraic over $K$ if and only if it is the directed union of difference field extensions of $K$ of finite algebraic transcendence degree over $K$.

The element $a\in L$ is said to be \textit{transformally transcendental} over
$K$ if the sequence $\left(a^{\sigma^{n}}\right)_{n=0}^{\infty}$
is algebraically independent over $K$; equivalently, it is transformally transcendental over $K$ if it is not transformally algebraic over $K$. More generally, the sequence $(a_i)_{i \in I}$ is \textit{transformally independent} over $K$ if the sequence $\left({a_i}^{\sigma^n}\right)_{i \in I, n \in \mathbf{N}}$ is algebraically independent over $K$.

\begin{fact}
\label{fact-ttranscendental-nofing}
Let $K$ be an inversive algebraically closed difference field and let $a$ be an element transformally transcendental over $K$ in an ambient extension; then the difference field $K\left(a\right)_{\sigma}$ has no nontrivial finite $\sigma$-invariant Galois extensions.
\end{fact}

\begin{proof}
See \cite{chatzidakis1999model}
\end{proof}

\subsection{~}

Let $K=\left(K,\sigma\right)$ be an inversive perfect difference
field. If $m\in\mathbf{Z}$ then $\tau=\sigma\circ\phi^{m}=\phi^{m}\circ\sigma$
is again an automorphism of $K$; the difference field $\left(K,\tau\right)$
is said to be a \textit{Frobenius twist} of $K$. If $L$ is a Galois
extension of $K$ then $L$ is $\sigma$-invariant if and only if
it is $\tau$-invariant. It will often be the case that a statement
holds after replacing $K$ with a Frobenius twist. For example, we
have the following:
\begin{lem}
\label{twist}Let $\nicefrac{L}{K}$ be an extension of inversive
perfect difference fields and assume that $L$ is transformally algebraic
over $K$. Fix an element $\alpha\in L$. Then after twisting one
can find a difference polynomial $g\left(x\right)\in K\left[x\right]_{\sigma}$
with $g\left(\alpha\right)=0$ and $g'\left(\alpha\right)\neq0$.
\end{lem}

\begin{proof}
The element $\alpha$ is a root of some nonzero difference polynomial
$g\left(x\right)$ with coefficients in $K$, say $\sum_{\nu\in I}c_{\nu}\alpha^{\nu}=0$
where $I\subset\mathbf{N}\left[\sigma\right]$ is a finite nonempty
subset and $c_{\nu}\in K$. We may assume that the degree of $g\left(x\right)$
is chosen as small as possible with respect to this property, uniformly
across all twists of $K$. If each $\nu\in I$ has constant term zero
then $\sum c_{\nu}^{\frac{1}{\sigma}}\alpha^{\frac{\nu}{\sigma}}=0$;
so inductively we can assume that some $\nu\in I$ has a nonzero constant
term. Similarly, if the constant term of each $\nu\in I$ is divisible
by $p$ then $\sum_{\nu\in I}c_{\nu}^{\frac{1}{p}}\alpha^{\frac{\nu}{p}}=0$;
so inductively after twisting we may assume that some $\nu\in I$
has a constant term not divisible by $p$. Then $g'\neq0$; by minimality
of the degree of $g$ we have $g'\left(\alpha\right)\neq0$ and we
conclude.
\end{proof}
\begin{rem}
In Lemma \ref{twist}, a Frobenius twist is necessary; consider a generic solution of the equation $x^{\sigma} - x^p = 1$ over $K$. If we replace $\sigma$ by $\tau = \sigma \cdot p^{-1}$, then with the change of variables $x \to x^{\frac{1}{p}}$, the equation transforms to $x^\tau - x = 1$.
\end{rem}
\newpage{}

\section{Transformal Valued Fields}

\subsection{$\VFA$}

We can now introduce the central object of interest in this work.
\begin{defn}
By a \textit{transformal valued field} we mean a valued field $K$
equipped with an endomorphism $\sigma:K\to K$ of valued fields, thus
$\sigma^{-1}\left(\mathcal{\mathcal{O}}\right)=\mathcal{O}$ and $\sigma^{-1}\left(\mathcal{M}\right)=\mathcal{M}$.
The transformal valued field $K$ is said to be $\omega$-\textit{increasing}
if the induced action of $\sigma$ on the valuation group is $\omega$-increasing;
equivalently, if $x\in\mathcal{M}$ is nonzero then $\frac{x^{\sigma}}{x^{n}}\in\mathcal{M}$,
for all $n\in\mathbf{N}$. By an \textit{embedding} of transformal
valued fields we mean an embedding of the underlying fields which
is simultaneously an embedding of the underlying difference fields
and of the underlying valued fields.
\end{defn}

\begin{rem}
\label{rem-equichar}
In this paper we adopt the convention that all valued fields are of
equal characteristic; but this is automatic for $\omega$-increasing
transformal valued fields. For suppose that $K$ is of characteristic
zero and let $a\in\mathbf{Q}\subset K$ be nonzero. Then $a^{\sigma}=a$,
hence $\sigma\cdot va=va$ which means that $va=0$.
\end{rem}

\begin{example}
(1) Let $K$ be a valued field of finite characteristic $p>0$. If
$\sigma=\phi^{n}$ is a power of the Frobenius endomorphism $\phi x=x^{p}$
of $K$ then $\left(K,\sigma\right)$ is a transformal valued field.
It is not $\omega$-increasing; but one can obtain an $\omega$-increasing
transformal valued field by passing to an ultraproduct of transformal
valued fields of this form. We will say that $K$ is a \textit{Frobenius
transformal valued field}.

(2) Let $K$ be a perfect and $\omega$-increasing transformal valued
field of finite characteristic $p>0$ and fix $m\in\mathbf{Z}$. Then
$\tau=\sigma\circ\phi^{m}$ is an $\omega$-increasing endomorphism
of $K$; we say that $\left(K,\tau\right)$ is a \textit{Frobenius
twist} of $\left(K,\sigma\right)$.

(3) Let $k$ be an abstract difference field and let $\Gamma$ be
an ordered transformal module. Then the valued field $k\left(\left(t^{\Gamma}\right)\right)$
of formal power series with well ordered support, exponents in $\Gamma$
and coefficients in $k$ can be canonically expanded to a transformal
valued field.
\end{example}

\begin{lem}
\label{perfect-inversive-hull-transformal}Let $K$ be a transformal
valued field. Let $\widetilde{K}$ be the perfect, inversive hull
of $K$, which we regard as an abstract difference field extension
of $K$. Then $\widetilde{K}$ can be uniquely expanded to a transformal
valued field extension of $K$. If $K$ is $\omega$-increasing then
$\widetilde{K}$ is again $\omega$-increasing; we have $\widetilde{\Gamma}=\Gamma\otimes\mathbf{Z}\left[\sigma^{\pm1},p^{\pm1}\right]$
and $\widetilde{k}=k^{p^{-\infty},\sigma^{-\infty}}$
\end{lem}

\begin{proof}
The perfect hull of $K$ is the direct limit of the diagram $K\xrightarrow{\phi}K\xrightarrow{\phi}\ldots$
of fields, where $\phi x=x^{p}$ is the Frobenius endomorphism. Note
that $\phi\circ\sigma=\sigma\circ\phi$ and $\phi^{-1}\left(\mathcal{O}\right)=\mathcal{O}$,
hence $\phi\colon K\hookrightarrow K$ is an embedding of transformal
valued fields. So $K^{p^{-\infty}}$ can be canonically expanded to
a transformal valued field as it is the direct limit of a diagram
of transformal valued fields where the transition maps are given by
embeddings of transformal valued fields. Similarly for the inversive
hull, namely $K^{\sigma^{-\infty}}$ is the direct limit of the diagram
$K\xrightarrow{\sigma}K\xrightarrow{\sigma}\ldots$ and $\sigma\colon K\hookrightarrow K$
is an embedding of transformal valued fields. 

Equivalently, in a two sorted language, let $\widetilde{\Gamma}=\Gamma\otimes\mathbf{Z}\left[\sigma^{\pm1},p^{\pm1}\right]$
be the perfect, inversive hull of $\Gamma$. If $a\in\widetilde{K}$
then $a^{\sigma^{n}p^{m}}\in K$ for $0\ll n, m\in\mathbf{N}$, so we
can define a valuation $\widetilde{v}\colon\widetilde{K}\to\widetilde{\Gamma_{\infty}}$
by the rule $\widetilde{v}\left(a\right)=\sigma^{-n}\cdot p^{-m}\cdot v\left(a^{{\sigma}^n\cdot{p^m}}\right)$.
\end{proof}
\begin{defn}
\label{def-language-vfa}
Let $\mathcal{L}_{\VFA}$ be the expansion of $\mathcal{L}_{\VF}$
by unary function symbols $\sigma^{\pm1}$ and $\phi^{\pm1}$. We
let $\VFA$ be the theory of perfect, inversive, $
\omega$-increasing transformal valued
fields in the language $\mathcal{L}_{\VFA}$ (with the obvious interpretations).
\end{defn}

\begin{rem}
If $K$ is an $\omega$-increasing transformal valued field, then
there is a model $\widetilde{K}$ of $\VFA$ over $K$ which
embeds uniquely over $K$ in any other, using Lemma \ref{perfect-inversive-hull-transformal}.
We nevertheless choose to work with transformal valued fields with
are perfect and inversive. Namely we demand that $K$ be perfect in
order to consider the Frobenius twists of $K$ and we demand that
$K$ be inversive for the notion of a finite $\sigma$-invariant Galois
extension to be meaningful.
\end{rem}

The functorial constructions available in the realm of difference
fields and in valued fields can be lifted to the realm of
transformal valued fields:
\begin{prop}
Let $K$ be a model of $\VFA$.

(1) The algebraic Henselization $K^{\text{h}}$ of $K$ can be canonically
expanded to a model of $\VFA$ over $K$.

(2) The completion $\widehat{K}$ of $K$ can be canonically expanded
to a model of $\VFA$ over $K$.
\end{prop}

\begin{proof}
Note that the Henselization and the completion of a perfect valued
field are again perfect.

(1) Let $i:K\hookrightarrow K^{\text{h}}$ be an algebraic Henselization
of $K$. Then $i\circ\sigma:K\hookrightarrow K^{\text{h}}$ is an
embedding of valued fields with a Henselian target; by the universal
property of $K^{\text{h}}$ the map $i\circ\sigma$ factors uniquely
through $i$, i.e, there is a unique automorphism $\sigma^{\text{h}}:K^{\text{h}}\to K^{\text{h}}$
of valued fields such that $\sigma^{\text{h}}\circ i=i\circ\sigma$.
Since $K^{h}$ is an immediate extension of $K$, the automorphism
$\tau$ of $K^{h}$ thus induced is $\omega$-increasing.

(2) This follows from the universal property of the completion as
in (1).
\end{proof}

The following construction is, of course, not canonical; the question
of uniqueness up to isomorphism will be dealt with in the next subsection.
\begin{prop}\label{prop:lift-to-alg}
Let $K$ be a model of $\VFA$ and let $K^{a}$ be an abstract
algebraic closure of $K$ in the category of valued fields. Then $K^{a}$
can be expanded to a model of $\VFA$ over $K$.
\end{prop}

\begin{proof}
Let $i:K\hookrightarrow K^{a}$ be an algebraic closure of $K$ in
the category of valued fields. Then $i\circ\sigma:K\hookrightarrow K^{a}$
is an embedding of valued fields of $K$ in an algebraically closed
valued field. Since $K$ admits a unique algebraic closure in the
category of valued fields, up to isomorphism, the map $i\circ\sigma$
can be factored through $i$; that is, there is an embedding $\tau:K^{a}\hookrightarrow K^{a}$
of valued fields such that $\tau\circ i=i\circ\sigma$. Then $\left(K^{a},\tau\right)$
is a transformal valued field extension of $K$ whose underlying field
is an algebraic closure of $K$, and to finish we must only check
that it is $\omega$-increasing. This follows from the fact that the
unique extension of $\sigma$ from $\Gamma$ to $\Gamma\otimes\mathbf{Q}\left[\sigma^{\pm1}\right]$
is $\omega$-increasing.

\end{proof}

\begin{example}
\label{example-no-qe-charp-wildly-ramified}
Let $K$ be an algebraically closed and nontrivially valued model of $\VFA$, of positive characteristic $p>0$. Let $A\left(x\right)=x^{p} - x$ denote the Artin-Schrier operator
and let $B\left(x\right)=x^{\sigma}-x$ denote its transformal counterpart. The kernel of $A$ is the prime field, and the kernel of $B$ is the fixed field; since the former is a subfield of the latter, the correspondence $C\left(x\right) = B \circ A^{-1}$ is a well defined additive homomorphism of $K$. Thus for every element $x \in K$, there is a distinguished choice of an Artin-Schrier root of $B\left(x\right) = x^{
\sigma} - x$, namely the element $C\left(x\right)$. The operators of this form will be called the \textit{twisted Artin-Schrier operators} of $K$.

Now fix elements $x,y \in K$ where $x$ is of strictly negative valuation, and suppose that one has the equality $B\left(x\right) = A\left(y\right)$; one verifies that these two facts determine that the quantifier free type\footnote{over the prime field, in the language $\omega{\VFA}$ defined in \ref{def-language-vfa}} $q$ of the pair $\left(x, y\right)$. Furthermore, the pair $\left(x, y + 1\right)$ obeys the same quantifier free formulas. On the other hand, the type $q$ is not complete: it does not determine whether or not $y$ is the \textit{distinguished} Artin-Schrier root of $x^{\sigma}-x$, i.e, whether or not one has $C\left(x\right) = y$.

Let us put $F=\mathbf{F}_{p}\left(x,y\right)_{\sigma}$ and $E=F\left(z\right)_{\sigma}=F\left(z\right)$, where $z$ is an Artin-Schrier root of $x$. Then $E$ is a finite $\sigma$-invariant Galois extension of $E$ which is wildly ramified. It follows from Fact \ref{figalois} that for some $0 < n \in \mathbf{N}$, there are two distinct non-isomorphic extensions of $\sigma^n$ from $F$ to $E$. In fact, here we could take $n = 1$, and the number of distinct non-isomorphic lifts is non-canonically identified with $\mathbf{F}_p$. Namely there is a distinguished lift of $\sigma$ from $F$ to $E$ with the property that all Galois theoretic automorphisms preserve $\sigma$, and $p-1$ other lifts with no nontrivial automorphisms at all.

Let $\widetilde{F}$ be a tame closure of $F$ in the category of models of $\VFA$; since $E$ is purely wildly ramified over $F$, the fields $E$ and $\widetilde{F}$ are linearly disjoint over $F$. Thus $\widetilde{E} = E \otimes_{F} \widetilde{F}$ provides us with an example of a finite $
\sigma$-invariant Galois extension of a model of $\VFA$ which is algebraically tamely closed.
\end{example}

\begin{example}
\label{example-char0-notuniquealg}
We give an example of a model $K$ of $\VFA$ of characteristic zero which is algebraically strictly Henselian and yet admits a finite $\sigma$-invariant Galois extension which is totally ramified. For the construction we use the fact that ramified extensions of Henselian valued fields of characteristic zero are well controlled by the value group; see \ref{subsec-ramificationtheory}.
Let $\Gamma$ be a model of $\omega\OGA$ which admits a nontrivial finite $\sigma$-invariant extension; that is, there is a model $\widetilde{\Gamma}$ of $\omega\OGA$ over $\Gamma$ with $1 < \left[\widetilde{\Gamma} \colon \Gamma \right] < \infty$. For example, take an ultraproduct of copies of $\Gamma = \mathbf{Z}\left[p^{-1}\right]$ equipped with multiplication by $p^n$, for some fixed prime $p$. Let $k$ be an inversive algebraically closed difference field of characteristic zero, and $K = k\left(\left(t^{\Gamma}\right)\right)$ the field of Hahn series over $k$ with coefficients in $\Gamma$. Then $K$ is spherically complete and algebraically strictly Henselian. Nevertheless, the field $K$ admits a finite $\sigma$-invariant Galois extension which is tamely but totally ramified, namely $L = k \left(\left(t^{\widetilde{\Gamma}}\right)\right)$.
\end{example}

\begin{cor}
\label{cor-alg-closure-not-unique}
Let $K$ be a model of $\VFA$ of characteristic zero which is algebraically strictly Henselian. Then a field theoretic algebraic closure of $K$ in the category of models of $\VFA$ over $K$ need not be unique over, even up to isomorphism.  In finite characteristic, uniqueness may fail even if $K$ is algebraically tamely closed.
\end{cor}

We will see later that the theory $\VFA$ admits a model companion $\widetilde{\VFA}$; in the latter, the residue field and the valuation group are stably embedded models of $\ACFA$ and $\widetilde{\omega\OGA}$, with no additional structure. Using Corollary \ref{cor-alg-closure-not-unique}, one concludes:

\begin{cor}
\label{cor-no-qe-relative-tores}
Let $\mathcal{L}$ be the expansion of the three sorted language of perfect valued fields by unary function symbols for the action of $\sigma$ on all three sorts. Then the theory $\widetilde{\VFA}$ fails to eliminiate quantifiers in the language $\mathcal{L}$, even if quantifiers over the residue field\footnote{The valuation group is o-minimal and eliminates quantifiers, so quantifiers over the valuation group sort are redundant} are allowed.
\end{cor}

\begin{rem}
\label{rem-ake-splitting-qe}
Let $\mathcal{L}_{\AKE}$ be the expansion of the three sorted language of valued fields of equal characteristic zero by unary function symbols for the action of $\sigma$ on all three sorts and a cross section, namely a section of the valuation map $K^{\times} \to \Gamma$ in the category of modules over $\mathbf{Z}\left[\sigma\right]$. Every model $K$ of $\widetilde{\VFA}$ admits a cross section; when $K$ is saturated, the choice of this map is unique up to isomorphism, hence up to elementary equivalence. 

It follows from the main result of \cite{azgin2010valued} that in characteristic zero, the theory $\widetilde{\VFA}_{\AKE}$ of models of $\widetilde{\VFA}$ equipped with a cross section eliminates quantifiers relative to the residue field. However, an arbitrary model of $\VFA$ need not admit a cross section, and no cross section can possibly be definable in $\widetilde{\VFA}$, hence there is no contradiction with Corollary \ref{cor-no-qe-relative-tores}.

\end{rem}

\begin{defn}
Let $\nicefrac{\widetilde{K}}{K}$ be an extension of models of $\VFA$ of finite transformal transcendence degree. We say that $\widetilde{K}$ is \textit{transformally Abhyankar} over $K$ if one has the equality $\tdeg_{\sigma}\nicefrac{\widetilde{K}}{K} = \tdeg_{\sigma}\nicefrac{\widetilde{k}}{k} + {\dimm}_{\mathbf{Q}\left(\sigma\right)} \nicefrac{\widetilde{\Gamma} \otimes {\mathbf{Q}\left(\sigma\right)}}{\Gamma \otimes {\mathbf{Q}\left(\sigma\right)}}$
\end{defn}

Here the notation $\text{tdeg}_{\sigma}$ is the transformal transcendence degree of an extension of difference fields; see Section \ref{sec:Difference-Algebra}.

\begin{lem}
\label{lem-transformally-Abhyankar}
Let $\nicefrac{\widetilde{K}}{K}$ be an extension of models of $\VFA$ of finite transformal transcendence degree. Then one has the inequality $\tdeg_{\sigma}\nicefrac{\widetilde{K}}{K} \geq {\tdeg}_{\sigma}\nicefrac{\widetilde{k}}{k} + {\dim}_{\mathbf{Q}\left(\sigma\right)} \nicefrac{\widetilde{\Gamma} \otimes {\mathbf{Q}\left(\sigma\right)}}{\Gamma \otimes {\mathbf{Q}\left(\sigma\right)}}$
\end{lem}

\begin{proof}
This is the usual proof of Abhyankar's inequality in the valued settings. If $\left(\alpha_{i}\right)$ is a transformal transcendence basis for $\widetilde{k}$ over $k$ and $\left(\gamma_{j}\right)$ is a basis for $\widetilde{\Gamma} \otimes \mathbf{Q}\left(\sigma\right)$ over $\Gamma \otimes \mathbf{Q}\left(\sigma\right)$, let $a_i$ and $b_j$ be elements with residue classes $\alpha_i$ and valuation $\gamma_{j}$, respectively; then using the $\omega$-increasing nature of $\sigma$ on the valuation group, one checks that the sequence $a_i, b_j$ is transformally independent over $K$.
\end{proof}

The following Corollary of Lemma \ref{lem-transformally-Abhyankar} will be used repeatedly, often without mention. We will revisit it in Section \ref{sec:The-Transformal-Herbrand}:

\begin{cor}
\label{cor-transformally-alg-no-valgp-adjunction}
Let $\nicefrac{\widetilde{K}}{K}$ be a transformally algebraic extension of models of $\VFA$. Then we have the isomorphism $\widetilde{\Gamma} \otimes \mathbf{Q}\left(\sigma\right) = {\Gamma} \otimes \mathbf{Q}\left(\sigma\right)$. In particular if $\Gamma \otimes \mathbf{Q}\left(\sigma\right) = \Gamma$, then a transformally algebraic extension of $K$ does not enlarge the valuation group.
\end{cor}

\begin{example}
\label{example-transformally-Abhyankar}
Let $K$ be an algebraically valued field which is spherically complete. Then every valued field extension of $K$ of transcendence degree one is Abhyankar over $K$; in fact, the latter property characterizes the spherically complete valued fields inside the class of algebraically closed valued fields. In the transformal regime, this is no longer true. 

Namely let $K = k \left(\left(t^{\mathbf{Q}\left(\sigma\right)}\right)\right)$ with $k$ algebraically closed and of characteristic zero, say equipped with the identity automorphism. Consider the extension $L = K\left(x\right)_{\sigma}$ with $x = a_0 + a_1t + a_2t^2 + \ldots$ where the $a_i$ are algebraically independent over $k$ and obey $a_i^{\sigma} = a_i$ for each $i$. The residue class of $x$ is $a_0$; replacing $x$ with $y = x^{\sigma} - x$ kills the constant term, and $y \cdot t^{-1}$ has residue class $a_1$. Continuing in this manner inductively, one shows that the residue field of $L$ coincides with $k\left(a_0, a_1, \ldots \right)$; so $L$ is not transformally Abhyankar over $K$. One can generalize this exmaple to show that every countably generated transformally algebraic extension of $k$ arises as the residue field of a monogenic extension of $K$.
\end{example}
\begin{prop}
\label{spherically-complete-extension}Let $K$ be a model of $\VFA$.
Then there is a model $M$ of $\VFA$ over $K$ whose underlying
valued field is spherically complete.
\end{prop}

\begin{proof}
We may assume that $K$ is algebraically closed. Let $M$ be an abstract
spherically complete immediate extension of $K$ in the category of
valued fields. By the theorem of Kaplansky \cite{kaplansky1942maximal}
on the uniqueness of maximal immediate extensions, it is uniquely
determined up to an isomorphism of valued fields over $K$ by this
property. Every automorphism of $K$ in the category of valued fields
must therefore lift to an automorphism of $M$; such a lift renders
$M$ an $\omega$-increasing transformal valued field since there
is evidently no valuation group adjunction.
\end{proof}

Before we turn to the next example we wish to make the following elementary
remark. Let $K$ be a model of $\VFA$ and let $L$
be an immediate extension of $K$. Let $\mathcal{C}$ be the class
of all maximally complete immediate extensions of $L$. Let $\equiv_{L}$
and $\equiv_{K}$ denote the equivalence relations on $\mathcal{C}$
given by isomorphism over $L$ and over $K$, respectively. The existence of a continuum of pairwise distinct equivalence classes
of $\equiv_{L}$ need not imply that $\equiv_{K}$ has as many classes.
Suppose, however, that $\text{Aut}\left(\nicefrac{L}{K}\right)$ is
countably infinite and that $L$ is setwise fixed by $\text{Aut}\left(\nicefrac{M}{K}\right)$,
whenever $M$ is a maximally complete immediate extension of $L$.
Then each equivalence class of $\equiv_{K}$ is the union of \textit{countably}
many equivalence classes of $\equiv_{L}$ and the conclusion is valid.
\begin{example}
\label{exa:non-unique-spherically}Let $k$ be a model of $\ACFA$
and $M=k\left(\left(t^{\mathbf{Q}\left(\sigma\right)}\right)\right)$
the displayed valued difference field of formal Hahn series. The field
$\mathbf{Q}\left(\sigma\right)$ admits a unique discrete valuation
with $v\sigma=1$ and trivial on $\mathbf{Q}$; write $R$ for the
valuation ring. Let $K$ be the valued difference subfield of $M$
consisting of all those formal power series supported on a bounded
subset of $\mathbf{Q}\left(\sigma\right)$ with respect to this valuation.
Then $K$ is the increasing union of difference subfields which are
algebraically closed and spherically complete, namely
\[
k\left(\left(t^{R}\right)\right)\subset k\left(\left(t^{\sigma^{-1}\cdot R}\right)\right)\subset\ldots\subset K
\]
Let us put $x=t^{-1}$ and write:
\[
\alpha=x+x^{\frac{1}{\sigma}}+\ldots
\]
Then $\alpha^{\sigma}-\alpha=t^{-1}$. Let $L$ be the valued difference
field generated over $K$ by the element $\alpha$. Then the automorphism
group $\text{Aut}\left(\nicefrac{M}{K}\right)$ leaves the field $L$
setwise fixed, since it can be characterized intrinsically as the
difference subfield generated by the solutions to the equation $X^{\sigma}-X=x$
in $M$. If $F$ is the fixed field of $K$, then the map $\text{Aut}\left(\nicefrac{L}{K}\right)\to F$
given by $\tau\mapsto\tau\alpha-\alpha$ is an injective group homomorphism,
so the group $\text{Aut}\left(\nicefrac{L}{K}\right)$ is countably
infinite, provided that $F$ is.

Fix a power $q$ of $p$ and let

\[
\beta=x+x^{\frac{1}{q}}+\ldots+x^{\frac{1}{\sigma}}+x^{\frac{1}{q\cdot\sigma}}+\ldots
\]
then $\beta$ is supported on a set of order type $\omega^{2}$ and
we have $\beta^{q}-\beta=\alpha$. Moreover, we have:
\[
\beta^{\sigma}-\beta=x^{\frac{1}{q}}+x^{\frac{1}{q^{2}}}+\ldots
\]
so that $\beta^{\sigma}-\beta$ is, in fact, an element of $K$. The
field $L$ therefore admits finite $\sigma$-invariant abelian Galois
extensions which are purely wildly ramified and of unbounded degree.
Reasoning as in Example \ref{example-no-qe-charp-wildly-ramified}, we find that $K$ has a continuum
of pairwise nonisomorphic algebraic closures in the category of models
of $\VFA$.
\end{example}

\begin{cor}
\label{immediate-non-unique}There exists an algebraically closed
model of $\VFA$ whose residue field is a model of $\ACFA$
and whose valuation group is a model of $\widetilde{\omega\OGA}$,
and which nevertheless has a continuum of pairwise nonisomorphic spherically
complete immediate extensions.
\end{cor}

\begin{proof}
See Example \ref{exa:non-unique-spherically} and discussion preceding it.
\end{proof}

\subsubsection{Composition of Valuations}
\begin{defn}\label{def:compos}
The category of \textit{pairs} is the fibered product $\VFA\times_{\FA}\VFA$
of categories, where the implicit functor $\VFA\to\text{FA}$
on the left is given by the residue functor to the category of difference
fields, and the implicit functor $\VFA\to\text{FA}$ on the
right is the forgetful functor.
\end{defn}

\textbf{Notation and Explanation}. Let $K_{21} = \left(K_2, \mathcal{O}_{21}\right)$ be a model of $\VFA$ with residue field $K_1$ and underlying difference field $K_2$. Let us assume that we are given an $\omega$-increasing transformal valuation ring $\mathcal{O}_{10}$ of $K_1$, so that $K_1$ is the underlying difference field of a second model $K_{10}$ of $\VFA$ with residue field $K_0$. In this way $K_2$ admits a canonical structure of a model of $\VFA$ as in the proof of Proposition \ref{prop:composition-of-valuations} below.

With this notation, an object of the category of pairs is a tuple $\left(K_{21}, K_{10}\right)$ where $K_{21}$ is a model of $\VFA$ and $K_{10}$ is a second model of $\VFA$ whose underlying difference field coincides with the residue field $K_1$ of $K_{21}$; an embedding $\left(K_{21},K_{10}\right)\hookrightarrow\left(\widetilde{K_{21}},\widetilde{K_{10}}\right)$
of pairs is an embedding $K_{21}\hookrightarrow\widetilde{K_{21}}$
of models of $\VFA$ with the property that the induced map
$K_{1}\hookrightarrow\widetilde{K_{1}}$ of residue fields is, in
fact, an embedding $K_{10}\hookrightarrow\widetilde{K_{10}}$ of models
of $\VFA$.

We all also refer to a triplet $\left(K_{20}, K_{21}, K_{10}\right)$ of models of $\VFA$ where the construction of $K_{20}$ is as in the proof of Proposition \ref{prop:composition-of-valuations} below (one can recover $K_{20}$ from $K_{21}$ and $K_{10}$ alone).

\begin{prop}\label{prop:composition-of-valuations}
The following three categories are equivalent:

(1) The category $\mathcal{C}_{1}$ of models $K$ of $\VFA$
together with a distinguished transformally prime ideal of $\mathcal{O}$,
where embeddings $K\hookrightarrow\widetilde{K}$ are required to
induce a pointed map $\Spec^{\sigma}\widetilde{\mathcal{O}}\to\Spec^{\sigma}\mathcal{O}$
of spectra

(2) The category $\mathcal{C}_{2}$ of models $K$ of $\VFA$
together with a distinguished transformally prime ideal of $\Gamma$,
where embeddings $K\hookrightarrow\widetilde{K}$ are required to
induce a pointed map $\Spec^{\sigma}\widetilde{\Gamma}\to\Spec^{\sigma}\Gamma$
of spectra

(3) The category $\mathcal{C}_{3}$ of pairs of models of $\VFA$ as in Definition \ref{def:compos}
\end{prop}

\begin{proof}
Let $K_{20}=\left(K_{2},\mathcal{O}_{20}\right)$ be a model of $\VFA$
and let $\mathfrak{p}\in\Spec^{\sigma}\mathcal{O}_{20}$ be a transformally
prime ideal. Let $\mathcal{O}_{21}$ be the localization of $\mathcal{O}_{20}$
away from $\mathfrak{p}$ and let $\mathcal{O}_{10}$ be obtained
by factoring $\mathfrak{p}$ out. Since localizations and quotients
commute, the residue field of the local ring $\mathcal{O}_{21}$ is
canonically isomorphic to the field of fractions of $\mathcal{O}_{10}$;
if we let $K_{21}=\left(K_{2},\mathcal{O}_{21}\right)$ then the residue
field $K_{1}$ of $K_{21}$ is canonically expanded to a model $K_{10}=\left(K_{1},\mathcal{O}_{10}\right)$
of $\VFA$. This gives a functor $\mathcal{C}_{1}\to\mathcal{C}_{3}$. 

Given a pair $\left(K_{21},K_{10}\right)$ of models of $\VFA$,
let $\mathcal{O}_{20}=\mathcal{O}_{21}\times_{K_{1}}\mathcal{O}_{10}$
be the pullback of $\mathcal{O}_{10}$ under the residue map $\mathcal{O}_{21}\to K_{1}$.
Then $\mathcal{O}_{21}$ is the localization of $\mathcal{O}_{20}$
away from the kernel of the surjective map $\mathcal{O}_{20}\to\mathcal{O}_{10}$;
this recovers $\mathfrak{p}$ and gives a functor in the other direction,
which an equivalence between $\mathcal{C}_{1}$ and $\mathcal{C}_{3}$.
The equivalence of $\mathcal{C}_{2}$ and $\mathcal{C}_{1}$ is obvious,
so we are done.
\end{proof}
Let $\left(K_{21},K_{20},K_{10}\right)$ be a triplet of models of
$\VFA$, per the notation used in the proof of the Proposition
above. We then have a canonical short exact sequence of ordered transformal
modules
\[
0\to\Gamma_{10}\to\Gamma_{20}\to\Gamma_{21}\to0
\]
and the topological space $\Spec^{\sigma}\mathcal{O}_{20}$ is obtained
by gluing the closed point of $\Spec^{\sigma}\mathcal{O}_{21}$ to
the generic point of $\Spec^{\sigma}\mathcal{O}_{10}$.

\subsubsection{~}

Let $K$ be a model of $\VFA$. The \textit{transformal height}
of $K$ is the transformal height of the valuation group, and $K$
is said to be \textit{transformally Archimedean} if its valuation
group is. If $\left(K_{21},K_{20},K_{10}\right)$ is a triplet of
models of $\VFA$ then the transformal height of $K_{20}$ is
the sum of the transformal heights of $K_{21}$ and $K_{10}$. Furthermore,
the transformal height is bounded from above by the absolute transformal
transcendence degree of a model over its prime field.
\begin{prop}
\label{prop:t-arch-devissage}Let $\mathcal{C}$ be a class of models
of $\VFA$ which is closed under isomorphisms. Let us assume:

(1) The class $\mathcal{C}$ is closed under directed unions.

(2) Let $\left(K_{21},K_{20},K_{10}\right)$ be a triplet of models
of $\VFA$. If $K_{21}$ and $K_{10}$ lie in $\mathcal{C}$,
then $K_{20}$ lies in $\mathcal{C}$.

(3) Every model of $\VFA$ which is transformally Archimedean
lies in $\mathcal{C}$.

Then $\mathcal{C}$ is the class of \textbf{all} models of $\VFA$.
\end{prop}

\begin{proof}
By (1) we reduce to finite transformal height and by (2) and induction
to the transformally Archimedean case, in which case (3) applies.
\end{proof}

\subsection{The Algebraic Ramification Tower\label{subsec:The-Algebraic-Ramification}}

\subsubsection{~}

Let $K$ be a model of $\VFA$. We have seen in Proposition \ref{prop:lift-to-alg} that $K$ admits
a separable closure in the category of models of $\VFA$ and
we now look into the question of uniqueness. In the category of difference
fields, by the theorem of Babitt \ref{figalois}, if a separable closure
is not unique, then this is always witnessed by the existence of a
$\sigma$-invariant separable extension which is finite as a field
extension. This difficulty is absorbed into $\VFA$ through
the residue field and through the existence of totally ramified extensions
to which the valuation extends uniquely. The optimistic hope that
the same criterion can be applied in the category of transformal valued
fields turns out to be false, and a slight modification is needed.

First we need to introduce some definitions. Let $\nicefrac{L}{K}$
be an algebraic extension of abstract valued fields. We say that $L$
is \textit{$\text{h}$-finite} over $K$ if $L^{\text{h}}$ is a finite
extension of $K^{\text{h}}$; the extension $\nicefrac{L}{K}$ is {\em trivial} as a  $\text{h}$-finite  extension
 if it induces an isomorphism upon the passage to Henselian hull. We
say that $L$ is a \textit{Henselian} extension of $K$ if the integral
closure of $\mathcal{O}_{K}$ in $L$ coincides with $\mathcal{O}_{L}$,
which is to say that $L$ can be uniquely expanded to a valued field
extension of $K$. If $L$ is a $\text{h}$-finite extension of $K$
then $L$ factors canonically as a Henselian extension of a trivial
$\text{h}$-finite extension, but a factorization in the opposite
direction is not canonical and may not exist. This motivates:
\begin{defn}
Let $\nicefrac{L}{K}$ be a $\text{h}$-finite extension of models
of $\VFA$. We say that $L$ is \textit{split} over $K$ if
can be factored as a trivial $\text{h}$-finite extension of a $\sigma$-invariant
Henselian extension (which must then be finite \textit{on the nose}).
\end{defn}

The next example shows that not all $\text{h}$-finite extensions
are split. By Krasner's lemma, a very slight perturbation of an irreducible
separable polynomial over a Henselian valued field is still irreducible
and the splitting fields are isomorphic. It is therefore possible
for a collection of irreducible separable polynomials to determine
linearly disjoint Galois extensions, each individually totally ramified,
and such that nevertheless all these polynomials simultaneously split
in every Henselian extension.

\begin{example}
\label{h-finite-example}Begin with the field $K=\mathbf{F}_{p}\left(x\right)_{\sigma}$
where $vx<0$. We identify the valuation group of $K$ with $\mathbf{Z}\left[\sigma\right]$
where $vx=-1$. Let $L$ be the Galois difference field extension
of $K$ obtained by splitting the polynomials $X^{p}-X-x^{\sigma^{n}}$
for all $n\in\mathbf{N}$. It is purely wildly ramified over $K$,
and the splitting fields of these polynomials are linearly disjoint
over $K$. Fix an element $t$ of the completion $\widehat{K}$ of
$K$ which is transformally transcendental over $K$ and whose valuation
is strictly positive, and let $\widetilde{K}=K\left(y\right)_{\sigma}$
where $y^{p}-y=x^{\sigma}-x+t$. Then $\widetilde{K}$ is an immediate,
purely transformally transcendental, regular extension of $K$. 

We claim that $\widetilde{L}=L\otimes_{K}\widetilde{K}$ is an $\text{h}$-finite
extension of $\widetilde{K}$. If this is indeed the case, then $\widetilde{L}$
cannot possibly split over $\widetilde{K}$, since every finite $\sigma$-invariant
extension of $\widetilde{K}$ descends to the prime field and is unramified,
whereas splitting the polynomial $X^{p}-X-x$ must enlarge the valuation
group: the ramification group of $\widetilde{L}$ over $\widetilde{K}$
is nontrivial.

It remains to be demonstrated that $\widetilde{L}$ is indeed $\text{h}$-finite
over $\widetilde{K}$. Over an algebraically Henselian valued field,
every element at strictly positive valuative distance from an Artin-Schreier
element is itself an Artin-Schreier element; since $x^{\sigma}-x+t$
is evidently an Artin-Schreier element in $\widetilde{K}$, it follows
that $x^{\sigma}-x$ must also be, and likewise inductively the elements
$x^{\sigma^{n+1}}-x^{\sigma^{n}}$ for all $n\in\mathbf{N}$. The
polynomials $X^{p}-X-x^{\sigma^{n}}$ therefore \textit{simultaneously}
split in every Henselian extension of $\widetilde{K}$, so $\widetilde{L}$
is $\text{h}$-finite over $\widetilde{K}$.
\end{example}

\begin{example}
\label{example-h-finite-2}
Let $K$ be an algebraically closed and nontrivially valued model
of $\VFA$. Let $a$ be a new element of the fixed field, algebraically
transcendental over $K$, and let $b$ be an element with $vc>\gamma$
for all $\gamma\in\Gamma$. Finally let $c=a+b$ and let $L$ be the
model of $\VFA$ generated over $K$ by the element $c$. Since $c$ is
transformally transcendental over $K$, the field $L$ has no nontrivial
finite $\sigma$-invariant Galois extensions. On the other hand, the
residue field of $l$ of $L$ has plenty of nontrivial finite
$\sigma$-invariant Galois extensions; for example, the splitting
field of $X^{m}-a$ for $m$ coprime to $p$; and these extensions all
lift to finite $\sigma$-invariant Galois extensions of the algebraic
Henselization $L^{h}$ of $L$.
\end{example}
\begin{cor}
\label{cor-nonsplit-h-finite}
Let $K$ be a model of $\VFA$. Then it may be that $K$ has
no nontrivial finite $\sigma$-invariant separable extensions, while
its algebraic Henselization does.
\end{cor}

\begin{proof}
See Example \ref{h-finite-example} and Example \ref{example-h-finite-2}.
\end{proof}

The uniqueness critierion for a lift of $\sigma$ to the algebraic closure can then be formulated as follows:
\begin{prop}
\label{criterion-unique-algclosure}Let $K$ be a model of $\VFA$
and $\widetilde{K}$ an abstract algebraic closure of $K$ in the
category of fields. Then the following are equivalent:

(1) Fix $0<n\in\mathbf{N}$; then the action of the absolute Galois
group $G_{K}$ of $K$ on the set of expansions of $\widetilde{K}$
to a model of $\VFA$ over $\left(K,\sigma^{n}\right)$ is transitive.

(2) The field $K$ has no nontrivial $\text{h}$-finite $\sigma$-invariant
separable extensions.
\end{prop}

\begin{proof}
Suppose first that $K$ is algebraically Henselian. Then the category
of models of $\VFA$ algebraic over $K$ is canonically equivalent
to the category of inversive difference fields algebraic over $K$,
since the valuation extends uniquely to every algebraic extension
and is again $\omega$-increasing. In this situation, the result follows
from \ref{figalois}. Now (2) merely states that $K^{\text{h}}$,
when equipped with its unique extension of the automorphism, has no
nontrivial finite $\sigma$-invariant separable extensions; since
$K^{\text{h}}$ is unique up to a unique isomorphism over $K$ already
as a valued field, the claim follows.
\end{proof}

Let $\widetilde{\VFA}$ be the model companion of $\VFA$. Using Corollary \ref{cor-nonsplit-h-finite}, we deduce:
\begin{cor}
\label{cor-no-qe-in-ACFA-functions}
Let $\mathcal{L}$ be the expansion of the language of $\VFA$ where basic formulas are of the form $\phi \wedge \theta$ with $\phi$ a quantifier free in the language of $\VFA$ and $\theta$ an arbitrary formula in the language of difference fields; then $\widetilde{\VFA}$ fails to eliminate quantifiers in the language $\mathcal{L}$.
\end{cor}

\begin{proof}
Assuming this quantifier elimination statement, unwinding definitions, we obtain: let $K$ be a model of $\VFA$ and $K_0$ its underlying difference field; if the category of difference fields admits the amalgamation property over $K_0$, then the category of models of $\VFA$ admits the amalgamation property over $K$. But for $K$ is as in Corollary \ref{cor-nonsplit-h-finite} (after possibly replacing $\sigma$ by $\sigma^n$), this is impossible.
\end{proof}

\subsubsection{~}
\begin{lem}\label{lem:unramified-classification}
Let $K$ be a model of $\VFA$ with residue field $k$.

(1) Let $k\hookrightarrow k^{\text{sep}}$ be a difference field extension
with underlying field a separable closure of $k$. Then there exists
an algebraically strictly Henselian model $K^{\text{sh}}$ of $\VFA$
reproducing the embedding $k\hookrightarrow k^{\text{sep}}$ residually,
and enjoying the following universal property. Let $L$ be an algebraically
Henselian model of $\VFA$ over $K$. Then every embedding of
$k^{\text{sep}}$ in $l$ over $k$ can be lifted to an embedding
of $K^{\text{sh}}$ in $L$ over $K$, and the lifting is unique.

(2) Let us assume that $k$ has no nontrivial finite $\sigma$-invariant
Galois extensions. Then $K$ admits a unique algebraic strict Henselization
in the category of models of $\widetilde{\VFA}$, up to a (generally
nonunique) isomorphism.
\end{lem}

\begin{proof}
(1) Let $k\hookrightarrow k^{\text{sep}}$ be a difference field extension
with underlying field a separable closure of $k$. Let $i:K\hookrightarrow K^{\text{sh}}$
be an abstract strict Henselization of $K$ reproducing the embedding
$k\hookrightarrow k^{\text{sep}}$ at the level of residue fields;
it is characterized by the universal property that an embedding of
$K^{\text{sh}}$ in a strictly Henselian abstract valued field extension
$L$ of $K$ is equivalent to the data of an embedding of $k^{\text{sep}}$
in $l$ over $k$. The map $i\circ\sigma:K\hookrightarrow K^{\text{sh}}$
is an embedding of $K$ in an abstract strictly Henselian valued field;
the universal property implies that the endomorphism of $k^{\text{sep}}$
over $k$ can be uniquely lifted to an endomorphism of $K^{\text{sh}}$.
The verification that $K^{\text{sh}}$ obeys the universal property
is immediate.

(2) By Fact \ref{figalois} the difference field $k$ admits a unique
separable closure in the category of difference fields up to isomorphism.
So the claim follows from (1).
\end{proof}

\subsubsection{~}
\begin{lem}\label{lem:tamely-ramified-classification}
Let $K$ be an algebraically strictly Henselian model of $\VFA$
and let $m$ be a natural number invertible in $\mathbf{F}_{p}$.

(1) The valuation map induces an isomorphism $K^{\times}\otimes\nicefrac{\mathbf{Z}}{m}=\Gamma\otimes\nicefrac{\mathbf{Z}}{m}$
of transformal modules.

(2) The finite $\sigma$-invariant abelian Galois extensions of $K$
of exponent dividing $m$ are classified by the finite transformal
submodules of $\Gamma\otimes\nicefrac{\mathbf{Z}}{m}$
\end{lem}

\begin{proof}
(1) Since $k$ is separably closed and $m$ is invertible in $k$
we have $k^{\times}=\left(k^{\times}\right)^{m}$, so $\mathcal{O}_{K}^{\times}=\left(\mathcal{O}_{K}^{\times}\right)^{m}$
by Hensel's lemma. It follows that $\mathcal{O}_{K}^{\times}\otimes\nicefrac{\mathbf{Z}}{m}=0$,
and we conclude by tensoring the short exact sequence $0\to\mathcal{O}_{K}^{\times}\to K^{\times}\to\Gamma\to0$
with $\nicefrac{\mathbf{Z}}{m}$. 

(2) Since $K$ is algebraically strictly Henselian and $m$ is invertible
in $K$, the field $K$ has all $m$-th roots of unity. Let $\iota:K^{\times}\to K^{\times}$
denote the raising to the $m$-th power map; it is an endomorphism
of abelian groups. By Kummer theory, if $\Delta\subset K^{\times}\otimes\nicefrac{\mathbf{Z}}{m}$
is a subgroup then the algebraic extension of $K$ obtained by adjoining
all $m$-th roots to the elements of $\iota^{-1}\left(\Delta\right)$
is a Galois extension of exponent dividing $m$; distinct subgroups
give rise to distinct Galois extensions; all abelian Galois extensions
of exponent dividing $m$ arise in this way; and the degree of the
extension is the cardinal of the subgroup. It is clear that under
this bijection, the $\sigma$-invariant extensions correspond to $\sigma$-invariant
subgroups, so the result follows from (1).
\end{proof}

\begin{rem}
    \label{rem:tame-classification}
    Let us assume that $K$ is algebraically Henselian (but not necessarily strictly Henselian), and that for all $m$ invertible in $\mathbf{F}_p$, the module $\Gamma \otimes \nicefrac{\mathbf{Z}}{m}$ has no nontrivial finite transformal submodules. Let $L$ be an algebraic strict Henselization of $K$ in the category of models of $\VFA$. The choice of $L$ is not in general unique, even up to isomorphism, but regardless of this choice, it has the same value group as $K$, and (by Lemma \ref{lem:tamely-ramified-classification}) has no nontrivial finite $\sigma$-invariant Galois extensions which are tamely ramified. It follows that $K$ has no nontrivial finite $\sigma$-invariant Galois extensions which are tamely but totally ramified (since such extensions are by definition linearly disjoint to $L$ over $K$).
\end{rem}

\subsubsection{~}

Let $K$ be a model of $\VFA$ which is algebraically tamely
closed and let $\wp x=x^{p}-x$ denote the algebraic Artin-Schreier
operator. By \ref{artin-schrier-ideal-map} we have a canonical map
\[
I:\mathbf{P}\left(\text{coker}\wp\right)\to\Spec\Gamma_{\infty}
\]
from the projective space of one dimensional subspaces of the coset
space $\text{coker}\wp$ to prime ideals of $\Gamma_{\infty}$. There
is an action of $\sigma$ on the right hand side, whose fixed points
are precisely the transformally prime ideals of $\mathcal{O}$. Since
$K$ is inversive, the elements $a$ and $a^{\sigma}$ both admit an Artin-Schreier root in $K$ or neither do, so the induced action of $\sigma$
on $\text{coker}\wp$ is injective and we obtain an action of $\sigma$
on the projectivization $\mathbf{P}\left(\text{coker}\wp\right)$;
since everything is sight is completely functorial, we find that $I$ restricts
to a map between the fixed points of $\sigma$ on both sides, i.e, it carries $\sigma$-invariant subspaces to transformally prime ideals.

\begin{lem}
\label{artins}Let $K$ be a model of $\VFA$ whose underlying
valued field is tamely closed.

(1) The finite $\sigma$-invariant Galois extensions of $K$
of degree $p$ are classified by the fixed points of $\sigma$ in
its action on $\mathbf{P}\left(\text{coker}\wp\right)$.

(2) Let us assume that the Artin-Schreier ideals of $K$ all fail
to be transformally prime. Then $K$ admits a unique separable closure
in the category of models of $\VFA$, up to a (generally nonunique)
isomorphism.
\end{lem}

\begin{proof}
(1) Let $G_K$ denote the absolute Galois group of $K$. By Artin-Schreier theory there is a canonical isomorphism $\text{Hom}\left(G_{K},\mathbf{F}_{p}\right)=\text{coker}\wp$
of discrete vector spaces over $\mathbf{F}_{p}$, so closed normal
subgroups of $G_{K}$ of index $p$ correspond to one dimensional
subspaces of $\text{coker}\wp$. By functoriality, under this correspondence,
the $\sigma$-invariant one dimensional subspaces correspond to $\sigma$-invariant
Galois extensions of degree $p$.

(2) The field $K$ is tamely closed and in particular Henselian, so
by Proposition \ref{criterion-unique-algclosure} it is enough to
prove that $K$ admits no nontrivial finite $\sigma$-invariant Galois
extensions. Arguing towards contradiction, fix a finite $\sigma$-invariant
Galois extension of $K$ with Galois group $G$. By the discussion of \ref{subsec-ramificationtheory},
the group $G$ is a $p$-group. The center of $G$ is nontrivial and
characteristic, hence $\sigma$-invariant; by passing to the corresponding
$\sigma$-invariant extension we may assume that $G$ is abelian.
The subgroup $pG$ of $G$ is nontrivial is again characteristic and
hence $\sigma$-invariant; so we may assume that $G$ is abelian of
exponent $p$, hence a direct sum of cyclic groups of order $p$.

Let $\mathcal{X}$ be the set of Artin-Schreier ideals of $K$. The
$\omega$-increasing assumption implies that the orbits of $\sigma$
in its action on $\mathcal{X}$ are either infinite or singletons.
Thus if no Artin-Schreier ideal is transformally prime, then no Artin-Schreier
ideal can have a finite orbit under the action of $\sigma$; by functoriality,
there is no Galois extension of $K$ of degree $p$ with finite orbit
under the action of $\sigma$, either.
\end{proof}

\begin{rem}
\label{rem:artin-classification}
    Let $K$ be an arbitrary algebraically Henselian model of $\VFA$. Then the map $I$ is still defined and (as the proof shows) in Lemma \ref{artins}, condition (1) continues to hold. Furthermore, in (2), one can at least conclude that $K$ has no nontrivial finite $\sigma$-invariant Galois extensions of degree a power of $p$. Note however that the Artin-Schreier ideal attached to an unramified Artin-Schreier extension is $\left[0, \infty\right]$, which is always transformally prime, and that a tame closure $L$ of $K$ might admit finite $\sigma$-invariant Artin-Schreier extensions which do not descend to an Artin-Schrier extension of $K$.
\end{rem}

\begin{rem}
\label{rem:degree-p-wildly-ramified}
Let $K$ be an abstract inversive difference field and assume that
all finite extensions of $K$ are of degree a power of $p$. If $x\in K$
and $n>1$ then it is possible that the splitting fields of the polynomials
$X^{p}-X-x$ and $X^{p}-X-x^{\sigma^{n}}$ over $K$ are isomorphic,
but the splitting fields of the polynomials $X^{p}-X-x$ and $X^{p}-X-x^{\sigma}$
are not. On the other hand, if $K$ is an algebraically tamely closed
model of $\VFA$ then this scenario is impossible: a finite
Galois extension left invariant under $\sigma^{n}$ is left invariant
under $\sigma$, too.
\end{rem}

\begin{rem}
In the settings of Lemma \ref{artins}, let $L$ be an Artin-Schreier extension of $K$. If $L$ is left invariant under $\sigma$, then the Artin-Schreier ideal of $L$ is transformally prime; the converse however is false in general. For example, working inside the field $k \left(\left(t^{\mathbf{Q}\left(\sigma\right)}\right)\right)$, let $a = t^{-1} + t^{-\nicefrac{1}{\sigma}} + t^{-\nicefrac{1}{\sigma^2}} + t^{-\nicefrac{1}{\sigma^4}} + \ldots $, let $K$ a tame closure of $k\left(t,a\right)_{\sigma}$, and consider the Artin-Schreier extension given by the adjunction of an Artin-Schreier root of the element $a$.
\end{rem}

\begin{cor}
\label{abhyankarcriterion}Let $K$ be an algebraically closed model
of $\VFA$. Let $\widetilde{K}$ be a model of $\VFA$
over $K$ with $\Gamma=\widetilde{\Gamma}$. Let us assume that $\widetilde{K}$
is the perfect, algebraically strictly Henselian hull of a finitely
generated abstract Abhyankar extension of $K$; then $\widetilde{K}$
has no nontrivial finite Galois extensions left invariant under $\sigma$.
\end{cor}

\begin{proof}
Since $\Gamma=\Gamma\otimes\mathbf{Q}\left[\sigma^{\pm1}\right]=\widetilde{\Gamma}$ is algebraically divisible,
the field $\widetilde{K}$ is perfect and algebraically tamely closed;
by Lemma \ref{artins}, it is then enough to prove that no Artin-Schreier
ideal of $\widetilde{K}$ is transformally prime. By \ref{artin-schrier-ideal-map},
the Artin-Schreier ideals of $\widetilde{K}$ are of the form $p^{-\infty}\cdot\left[\gamma,\infty\right]$
where $0<\gamma\in\Gamma$; by the $\omega$-increasing nature of
$\sigma$, no ideal of this form is $\sigma$-invariant.
\end{proof}
\begin{cor}
\label{uniquesep}Let $K$ be an algebraically Henselian model of
$\VFA$. Let us assume:

(1) The residue field $k$ of $K$ has no nontrivial finite $\sigma$-invariant
Galois extensions.

(2) For every natural number $m$ invertible in $\mathbf{F}_{p}$,
the transformal module $\Gamma\otimes\nicefrac{\mathbf{Z}}{m}$ has
no nontrivial finite transformal submodules.

(3) The Artin-Schreier ideals of $K$ all fail to be transformally
prime.

Then $K$ admits a unique separable closure in the category of models
of $\VFA$, up to isomorphism.
\end{cor}

\begin{proof}
We use Proposition \ref{criterion-unique-algclosure}. Let $L$ be a finite $\sigma$-invariant Galois extension of $K$; we must prove that $K = L$. There is a canonical factorization $K \subseteq L_1 \subseteq L_2 \subseteq L_3 = L$ of $\sigma$-invariant Galois extensions where $L_1$ is unramified over $K$, $L_2$ is tamely (but totally) ramified over $L_1$ and $L_3$ is purely wildly ramified over $L_2$, hence of degree a power of $p$. By Lemma \ref{lem:unramified-classification}, the finite $\sigma$-invariant unramified Galois extensions of $K$ correspond to the finite $\sigma$-invariant Galois extensions of $k$; using (1) we get $K = L_1$. Similarly (2) and Lemma \ref{lem:tamely-ramified-classification} and Remark \ref{rem:tame-classification} imply that $L_2 = K$. Finally, by (3) and Remark \ref{rem:artin-classification} we have $L_3 = L_2$, so that overall we have $K = L$.
\end{proof}

\newpage{}

\section{\label{sec:TheTransformalHens}The Transformal Henselization}

The purpose of this section is to study the transformal analogues of the notion of Henselization and strict Henselization.

\subsection{Summary}

We briefly describe some of the main results of this section; see \ref{subsec:transformder}
for unexplained notation.
\begin{defn}
\label{def:transformally-henselian}Let $K$ be a model of $\VFA$. 

We say that $K$ is\textit{ transformally Henselian} if $K$ and all
of its Frobenius twists obey the following transformal analogue of
Hensel's lemma. Let $fx\in\mathcal{O}_{K}\left[x\right]_{\sigma}$
be a difference polynomial; then every residual simple root can be
lifted to an integral root. In other words, if $a\in\mathcal{O}_{K}$
is such that $vfa>0$ and $vf'a=0$ then an element $b\in\mathcal{O}_{K}$
is found with $fb=0$ and whose residue class coincides with that
of the element $a$
\end{defn}

\begin{thm}
\label{transformalhenselization}Let $K$ be a model of $\VFA$.
Let us assume that we are given a perfect, inversive, transformally algebraic extension
$\widetilde{k}$ of $k$. Then there exists a transformally Henselian,
transformally algebraic extension $\widetilde{K}$ of $K$ reproducing
the embedding $k\hookrightarrow\widetilde{k}$ residually, and enjoying
the following properties:

(1) Let $L$ be a transformally Henselian model of $\VFA$ over
$K$. Then every embedding of $\widetilde{k}$ in $l$ over $k$ can
be lifted to an embedding of $\widetilde{K}$ in $L$ over $K$, and
the lifting is unique.

(2) The induced embedding of valuation groups is an isomorphism.

(3) The finite $\sigma$-invariant separable extensions of $\widetilde{K}$
are controlled by the finite $\sigma$-invariant separable extensions
of $\widetilde{k}$ in the following sense. Let us assume that $K$
is algebraically Henselian and has no nontrivial finite $\sigma$-invariant
Galois extensions. Then every finite $\sigma$-invariant Galois extension of $\widetilde{K}$ is unramified; hence the finite $\sigma$-invariant Galois extensions
of $\widetilde{K}$ are in bijection with the finite $\sigma$-invariant Galois
extensions of $\widetilde{k}$ via the residue map.
\end{thm}

\begin{cor}\label{cor:existence-of-absolute}
Let $K$ be a model of $\VFA$. Then there exists a transformally
Henselian, transformally algebraic, immediate extension $K^{\text{th}}$
of $K$ which embeds uniquely over $K$ in every other transformally
Henselian extension; it is called \textbf{the transformal Henselization}
of $K$ and it is unique up to a unique isomorphism of models of $\VFA$
over $K$.

The relative algebraic closure of $K$ in $K^{\text{th}}$ is an algebraic
Henselization $K^{\text{h}}$ of $K$, and every finite $\sigma$-invariant
separable extension of $K^{\text{th}}$ uniquely descends to a finite
$\sigma$-invariant separable extension of $K^{\text{h}}$.
\end{cor}

\begin{proof}
In the statement of Theorem \ref{transformalhenselization}, let $\widetilde{k}=k$.
\end{proof}

\begin{rem}
    In the situation of Theorem \ref{transformalhenselization}, we will refer to $\widetilde{K}$ as the \emph{absolute transformal Henselization} of $K$ when $k = \widetilde{k}$; otherwise we say that $\widetilde{K}$ is a \emph{relative transformal Henselization} (with respect to the embedding $\nicefrac{\widetilde{k}}{k}$).
\end{rem}
\begin{cor}
\label{transformalHenselizationalg} Let $K$ be a model of $\VFA$
which has no nontrivial $\text{h}$-finite $\sigma$-invariant separable
extensions. Then there exists a transformally algebraic, algebraically
closed, transformally Henselian extension $\widetilde{K}$ of $K$
which embeds over $K$ in every other extension with these properties.
It is unique up to isomorphism over $K$; the isomorphism however
is not in general unique, even if $K$ is separably
closed.
\end{cor}

\begin{proof}
We may assume that $K$ is algebraically Henselian. Define a sequence
$K=K_{0}\subset K_{1}\subset K_{2}\subset\ldots$ of models of $\VFA$,
transformally algebraic over $K$, as follows. Let $K_{2i+2}$ be
the transformal Henselization of $K_{2i+1}$ and let $K_{2i+1}$ be
an algebraic closure of $K_{2i}$. Then each $K_{j}$ is algebraically
Henselian and without any nontrivial finite $\sigma$-invariant Galois
extensions. Indeed for $j=0$ this is true by assumption; at odd stages
the property is not lost, as the category of finite $\sigma$-invariant
Galois extensions remains the same upon the passage to the transformal
Henselization; and at even stages this holds trivially since an algebraically
closed field is Henselian and has no nontrivial Galois extensions,
left invariant under $\sigma$, or otherwise. The union $\widetilde{K}$
of this tower is, on the one hand, the union of algebraically closed
fields, by considering even indices, and thus algebraically closed;
and on the other hand, by considering odd indices, it is the union
of transformally Henselian models, and thus transformally Henselian.
The verification that $\widetilde{K}$ is as advertised is immediate,
using the fact that each $K_{j}$ admits a unique algebraic closure
in the category of models of $\VFA$, up to isomorphism, and
the functoriality of the transformal Henselization.
\end{proof}
\begin{rem}
The notion of a transformally Henselian field introduced in \ref{def:transformally-henselian}
should not be confused with the notion of $\sigma$-\textit{Henselian}
studied in \cite{durhan2015quantifier}; see Remark \ref{rem-terminology}.
\end{rem}

\begin{rem}
The tower constructed in Corollary \ref{transformalHenselizationalg}
need not terminate at any finite stage; see Example \ref{algebraic-extension-of-t-henselian}
and Example \ref{t-henselization-of-alg-closed}.
\end{rem}

\begin{rem}
\label{rem-thenselian-not-inversive}
See \cite{dor2023contracting}, Section 4 and Section 5 for an alternative definition of transformally Henselian models using "twisted" difference polynomials such $x^{\frac{\sigma}{p}} - x$ , avoiding reference to the various Frobenius twists.
\end{rem}

\subsection{The Transformal Henselization}

\begin{lem}
(Uniqueness of Hensel lifts) Hensel lifts are unique if they exist.
More precisely, let $K$ be a model of $\VFA$, let $fx\in\mathcal{O}_{K}\left[x\right]_{\sigma}$
be a difference polynomial, and suppose that $a\in\mathcal{O}_{K}$
is such that $vfa>0$ and $vf'a=0$. Then there exists at most one
element $b\in\mathcal{O}_{K}$ such that $fb=0$ and $v\left(a-b\right)>0$.
For such $b$ it is in fact the case that $v\left(a-b\right)=vfa$.
\end{lem}

\begin{proof}
We use the Taylor expansion. Let us write:
\[
f\left(x+h\right)=fx+h\cdot f'x+h^{2}\cdot f_{2}x+\ldots+h^{\nu}\cdot f_{\nu}x
\]
Let $a\in\mathcal{O}_{K}$ be such that $fa=0$ and $vf'a=0$. If
$v\left(a-b\right)>0$ and $a\neq b$ then by plugging in $x=a$ and
$h=b-a$ we obtain:
\[
0=fb+h\cdot f'a+h^{2}\cdot f_{2}a+\ldots+h^{\nu}\cdot f_{\nu}a
\]
it follows that $vfb=vh$, whence $fb\neq0$ and $v\left(a-b\right)=vfb$.
\end{proof}
\begin{lem}
\label{lem:(Basic-properties)}(Basic properties)

(1) The ultraproduct of perfect and Henselian Frobenius transformal
valued fields is transformally Henselian.

(2) Every transformally Henselian model is algebraically Henselian.

(3) The directed union of transformally Henselian models of $\VFA$
is transformally Henselian.

(4) The intersection of a family of transformally Henselian models,
within some fixed model, is transformally Henselian.

(5) Let $K$ be a model of $\VFA$ which is transformally Henselian.
Then its twisted reducts $\left(K,\sigma^{n}\circ\phi^{m}\right)$
for $0<n\in\mathbf{N}$ and $m\in\mathbf{Z}$ are transformally Henselian.
\end{lem}

\begin{rem}\label{rem:alg-vs-trans-henselization} 
    (1) refers to perfect Henselian valued fields endowed with a power of the Frobenius endomorphism so as to become transformal valued fields. It is this statement which \emph{motivates} the definition of transformal Henselianity. See also Remark \ref{rem:transformal-analogues}.
\end{rem}

\begin{proof}
The class of models of $\VFA$ which are transformally Henselian
admits a $\forall\exists$ axiomatization, and the witness is unique
if it exists by uniqueness of Hensel lifts, so (3) and (4) follow.
The transformal Hensel lemma, when specialized to algebraic polynomials,
yields Hensel's lemma; hence (2).

Now let $K$ be an ultraproduct of perfect and Henselian Frobenius
transformal valued fields $\left(K_{\alpha},\sigma_{\alpha}\right)$
where $\sigma_{\alpha}=x \mapsto x^{q_{\alpha}}$ for prime powers $q_{\alpha}$;
then $K$ is a model of $\VFA$. Fix a difference polynomial
$fx\in\mathcal{O}_{K}\left[x\right]_{\sigma}$ and an element $a\in\mathcal{O}_{K}$.
The difference polynomials $f$ and $f'$ are both supported on a
finite subset $I\subset\mathbf{N}\left[\sigma\right]$. If $q$ is
a prime power which is sufficiently large then the substitution map
$\mathbf{N}\left[\sigma\right]\to\mathbf{N}$ given by $\sigma\mapsto q$
is injective when restricted to $I$; so as $K$ is the ultraproduct
of Henselian fields it obeys the transformal analogue of Hensel's
lemma. Thus if $K$ is of characteristic zero, then it is transformally
Henselian. In characteristic $p$, for fixed $m\in\mathbf{Z}$, the
Frobenius twist $\left(K,\sigma\circ\phi^{m}\right)$ is the ultraproduct
of the Frobenius transformal valued fields $\left(K_{\alpha},\sigma_{\alpha}\circ\phi^{m}\right)$,
so the same argument applies. This gives (1), and (5) is obvious.
\end{proof}
\begin{lem}
\label{newton}Let $K$ be a model of $\VFA$. Then the following
conditions are equivalent:

(1) (``Hensel's Lemma'') The model $K$ is transformally Henselian

(2) (``Newton's Lemma'') The model $K$ and all of its Frobenius
twists obey the following requirement. Let $fx\in\mathcal{O}_{K}\left[x\right]_{\sigma}$
be a difference polynomial, and assume that an element $a\in\mathcal{O}_{K}$
is found with $vfa>2\cdot vf'a$. Then there exists an integral root
$b\in\mathcal{O}_{K}$ of $fx$ such that $v\left(a-b\right)=vfa-vf'a$.

In case (2) the element $b$ is uniquely determined subject to the
constraint that $v\left(a-b\right)\geq vfa-vf'a$. Furthermore, in
the settings of (2) even if $K$ is merely assumed to be algebraically
Henselian, one can find at least some $b\in\mathcal{O}_{K}$ with
$vfb-vfa\geq\left(vfa-vf'a\right)^{\sigma}$ and $v\left(a-b\right)=vfa-vf'a$.
\end{lem}

\begin{proof}
(1) is a special case of (2). For the implication $\left(1\right)\Rightarrow\left(2\right)$,
use once again the Taylor expansion:
\[
f\left(x+h\right)=f\left(x\right)+f'\left(x\right)\cdot h+f_{2}\left(x\right)\cdot h^{2}+\ldots+f_{\nu}\left(x\right)\cdot h^{\nu}
\]
Let us assume that an element $a\in\mathcal{O}_{K}$ is found so that
$vfa>2\cdot vf'a$. Let $x=a$ and $h=\frac{f\left(a\right)}{f'\left(a\right)}$;
and let $y$ be a formal variable. Then we have:
\[
f\left(a+h\cdot y\right)=\left(1+y\right)\cdot f\left(a\right)+\left(hy\right)^{2}\cdot f_{2}\left(a\right)+\ldots+\left(hy\right)^{\nu}\cdot f_{\nu}\left(a\right)
\]
Since $vfa>2vf'a$ dividing this expression by $fa$ leaves all coefficients
integral. Letting
\[
g\left(y\right)=\frac{f\left(a+hy\right)}{f\left(a\right)}
\]
we have
\[
g\left(y\right)=1+y+\widetilde{g_{0}\left(y\right)}+\widetilde{g_{1}\left(y\right)}
\]
where $\widetilde{g_{0}\left(y\right)}$ is an algebraic polynomial,
all of whose coefficients lie in the maximal ideal, and $\widetilde{g_{1}\left(y\right)}$
is a difference polynomial with no algebraic part, all of whose coefficients
lie in the maximal ideal as well (More precisely, $\widetilde{g_0}$ is obtained from $g$ by restricting to monomials of the form $x^n$ for $0 < n \in \mathbf{N}$, and $\widetilde{g_1}$ is obtained from $g$ by restricting to monomials of the form $x^{\nu}$ where $1 \ll \nu \in \mathbf{N}\left[\sigma\right]$).

If $K$ is transformally Henselian
then there exists an integral root $c$ of $g\left(y\right)$ with
residue class equal to $-1$ and setting $b=a+hc$ gives the equivalence
of (1) and (2). If $K$ is merely assumed to be algebraically Henselian,
then we can still find $c$ with $vgc>\left(vh\right)^{\sigma}$ by
lifting an integral root of the algebraic polynomial of $1+y+\widetilde{g_{0}\left(y\right)}$.
In this case, setting $b=a+hc$ gives the second part of the assertion.
\end{proof}
 
We will say that a class $\mathcal{C}$ of models of $\VFA$ is {\em transitive in triplets} if
whenever $\left(K_{21},K_{20},K_{10}\right)$ is 
 a triplet of models of $\VFA$,  then $K_{20} \in \mathcal{C}$  if and only if $K_{21} \in 
 \mathcal{C}$ and $K_{10} \in 
 \mathcal{C}$ 
 (The word "transitive" is a usually property of binary relations, not unary, but  we have in mind here the
  relation between the valued field and residue field.)
 
\begin{lem}
\label{transitivetriplets}
 The class of transformally Henselian models of $\VFA$ is transitive in triplets.
\end{lem}

\begin{proof}
For the sake of the argument let us say that a field is \textit{quasi
transformally Henselian} if it obeys the transformal analogue of Hensel's
lemma; so $K$ is transformally Henselian if and only if its Frobenius
twists are all quasi transformally Henselian. We will show a slightly
stronger statement, namely that $K_{20}$ is quasi transformally Henselian
if and only if $K_{21}$ and $K_{10}$ are quasi transformally Henselian.
The tricky part is to prove that $K_{21}$ is quasi transformally
Henselian provided that $K_{20}$ is; the other implications are formal.

First assume that $K_{21}$ and $K_{10}$ are quasi transformally
Henselian. We will prove that $K_{20}$ is quasi transformally Henselian.
Let $g_{2}\left(x\right)\in\mathcal{O}_{K_{20}}\left[x\right]_{\sigma}$
be a difference polynomial, and let $g_{1}\left(x\right)\in\mathcal{O}_{K_{10}}\left[x\right]_{\sigma}$
and $g_{0}\left(x\right)\in K_{0}\left[x\right]_{\sigma}$ be its
reductions under the maps $\mathcal{O}_{K_{20}}\to\mathcal{O}_{K_{10}}$
and $\mathcal{O}_{K_{20}}\to K_{0}$, respectively. We wish to prove
that every simple root of $g_{0}$ lifts to an integral root of $g_{2}$.
Since $K_{10}$ is quasi transformally Henselian, every simple root
of $g_{0}$ lifts to an integral root of $g_{1}$. Viewing $g_{1}$
as a difference polynomial with coefficients in $K_{10}$ and $g_{2}$
as a difference polynomial with coefficients in $\mathcal{O}_{K_{21}}$,
and using the fact that $K_{21}$ is quasi transformally Henselian,
this root can be lifted further to a root of $g_{0}$ lying in $\mathcal{O}_{K_{21}}$;
but in fact this root will lie in $\mathcal{O}_{K_{20}}$, since $\mathcal{O}_{K_{20}}$
is precisely the pullback of $\mathcal{O}_{K_{10}}$ under the residue
map $\mathcal{O}_{K_{21}}\to K_{10}$.

Now assume that $K_{20}$ is quasi transformally Henselian. We will
prove that $K_{10}$ and $K_{21}$ are quasi transformally Henselian.

First we verify that $K_{10}$ is quasi transformally Henselian. Let
$g_{10}\left(x\right)\in\mathcal{O}_{K_{10}}\left[x\right]_{\sigma}$
be a difference polynomial and suppose that $a\in\mathcal{O}_{K_{10}}$
is such that $vg_{10}a>0$ and $vg_{10}^{'}a=0$. Let $b\in\mathcal{O}_{K_{20}}$
and $g_{20}\left(x\right)\in\mathcal{O}_{K_{20}}\left[x\right]$ lift
$a$ and $g_{10}$ respectively; then since $K_{20}$ is quasi transformally
Henselian we can find $c\in\mathcal{O}_{K_{20}}$ with the same residue
class as that of $b$ and $g_{20}\left(c\right)=0$. The image of
$c$ in $\mathcal{O}_{K_{10}}$ is then a root of $g_{20}$ and has
the same residue class as that of $a$; thus $K_{10}$ is quasi transformally
Henselian.

Now we prove that $K_{21}$ is quasi transformally Henselian. So assume
$a\in\mathcal{O}_{K_{21}}$ and $f\left(x\right)\in\mathcal{O}_{K_{21}}\left[x\right]_{\sigma}$
are such that $v_{21}fa>0$ and $v_{21}f'a=0$. As in the proof of
Lemma \ref{newton}, after an affine change of variables, we may assume
that $f\left(x\right)=1+x+c_{2}x^{2}+\ldots+c_{\sigma}x^{\sigma}+\ldots$
where the $c_{\nu}$ are all in the maximal ideal $\mathcal{M}_{21}$
of $\mathcal{O}_{K_{21}}$. So the residue image of $f\left(x\right)$
in $K_{1}\left[x\right]_{\sigma}$ is the polynomial $1+x$; it has
a unique root in $K_{1}$, namely $-1$. This implies that $\text{res}_{21}\left(a\right)=-1$
actually lies in $\mathcal{O}_{K_{10}}$; since $\mathcal{O}_{K_{20}}$
is the pullback of $\mathcal{O}_{K_{10}}$ under the residue map $\mathcal{O}_{K_{21}}\to K_{1}$
we learn that the element $a$, lies, in fact in the smaller ring
$\mathcal{O}_{K_{20}}$. Furthermore, using the inclusion $\mathcal{M}_{21}\subset\mathcal{M}_{20}$
we find that the difference polynomial $f\left(x\right)$ lies in
$\mathcal{O}_{K_{20}}\left[x\right]_{\sigma}$. Finally, as $v_{21}fa>0$
and $v_{21}f'a=0$ we learn that $v_{20}fa\gg v_{20}f'a$; in particular
we have $v_{20}fa>2\cdot v_{20}f'a$. By Lemma \ref{newton} there
is a unique $b\in\mathcal{O}_{K_{20}}$ with $fb=0$ and $v_{20}\left(a-b\right)=v_{20}fa-v_{20}f'a$.
Then $v_{21}\left(a-b\right)=v_{21}fa-v_{21}f'a>0$; it follows that
$K_{21}$ is quasi transformally Henselian.
\end{proof}

\subsubsection{~}
\begin{lem}
\label{complete-t-henselian}Let $K$ be a model of $\VFA$
and let $\widehat{K}$ be the completion of $K$.

(1) Let us assume that $K$ is transformally Henselian. Then $K$
is transformally algebraically closed in $\widehat{K}$.

(2) Conversely, let us assume that $K$ is transformally Archimedean,
algebraically Henselian, and transformally algebraically closed in
$\widehat{K}$; then $K$ is transformally Henselian.
\end{lem}

\begin{proof}
(1) Let $b$ be an element of the completion of $K$ which is transformally
algebraic over $K$; we must show that it lies in $K$. We may assume
that $vb=0$. Replacing $K$ by a Frobenius twist and using Lemma
\ref{twist}, we can find a difference polynomial $fx\in\mathcal{O}_{K}\left[x\right]_{\sigma}$
such that $fb=0$ and $f'b\neq0$. If $a\in\mathcal{O}_{K}$ is sufficiently
close to $b$ then $vfa>2vf'a$. Since $K$ is transformally Henselian,
there is by Newton's lemma (\ref{newton}) an integral root of $fx$
in the open ball of valuative radius $vfa-vf'a$ around $a$; this
root is unique, and is therefore equal to $b$.

(2) Let us assume first that $K=\widehat{K}$ is complete. Suppose
that $fx\in\mathcal{O}_{K}\left[x\right]_{\sigma}$ is given and that
$a\in\mathcal{O}_{K}$ is such that $vfa>0$ and $vf'a=0$. Let $a_{0}=a$.
Since $K$ is algebraically Henselian, by Lemma \ref{newton} we can
find a sequence $a_{0},a_{1},\ldots$ of elements of $\mathcal{O}_{K}$
with $v\left(a_{n+1}-a_{n}\right)=vfa_{n}$ and $vfa_{n+1}\geq\left(vfa_{n}\right)^{\sigma}$.
Since $K$ is transformally Archimedean, the sequence $\left(a_{n}\right)_{n=0}^{\infty}$
Cauchy, and since $K$ is complete, it is convergent; if $b$ is the
limit then $fb=0$ and $v\left(a-b\right)>0$. By applying the same
reasoning to the various twists, we find that $K$ is transformally
Henselian. The general case follows from the fact that a difference
field relatively transformally algebraically closed in a transformally
Henselian extension is itself transformally Henselian, as in \ref{lem:(Basic-properties)}.
\end{proof}
\begin{cor}
\label{absolute-t-hens-in-t-arch}Let $K$ be a model of $\VFA$
which is transformally Archimedean. Then an (absolute) transformal
Henselization of $K$ exists, and it is as advertised by the statement
of Theorem \ref{transformalhenselization}.
\end{cor}

\begin{proof}
We may assume that $K$ is algebraically Henselian. Let $K^{\text{th}}$
be the relative transformal algebraic closure of $K$ in its completion
$\widehat{K}$. We have already seen that $K^{\text{th}}$ is transformally
Henselian; let us check that $K^{\text{th}}$ is worthy of the notation
chosen. Let $\widetilde{K}$ be a transformally Henselian extension
of $K$. Replacing $\widetilde{K}$ by the relative transformal algebraic
closure of $K$ in $\widetilde{K}$, we may assume that $\widetilde{K}$
is transformally algebraic over $K$; this does not disturb the assumption
that $\widetilde{K}$ is transformally Henselian, and every embedding
of $K^{\text{th}}$ in $\widetilde{K}$ over $K$ must factor through
the relative transformal algebraic closure, anyway. Then $\Gamma$
is cofinal in $\widetilde{\Gamma}$, so the embedding of $K$ in $\widetilde{K}$
induces an embedding of completions; since $\widetilde{K}$ is transformally
Henselian, it is transformally algebraic closed in its completion,
and thus we obtain an embedding of $K^{\text{th}}$ in $\widetilde{K}$
over $K$. Since $K$ lies dense in $K^{\text{th}}$ with respect
to the valuation topology, the embedding is unique subject to the
constraint that it preserves the valued field structure.

Since $K$ is algebraically Henselian, by Krasner's lemma (see \ref{subsec-ramificationtheory}), the embedding
$K\hookrightarrow K^{\text{th}}$ induces an isomorphism $G_{K}=G_{K^{\text{th}}}$
of absolute Galois groups; every finite separable extension of $K^{\text{th}}$
descends to $K$, left invariant under $\sigma$, or otherwise. Moreover,
the completed algebraic Henselization is an immediate extension, so
that the transformal Henselization enjoys the additional properties
promised by the statement of the Theorem.
\end{proof}
\begin{cor}
\label{finite-t-henselian}Let $K$ be a model of $\VFA$ which
is transformally Henselian. Then every finite $\sigma$-invariant
extension of $K$ is transformally Henselian. 
\end{cor}

\begin{proof}
Let $L$ be a finite $\sigma$-invariant extension of $K$; we must prove that $L$ is transformally Henselian. By a limit argument we may assume that $K$ is of finite transformal
height; and then by Lemma \ref{transitivetriplets} we may assume
that $K$ is transformally Archimedean. Every algebraic extension
of an algebraically Henselian model of $\VFA$ is algebraically
Henselian, so as $L$ is transformally Archimedean it is enough to
show that $L$ is transformally algebraically closed in $\widehat{L}$.
But this is an immediate consequence of Krasner's lemma and the fact
that $K$ is transformally algebraically closed in its completion.
\end{proof}
\begin{example}
\label{algebraic-extension-of-t-henselian}Let $K$ be a model of
$\VFA$ which is transformally Henselian. By Lemma \ref{finite-t-henselian},
every finite $\sigma$-invariant Galois extension of $K$ is transformally
Henselian. On the other hand, an arbitrary algebraic extension of
$K$ need not be transformally Henselian. For example, let $k$ be
an inversive algebraically closed field of characteristic zero with
the trivial valuation. Let $\Gamma=\mathbf{Z}\left[\sigma^{\pm1}\right]$
and $\widetilde{\Gamma}=\mathbf{Q}\left[\sigma^{\pm1}\right]$, and
set $K=k\left(\left(t^{\Gamma}\right)\right)$ and $M=k\left(\left(t^{\widetilde{\Gamma}}\right)\right)$;
then $K$ is a model of $\VFA$ which is spherically complete
and therefore transformally Henselian. The algebraic closure of $K$
in $M$ is the directed union of the fields $k\left(\left(t^{A}\right)\right)$
where $\Gamma\subset A\subset\widetilde{\Gamma}$ is an abstract group
extension of $\Gamma$ with $\left[A\colon\Gamma\right]<\infty$. 

Now let $\widetilde{K}=k\left(t^{\frac{1}{2}}\right)_{\sigma}$; we
claim that $\widetilde{K}$ is not transformally Henselian. Indeed,
the element:
\[
x=t^{\frac{1}{2}}+t^{\frac{\sigma}{2}}+t^{\frac{\sigma^{2}}{2}}\ldots
\]

obeys $x^{\sigma}-x=t^{\frac{1}{2}}$, so $x$ lies in the transformal
Henselization of $\widetilde{K}$. But $x$ cannot be algebraic over
$K$, since it is not in $k\left(\left(t^{A}\right)\right)$ for any
finite extension $A$ of $\Gamma$.
\end{example}

\begin{example}
\label{ex:hens-talg-extensions-not-unique}Let $K$ be a Henselian
valued field. Then the valuation lifts uniquely to every algebraic
extension of $K$; in fact, Henselian valued fields are characterized by this property. In the $\omega$-increasing regime, this need not be the case. As a matter of fact, if $K$ is \textit{any} model of $\VFA$ which is nontrivially valued, then one can always find a transformally algebraic extension $L$ of $K$ which is nontrivially valued, and admitting two distinct $\omega$-increasing valuations.

For example, fix a nonzero element $t \in \mathcal{M}$. Let $L=K\left(x,y\right)_{\sigma}$ where $x^{\sigma} - x = y$ and $y^{\sigma}=ty$ with $x,y$ algebraically independent over $K$. Then $y$ has positive valuation, and $x$ can chosen of valuation zero or of strictly positive valuation (the latter case happens when $x$ is in the transformal Henselization of $K(y)_{\sigma}$).
\end{example}

\subsection{Construction of a Relative Transformal Henselization}

We now turn to the construction of a relative transformal Henselization
and the proof that its finite $\sigma$-invariant extensions are controlled.
We prove this in several steps.
\begin{lem}
\label{reductionarch}It is enough to prove Theorem \ref{transformalhenselization}
in the special case where $K$ is transformally Archimedean and transformally
Henselian.
\end{lem}

\begin{proof}
By Corollary \ref{absolute-t-hens-in-t-arch} an absolute transformal
Henselization exists in the transformally Archimedean settings, and
it is advertised by Theorem \ref{transformalhenselization}; we now
show how to reduce to the transformally Archimedean settings. It is
here that it is essential to give a relative version in the statement
of the theorem.

First, we may assume that $K$ is of finite transformal height. For
suppose that $K$ is written as the directed union $\left(K_{\alpha}\right)$
of models of this form, and the implicit index set has a minimal element
$\alpha_{0}$. Let $k_{\alpha}$ be the residue field of $K_{\alpha}$
and $\widetilde{k_{\alpha}}$ the transformal algebraic closure of
$k_{\alpha}$ in $\widetilde{k}$. Then the embeddings $\widetilde{k_{\alpha}}\hookrightarrow\widetilde{k_{\beta}}$
over $\widetilde{k_{\alpha_{0}}}$ uniquely lift to embeddings $\widetilde{K_{\alpha}}\hookrightarrow\widetilde{K_{\beta}}$
over $\widetilde{K_{\alpha_{0}}}$ and so the $\left(\widetilde{K_{\alpha}}\right)$
compatibly glue to a single transformally Henselian extension $\widetilde{K}$
of $K$ which is at once seen to obey the universal property. The
additional statement regarding the control of the finite $\sigma$-invariant
extensions being of finitary nature, its truth over the individual
$\widetilde{K_{\alpha}}$ implies it over $\widetilde{K}$.

Given this, by induction on the transformal height, we may assume that $K$ is
transformally Archimedean. For suppose that $\left(K_{21},K_{20},K_{10}\right)$
is a triplet of models of $\VFA$ and let $\widetilde{K_{0}}$
be a transformally algebraic extension of $K_{0}$. Let $\widetilde{K_{10}}$
be the universal transformally Henselian extension of $K_{0}$ reproducing
the embedding $K_{10}\hookrightarrow\widetilde{K_{0}}$ residually
and let $\widetilde{K_{21}}$ be the universal transformally Henselian
extension of $K_{21}$ reproducing the embedding of abstract difference
fields $K_{10}\hookrightarrow\widetilde{K_{10}}$. Gluing the valuations,
we obtain a model $\widetilde{K_{20}}$ of $\VFA$ over $K_{20}$
which is transformally Henselian as it is glued from a pair of transformally
Henselian valuations. The verification that $\widetilde{K_{20}}$
satisfies the universal property is immediate. Finally, the finite
$\sigma$-invariant separable extensions of $\widetilde{K_{20}}$
are the same as those of $\widetilde{K_{21}}$, since their underlying
difference fields are the same; the finite $\sigma$-invariant separable
extensions of $\widetilde{K_{21}}$ are controlled by those of $\widetilde{K_{10}}$,
which are in turn controlled by those of $\widetilde{K_{0}}$, giving
the additional statement.
\end{proof}
\begin{defn}
Let $K$ be a model of $\VFA$ and let $F$ be a difference
subfield. We say that $F$ is a \textit{field of representatives}
for $K$ if the following conditions are satisfied:

(1) The field $F$ inherits from $K$ the trivial valuation

(2) Every residue class of $K$ contains exactly one representative
from $F$. In other words, the residue map $\mathcal{O}_{K}\to k$
is bijective when restricted to $F$.
\end{defn}

\begin{lem}
\label{field-of-rep-exists-th}Let $K$ be a model of $\VFA$
which is transformally Henselian. Then $K$ admits a field of representatives
for the valuation.
\end{lem}

\begin{proof}
Let $F$ be a difference subfield of $K$ maximal with respect to
the property that it inherits from $K$ the trivial valuation. We
claim that $F$ is a field of representatives for the valuation of
$K$. To see this, let $F_{0}$ denote the isomorphic image of $F$
under the residue map. Suppose that an element $\alpha\in k$ is found,
transformally algebraic over $F_{0}$ but lying outside $F_{0}$.

Then after twisting (Lemma \ref{twist}), we have $g\alpha=0$ and $g'\alpha\neq0$
for some $gx\in F_{0}\left[x\right]_{\sigma}$. If $fx\in F\left[x\right]_{\sigma}$
is the canonical lift of $gx$ then by Hensel lifting one can find
an element $a\in\mathcal{O}_{K}$ with $fa=0$ and whose residue class
is equal to $\alpha$. Since $a$ is transformally algebraic over
$F$ the difference field $F\left(x\right)_{\sigma}$ is likewise
trivially valued, contradicting maximality; this shows that no element
of $k$ is transformally algebraic over $F_{0}$, unless it lies in
$F_{0}$. But in fact, if $\alpha$ is transformally transcendental
over $F_{0}$ and $a\in\mathcal{O}_{K}$ is any lift, then $F\left(a\right)_{\sigma}$
inherits the trivial valuation from $K$ as well. These two statements
combined show that the equality $F_{0}=k$ must necessarily hold,
which amounts precisely to the assertion that $F$ is a field of representatives
for $K$.
\end{proof}
\begin{rem}
\label{remhens}One important point, shown in the proof, is the following.
Let $K$ be a transformally Henselian model of $\VFA$ and let
$F_{0}$ be a difference subfield which inherits from $K$ the trivial
valuation. Let $k_{0}$ denote the image of $F_{0}$ under the residue
map and assume that $k$ is transformally algebraic over $k_{0}$.
Then there is a unique choice of a field of representatives for the
valuation of $K$ containing $F_{0}$, namely the relative transformal
algebraic closure $F$ of $F_{0}$ in $K$. 

The choice of a field of representatives for the valuation is not
canonical. However, if $F_{i}$ is a field of representatives for
$i=0,1$ then not only are they are isomorphic, but there is a distinguished
choice of an isomorphism $F_{0}\cong F_{1}$ between them, since both
are isomorphic to $k$ via the residue map. Furthermore, if $L$ is
a transformally Henselian extension of $K$ and $E_{i}$ is the relative
transformal algebraic closure of $F_{i}$ in $L$ then the isomorphism
$F_{0}\cong F_{1}$ extends to an isomorphism $E_{0}\cong E_{1}$
since both are identified, via the residue map, with the relative
transformal algebraic closure of $k$ in $l$.
\end{rem}

\begin{lem}
\label{canonicalva}Let $K$ be a model of $\VFA$ equipped
with a field of representatives $F$ for the valuation. Let us assume
that we are given an abstract transformally algebraic extension $\widetilde{F}$
of $F$ and denote by $\widetilde{K}=\widetilde{F}\otimes_{F}K$ the
tensor product, which we regard for now as an abstract difference
field.

(1) The difference field $\widetilde{K}$ can be uniquely expanded
to a model of $\VFA$ over $K$ (in particular, $\widetilde{F}$ admits a unique structure of a model of $\VFA$ over $F$, namely with the trivial valuation).

Let $\widetilde{K}$ be equipped with the valuation thus defined.
Then the following holds:

(2) Let $L$ be a model of $\VFA$ over $K$; then the restriction
map induces a bijection:
\[
\text{Emb}_{K}\left(\widetilde{K},L\right)=\text{Emb}_{F}\left(\widetilde{F},L\right)
\]

(3) The field $\widetilde{F}$ is a field of representatives for the
valuation of $\widetilde{K}$.

(4) Let $M$ be a $\text{h}$-finite $\sigma$-invariant separable
extension of $\widetilde{K}$ which is regular over $K$. Then $M$
is algebraically unramified over $\widetilde{K}$.
\end{lem}

\begin{proof}
First we claim that the valuation on $\widetilde{F}$ is uniquely determined, namely it must have the trivial valuation.
Since $\widetilde{F}$ is transformally algebraic over $F$, it can
be written as the directed union of difference field extensions of
finite algebraic transcendence degree. By Abhyankar's inequality,
the valuation group of each of these is of finite rational rank as
an abstract abelian group. Since no ordered transformal module is
of finite rational rank other than the zero module, we find that $\widetilde{F}$
is the directed union of trivially valued fields, and is therefore
itself trivially valued.

Furthermore, we assert that whenever $\widetilde{F}$ and $K$ are
jointly embedded over $F$ in a third model of $\VFA$, then
they are linearly disjoint over $F$ in that extension. For suppose
given a nontrivial linear combination $\sum a_{i}b_{i}=0$ where $a_{i}\in\widetilde{F}$
and the $b_{i}\in K$ are not all zero. Rescaling, we may assume that
the $b_{i}$ all lie in $\mathcal{O}_{K}$ and at least one is of
valuation zero. Now $\widetilde{F}$ is trivially valued; applying
the residue map and using the fact that $F$ is a field of representatives
for $K$ we find that the $a_{i}$ must have been linearly dependent
over $F$ to begin with. In fact, this shows that the valuation on
the tensor product is unique, since we are forced to set $v\left(\sum a_{i}b_{i}\right)=\min vb_{i}$
for $a_{i}\in\widetilde{F}$ linearly independent over $F$ and $b_{i}\in K$
not all zero. This formula is compatible with the action of $\sigma$,
so $\left(1\right)$, $\left(2\right)$ and $\left(3\right)$ follow.

It remains to show $\left(4\right)$. We may assume that $\widetilde{F}$
and $K$ are algebraically closed, and that $\widetilde{F}$ is of
finite algebraic transcendence degree over $F$; then the Henselian hull of $\widetilde{K}$ is the strictly Henselian hull of a finitely generated abstract Abhyankar extension of $K$ (generated by a transcendence basis of $\widetilde{F}$ over $F$, in which case Corollary
\ref{abhyankarcriterion} applies.
\end{proof}
\textbf{Proof of Theorem \ref{transformalhenselization}} By Lemma
\ref{reductionarch} it will be sufficient to handle the case where
$K$ is transformally Archimedean and transformally Henselian. Let
us fix, once and for all, a field of representatives $F$ for the
valuation of $K$, and assume given an abstract transformally algebraic
extension $\widetilde{F}$ of $F$. Let $\widetilde{K}$ be the absolute
transformal Henselization of $\widetilde{F}\otimes_{F}K$, equipped
with the canonical valuation of the previous Lemma.

Let $L$ be a transformally Henselian model of $\VFA$ over
$K$. By the universal property of the absolute transformal Henselization
in the transformally Archimedean settings, applying \ref{canonicalva}
again, we have the following canonical bijections induced by restriction:
\[
\text{Emb}_{K}\left(\left(\widetilde{F}\otimes_{F}K\right)^{\text{th}},L\right)=\text{Emb}_{K}\left(\widetilde{F}\otimes_{F}K,L\right)=\text{Emb}_{F}\left(\widetilde{F},L\right)
\]
Since $\widetilde{F}$ is transformally algebraic over $F$, every
embedding of $\widetilde{F}$ in $L$ over $F$ factors through the
relative transformal algebraic closure of $F$ in $L$, anyway, so
we have a further bijection
\[
\text{Emb}_{F}\left(\widetilde{F},L\right)=\text{Emb}_{F}\left(\widetilde{F},E\right)
\]
where $E$ is the transformal algebraic closure of $F$ in $L$. By
Remark \ref{remhens}, composition with the residue map of $L$ induces
a bijection:
\[
\text{Emb}_{F}\left(\widetilde{F},E\right)=\text{Emb}_{k}\left(\widetilde{k},l\right)
\]
It follows that $\widetilde{K}$ is the universal transformally Henselian
extension of $K$ reproducing the embedding $k\hookrightarrow\widetilde{k}$
residually, and by \ref{canonicalva} it has no unexpected finite
$\sigma$-invariant separable extensions. To finish we must only note
that every transformally algebraic extension $\widetilde{k}$ of $k$
arises as the residue field of a transformally Henselian extension
of $K$ which takes the form described above, but this is obvious:
set $\widetilde{F}=F\otimes_{k}\widetilde{k}$ where the implicit
map $k\hookrightarrow F$ is the map inverse to the isomorphism $F\cong k$
induced via the residue map of $K$.
\begin{example}
\label{t-henselization-of-alg-closed}Let $K$ be a model of $\VFA$
which is algebraically closed. Then the transformally Henselian hull of $K$ is not in general algebraically closed.

For example, let $\left(K_{20}, K_{21}, K_{10}\right)$ be a triplet with the following properties:

\begin{enumerate}
    \item The field $K_{20}$ is algebraically closed.
    \item The fields $K_{21}$ and $K_{10}$ are nontrivially valued.
    \item The field $K_{21}$ is transformally Henselian, but $K_{10}$ is not.
\end{enumerate}

In this situation the transformally Henselian hull of $K_{20}$ is not algebraically closed. 

Indeed, let $\widetilde{K_{10}}$ be the transformally Henselian hull of $K_{10}$; then after identifying $K_{10}$ with a trivially valued subfield of $K_{21}$ we have that the transformally Henselian hull of $K_{20}$ coincides as a field with the transformally Henselian hull of $\widetilde{K_{10}} \otimes_{K_{10}} K_{21}$. To show that the transformally Henselian hull of the latter is not algebraically closed, it will be sufficient to prove that the ordinary Henselian hull is not algebraically closed, since the transformally Henselian hull is regular over the algebraically Henselian hull.

So we now show that the Henselian hull of $L = \widetilde{K_{10}} \otimes_{K_{10}} K_{21}$ is not algebraically closed; for this, it is enough to show that $L$ admits a Galois extension which is wildly ramified. For example, we can consider the splitting field
of $X^{p}-X-at$ where $a\in\widetilde{K_{1}}$ lies
outside $K_{1}$ and $t\in K_{21}$ is an element of strictly negative
valuation.
\end{example}

\begin{cor}
\label{cor-thenselian}
Let $K$ be a model of $\VFA$ which is algebraically closed. If $K$ transformally Archimedean, or if $K$ is of characteristic zero, then the transformal Henselization of $K$ is algebraically closed. On the other hand:
\begin{enumerate}
    \item The transformal Henselization of an algebraically closed model of $\VFA$ need not be algebraically closed.
    \item Algebraic extensions of transformally Henselian models are not in general transformally Henselian.
\end{enumerate}
\end{cor}
\begin{proof}
For (1) and (2) see Example \ref{t-henselization-of-alg-closed} and Example \ref{algebraic-extension-of-t-henselian}.

For the proof, first assume that $K$ is algebraically closed and transformally Archimedean. By Lemma \ref{complete-t-henselian}, the transformal Henselization of $K$ is the relative transformal algebraic closure in the completion. By Krasner's lemma the completion of an algebraically closed valued field is algebraically closed; since the transformal Henselization is relatively algebraically closed inside an algebraically closed field, it is likewise algebraically closed.

Next assume that $K$ is of characteristic zero. By Theorem \ref{transformalhenselization}, the transformal Henselization of $K$ is an immediate extension of $K$ which is algebraically Henselian; using the ramification theory of valued fields \ref{subsec-ramificationtheory} we find that the transformal Henselization of $K$ is algebraically closed.
\end{proof}

\begin{example}
\label{example-monogonic-t-henselian}
Let $K$ be a model of $\VFA$ which is transformally Henselian and nontrivially valued. Let $\widetilde{K}$ be a transformally Henselian, transformally algebraic extension of $K$ which is finitely generated over $K$ as a transformally Henselian extension, that is, the field $\widetilde{K}$ is the transformally Henselian hull of $K\left(a\right)_{\sigma}$ for a finite tuple $a \in \widetilde{K}$. Then $\widetilde{K}$ is in fact generated as a transformally Henselian extension of $K$ by a single element. By induction, this reduces to the case of a two generators $a$ and $b$, which we may take to lie in $\widetilde{\mathcal{O}}$. After twisting we have $fa=0$ and $f'a \neq 0$ for some difference polynomial $fx$ with coefficients in $\mathcal{O}$. Let us set $c = a + tb$ where $t \in K$ is nonzero but of large valuation $\gamma$. If $\gamma$ is chosen large enough then it follows from Lemma \ref{newton} that $a$ is the unique root of $fx$ in the open ball of valuative radius $\gamma$ around $c$, so it lies in the transformally Henselian hull of $K\left(c\right)_{\sigma}$. Thus $a$ and $b$ both lie in the transformally Henselian extension of $K$ generated by $c$.
\end{example}

\subsection{Strict Transformal Henselization}
\begin{defn}
Let $K$ be a model of $\VFA$. We say that $K$ is \textit{strictly
transformally Henselian} if the following conditions are satisfied:

(1) The field $K$ is transformally Henselian

(2) The residue field $k$ of $K$ is a model of $\ACFA$
\end{defn}

\begin{prop}
\label{strict-th}(1) Every strictly transformally Henselian field
is algebraically strictly Henselian

(2) The ultraproduct of algebraically strictly Henselian, perfect
Frobenius transformal valued fields is strictly transformally Henselian

(3) The directed union of strictly transformally Henselian fields
is strictly transformally Henselian
\end{prop}

\begin{proof}
Clear, using \ref{lem:(Basic-properties)}, model completeness of
$\ACFA$, and Fact \ref{frobacfa}.
\end{proof}
\begin{defn}
Let $\nicefrac{\widetilde{K}}{K}$ be an extension of models of $\VFA$.
We say that $\widetilde{K}$ is a \textit{strict transformal Henselization}
of $K$ if the following conditions are satisfied:

(1) The field $\widetilde{K}$ is strictly transformally Henselian

(2) The field $\widetilde{K}$ is the universal transformally Henselian
extension of $K$ reproducing the embedding $k\hookrightarrow\widetilde{k}$
residually, as in the statement of Theorem \ref{transformalhenselization}
\end{defn}

\begin{prop}
\label{strict-t-henselization}Let $K$ be a model of $\VFA$
which is algebraically strictly Henselian. Then $K$ admits a strict
transformal Henselization $K^{\text{sth}}$, which is unique up to
elementary equivalence over $K$. Furthermore, every finite $\sigma$-invariant
Galois extension of $K^{\text{sth}}$ uniquely descends to a finite
$\sigma$-invariant Galois extension of $K$.
\end{prop}

\begin{proof}
For $i=0,1$ let $K_{i}$ be a strict transformal Henselization of
$K$ with residue field $k_{i}$. Since $k$ is separably closed,
the theory $\ACFA_{k}$ is complete, so $k_{0}$ and $k_{1}$ are
elementarily equivalent over $k$. Every difference subfield of $k_{0}$
finitely generated over $k$ then embeds in $k_{1}$ over $k$ and
conversely. By the universal property of the strict transformal Henselization,
every difference subfield of $K_{0}$ finitely generated over $K$
as a transformally Henselian extension then embeds in $K_{1}$ over
$K$; by symmetry we can apply a back and forth argument to conclude.
The last part regarding the control of the finite $\sigma$-invariant
separable extensions, as well as the existence of a strict transformal
Henselization, follows again from Theorem \ref{transformalhenselization}.
\end{proof}
\begin{defn}
\label{def:tunramified}
Let $\nicefrac{L}{K}$ be an extension of models of $\VFA$.
We say that $L$ is \textit{transformally unramified} over $K$ if
it embeds over $K$ in a strict transformal Henselization of $K$.
\end{defn}

\begin{rem}
Let $\nicefrac{L}{K}$ be an extension of models of $\VFA$
with $L$ transformally Henselian. Then $L$ is transformally unramified
over $K$ if and only if it is the universal transformally Henselian
extension of $K$ associated with the embedding $k\hookrightarrow l$
of residue fields as in Theorem \ref{transformalhenselization}.
\end{rem}

\begin{prop}
\label{tunr}Let $\nicefrac{L}{K}$ be an extension of models of $\VFA$.
Then there is a canonical factorization $K\subset L^{\text{tunr}}\subset L$
where $L^{\text{tunr}}$ is transformally unramified over $K$ and
maximal with respect to this property under inclusion.
\end{prop}

\begin{proof}
Functoriality here means that if $K\subset M\subset L$ is a tower
of models of $\VFA$ then $M^{\text{tunr}}=L^{\text{tunr}}\cap M$.
Let $l_{0}$ be the relative transformal algebraic closure of $k$
in $l$ and $L_{0}$ the universal transformally Henselian extension
of $K$ reproducing the embedding $k\hookrightarrow l_{0}$ residually;
then there is a canonical embedding $L_{0}\hookrightarrow L^{\text{th}}$
over $K$, and the intersection $L^{\text{tunr}}=L_{0}\cap L$ (taken inside $L^{\text{th}}$)
is verified at once to be as advertised by the statement of the Proposition.
\end{proof}
\begin{rem}
\label{tunr-lifting}Let $K$ be a model of $\VFA$ which is
transformally Henselian and let $\widetilde{K}$ be a transformally
Henselian extension of $K$ which is transformally unramified over
$K$. Let $F$ be a field of representatives of $K$ and $\widetilde{F}$
the relative transformal algebraic closure of $F$ in $\widetilde{K}$.
It follows from Lemma \ref{canonicalva} that $\widetilde{F}$ is
a field of representatives for the valuation of $\widetilde{K}$,
the fields $\widetilde{F}$ and $K$ are linearly disjoint over $F$
in $\widetilde{K}$ and that $\widetilde{K}$ is the transformal Henselization
of $\widetilde{F}\otimes_{F}K$.
\end{rem}

\begin{lem}
\label{transformally-transcendental-hens}Let $\nicefrac{\widetilde{K}}{K}$
be an extension of transformally Henselian models of $\VFA$.
Let us assume no element of $\widetilde{K}$ is transformally algebraic
over $K$, unless it already lies in $K$; then the same is true of
$\nicefrac{\widetilde{k}}{k}$.
\end{lem}

\begin{proof}
We show the contrapositive. Let $\alpha\in\widetilde{k}$ be an element
transformally algebraic over $k$; by Lemma \ref{twist}, after twisting,
we have $f\alpha=0$ and $f'\alpha\neq0$ for some nonzero difference
polynomial $fx\in k\left[x\right]_{\sigma}$. Let $gx\in\mathcal{O}\left[x\right]_{\sigma}$
and $a\in\widetilde{\mathcal{O}}$ lift $fx$ and $\alpha$ respectively ; then $vga>0$
and $vg'a=0$. Since $\widetilde{K}$ is transformally Henselian we
can find $b\in\widetilde{\mathcal{O}}$ with $gb=0$ and whose residue
class coincides with that of $\alpha$. So $b$ is transformally algebraic
over $K$ and yet it lies outside $K$.
\end{proof}
\begin{lem}
\label{linearly-disjoint-unramified}Let $L\hookleftarrow K\hookrightarrow M$
be transformally Henselian models of $\VFA$ jointly embedded
over $K$ inside some model $N$. Let us assume that $L$ is transformally
algebraic over $K$ and that $L$ and $M$ are linearly disjoint over
$K$ in $N$; then $l$ and $m$ are linearly disjoint over $k$ in
$n$.
\end{lem}

\begin{proof}
Let $\widetilde{M}$ be the relative transformal algebraic closure
of $K$ in $M$. Since $M$ is transformally Henselian, it follows
from Hensel lifting that $\widetilde{m}$ is the relative transformal
algebraic closure of $k$ in $m$. By transitivity of linear disjointness,
in order to show that $l$ and $m$ are linearly disjoint over $k$
it is enough to prove that $l$ and $\widetilde{m}$ are linearly
disjoint over $k$ and that $l\otimes_{k}\widetilde{m}$ and $m$
are linearly disjoint over $\widetilde{m}$. The second condition
is automatic as $l\otimes_{k}\widetilde{m}$ is transformally algebraic
over $\widetilde{m}$ whereas $m$ is purely transformally transcendental
over $m$; see \cite{chatzidakis2004model}. We may therefore replace
$M$ by $\widetilde{M}$ and assume that $M$ is transformally algebraic
over $K$. 

We may assume that $M$ and $L$ are transformally unramified over
$K$ as in Proposition \ref{tunr}; the hypothesis is preserved and
the residue fields are unchanged. Now $K$ is transformally Henselian;
fix a field of representatives for the valuation and denote it by $k_0$. By \ref{tunr-lifting} we find that the relative
transformal algebraic closure of $k_0$ in $L$ is a field of representatives
for $L$; let us denote it by $l_{0}$ and likewise for $M$. Then
$l_{0}$ and $K$ are linearly disjoint over $k_{0}$ and similarly
$m_{0}$ and $K$ are linearly disjoint over $k_{0}$. Since $L$
and $M$ are linearly disjoint over $K$ it follows in particular
that $l_{0}\otimes_{k}K$ and $m_{0}\otimes_{k}K$ are linearly disjoint
over $K$ in $N$. We may assume that $N$ is transformally Henselian;
so $m_{0}$ and $l_{0}$ are linearly disjoint over $k_{0}$ in $N$
and $m_{0}\otimes_{k_{0}}l_{0}$ is contained in a trivially valued
difference subfield of $N$. By passing to residue classes we find
that $l$ and $m$ are linearly disjoint over $k$ in $n$, which
is what we wanted.
\end{proof}
\begin{prop}
\label{unramified-linearly-disjoint amalgamation}Let $K$ be algebraically
closed and transformally Henselian. Let $M$ be an algebraically closed,
transformally Henselian extension of $K$. Let $l$ be a transformally
algebraic extension of $k$ and $L$ the universal transformally Henselian
extension of $K$ reproducing the embedding $k\hookrightarrow l$
at the level of residue fields. Let us assume that $L$ and $M$ are
jointly embedded over $K$ in some fourth extension so as to render
$m$ and $l$ linearly disjoint over $k$;
then $L$ and $M$ are linearly disjoint over $K$ in that extension.
Similarly if $L$ is replaced by an algebraically closed and transformally
Henselian hull.
\end{prop}

\begin{rem}
    In retrospect, heuristically, the statement of Proposition \ref{unramified-linearly-disjoint amalgamation} is a consequence of the full embeddedness of the residue field in $\widetilde{\VFA}$. Indeed, using parameters in $K$, every element of $L$ is in quantifier free definable bijection with an element of the residue field; by full embeddedness and elimination of imaginaries in $\ACFA$, the interaction of $L$ and $M$ over $K$ is governed by the interaction of the residue fields. Our construction uses an automorphism argument as a substitute for the universal domain (whose existence is yet to be established).
\end{rem}

\begin{proof}

For every rational number $q \in \mathbf{Q}$ let $l_q$ be a copy of $l$, with the convention that $l_0 = l$. Let $L_q$ be the universal transformally Henselian extension of $K$ reproducing the embedding $k \hookrightarrow l_q$ at the level of residue fields; with this convention, we have $L_0 = L$. Let also $\widetilde{l}$ be the compositum of linearly disjoint copies of $l_q$ over $k$\footnote{this is possible since $K$, and hence $k$, are algebraically closed}.

For every $r_0 < \ldots < r_n \in \mathbf{Q}$ and $q_0 < \ldots < q_n \in \mathbf{Q}$ there is an automorphism of $\widetilde{l}$ over $k$ which carries $l_{r_i}$ to $l_{q_i}$; by the universal mapping property of $\widetilde{L}$, this lifts to an automorphism of $\widetilde{L}$ over $K$ which carries $L_{r_i}$ to $L_{q_i}$.

These automorphisms of $\widetilde{L}$ over $K$ can be lifted to an automorphism over $M$. More precisely, for every $q \in \mathbf{Q}$ let $m_q = l_q \otimes_{k} m$. Let $\widetilde{m}$ be the compositum of linearly disjoint copies of $m_q$ over $m$, and let $M_q$ and $\widetilde{M}$ be as before; then every automorphism of $\widetilde{L}$ over $K$ lifts (uniquely) to an automorphism of $\widetilde{M}$ over $M$.

It follows that every element of $\widetilde{L}$ not in $K$ has infinite orbit under automorphisms of $\widetilde{M}$ over $M$. Thus if $I \subseteq \mathbf{Q}$ is a finite subset, and $L_I$ is the compositum of the $L_q$ for $q \in I$, then the intersection $L_I \cap M$ inside $\widetilde{M}$ coincides with $K$. Since all fields are perfect, it follows from the theory of canonical bases that $\widetilde{L}$ and $M$ are linearly disjoint over $K$, hence in particular $L = L_0$ and $M$ are linearly disjoint over $K$.

This gives the result when $L$ is the universal transformally Henselian extension of $K$ with residue field $l$. In general let $E$ be an algebraically closed and transformally Henselian hull of $K$. Then the proof follows with the same reasoning and notation. 

Namely, find copies $E_q$ of $E$ inside an ambient extension of $K$ with linearly disjoint residue fields. Then the $E_q$ are themselves linearly disjoint over $K$; namely inside the compositum $\widetilde{E}$ of the $E_q$, one can characterize $E_q$ as the set of elements with finite orbit under automorphisms of $\widetilde{E}$ over $K$ fixing $l_q$ pointwise.

\end{proof}
In particular:
\begin{cor}
Let $K$ be algebraically closed and transformally Henselian. Let
$L$ be a transformally unramified transformally Henselian extension
of $K$, or the algebraically closed and transformally Henselian hull
of such an extension. Let $M$ be a model of $\VFA$ over $K$
and assume that $L$ and $M$ are jointly embedded over $K$ in some
ambient extension. Then $l$ and $m$ are linearly disjoint over $k$
if and only if $L$ and $M$ are linearly disjoint over $K$.
\end{cor}

\begin{proof}
This follows immediately from Proposition \ref{unramified-linearly-disjoint amalgamation}
and Lemma \ref{linearly-disjoint-unramified}.
\end{proof}
\newpage{}

\section{\label{sec:The-Transformal-Herbrand}The Transformal Herbrand Function}

\subsection{The Transformal Herbrand Function}

Let $K$ be an (abstract) algebraically closed  and nontrivially valued field. It follows from quantifier elimination in $\ACVF$ that every definable subset of $K$ in one variable is a boolean combination of balls. Thus if $b$ is a ball in $K$ then the partial type of an element of the ball avoiding any finite union of proper definable subballs is complete; it is called the \textit{generic type} of the ball. The purpose of this subsection is to lift this notion to the transformal settings.

\begin{defn}
\label{genericofaball}Let $K$ be a model of $\VFA$ and let
$a$ be an element within an extension.
\begin{enumerate}
    \item We say that $a$ is \textit{generic in $\mathcal{O}$} over $K$
if $va=0$ and the residue class of $a$ is transformally transcendental
over $k$

\item We say that $a$ is \textit{generic in $\mathcal{M}$} over $K$
if $0<\sigma^{n}\cdot\varepsilon<\gamma$ for all $0<\gamma\in\Gamma$
and all $0<n\in\mathbf{N}$, where $\varepsilon=va$
\end{enumerate}

The group of affine $K$-transformations acts transitively on the
space of balls with center and radii in $K$, open or closed; the
notion of the generic of any other ball is then clear. We will often
confuse between the realization of the generic of a closed ball and
the ball itself, using the notation $vfb$ to denote the generic value
of the difference polynomial $f$ on the closed ball $b$.
\end{defn}

\begin{lem}
\label{generic-qf}
Let $K$ be a model of $\VFA$. Then the quantifier free type of an element realizing the generic type of $\mathcal{O}$ or $\mathcal{M}$ over $K$ is a complete quantifier free type over $K$ in the language of transformal valued fields.
\end{lem}
\begin{proof}
This follows from Lemma \ref{generic-acvf} below.  Explicitly, if $fx=\sum c_{\mu}x^{\mu}$ where $c_{\mu} \in K$ then the generic $x$ of
$\mathcal{O}$ is characterized by the formula $vfx=\min\left\{ vc_{\nu}\right\} $.
More generally, if $a\in K$ and $\gamma\in\Gamma\otimes\mathbf{Q}\left(\sigma\right)$
then the generic of the closed ball of valuative radius $\gamma$
around the element $a$ is characterized by the formula $vfx=\min\left\{ \beta_{\mu}+\mu\cdot\gamma\right\} $
where $\beta_{\mu}=vf_{\mu}a$; here $f_{\mu}$ is the $\mu$-th transformal
derivative of $f$ as in \ref{eq:taylor}.
Similarly, the generic type of $\mathcal{M}$ is characterized by the formula $vfx = \min \left\{vc_{\mu} + \mu \cdot \epsilon \right\}$, where $\epsilon = vx$.

\end{proof}

In what follows we will generalize Definition \ref{genericofaball} to the case of an infinite intersection of balls. Moreover, it will be shown that if $K$ is algebraically closed and $a$ is generic over $K$ in a ball, then the field $K\left(a\right)_{\sigma}$ has no nontrivial $h$-finite $\sigma$-invariant separable extensions. This is easy to rule out in characteristic zero, but in finite characteristic the argument is slightly more subtle; some valuation theoretic preliminaries are given below.

\begin{lem}
\label{generic-acvf}Let $K$ be a valued field. Let $a,b$ be elements
within an extension and put $\alpha=va,\beta=vb$. Assume $\alpha$
does not lie in $\Gamma\otimes\mathbf{Q}$; then every isomorphism
$\Gamma\left[\alpha\right]\cong\Gamma\left[\beta\right]$ over $\Gamma$
which carries $\alpha$ to $\beta$ lifts to an isomorphism $K\left(a\right)\cong K\left(b\right)$
of valued fields over $K$ which carries $a$ to $b$. Similarly if
$va=vb=0$ and the residue classes are transcendental over $k$.
\end{lem}

\begin{proof}
This follows easily from quantifier elimination in $\ACVF$. If $a, b$ have a transcendental residue class then they are generic in $\mathcal{O}$ and if $\alpha, \beta$ are not in $\Gamma \otimes \mathbf{Q}$ then $a, b$ are generic in an open ball or an $\infty$-definable ball around $0$.
\end{proof}

\begin{lem}
\label{hensel-disjoint}Let $K$ be a valued field and $L$ a Galois
extension of $K$. Let us assume that every automorphism of $L$ over
$K$ is, in fact, an automorphism of valued fields; then $L$ and
$K^{h}$ are linearly disjoint over $K$.
\end{lem}

\begin{proof}
Fix an algebraic closure $K^{a}$ of $K$ in the category of valued
fields. If every automorphism of $L$ over $K$ preserves the valued
field structure, then every such automorphism lifts to an automorphism
of $K^{a}$; it must then fix $K^{h}$ pointwise. So the claim follows
from Galois theory.
\end{proof}
\begin{lem}
\label{lem:Let--be}Let $K$ be an abstract valued field with residue
field $k$ and valuation group $\Gamma$. Let $a$ be an element within
some extension and $M=K\left(a\right)$. Let us assume that one of
the following holds:

(1) We have $va=0$ and the residue class of $a$ is transcendental
over $k$.

(2) The quantity $va$ lies outside $\Gamma\otimes\mathbf{Q}$

(3) There is a nested transfinite family $b_{0}\supset b_{1}\supset\ldots\supset b_{\alpha}\supset\ldots_{\alpha<\lambda}$
of closed balls in $K$ with the following property: for every polynomial
$fx\in K\left[x\right]$ the quantity $vfb_{\alpha}$ stabilizes for
$\alpha$ sufficiently large. Moreover the element $a$ lies in the
intersection.

Then the following holds. Let $L$ be a Galois valued field extension
of $K$. Then the valuation lifts uniquely from $L$ to $L\otimes_{K}M$.
\end{lem}

\begin{proof}
We work in $\ACVF$. The condition that the valuation lifts uniquely
from $L$ to $L\otimes_{K}M$ can be reformulated in model theoretic terms:
it means that $\tp_{\ACVF}\left(\nicefrac{a}{K}\right)\models\tp_{\ACVF}\left(\nicefrac{a}{L}\right)$.
First we deal with (3). The requirement that $a$ lies in the intersection
determines the valued field structure on $K\left(a\right)$, namely
we must set $vfa=vfb_{\alpha}$ where $\alpha$ is taken sufficiently
large. In the language of pseudo-Cauchy sequences, the element $a$
is a pseudo-limit of a pseudo-Cauchy sequence of transcendental type
with no limit in $K$. So there are no points algebraic over $K$
in the intersection. If $L$ is an algebraic closure of $K$ then
the intersection can have no points over $L$ either; so the result
follows from C-minimality in $\ACVF$. For (1) use the fact that if $L$ is algebraic over $K$ then $l$ is algebraic over $k$, combined with Lemma \ref{generic-acvf}; for
(2) use the fact that $\Gamma_{L}\subset\Gamma\otimes\mathbf{Q}$
and Lemma \ref{generic-acvf}
\end{proof}
\begin{rem}
\label{linearly-disjoint-generic}In the settings of the Lemma we find
that $M^{h}$ is a regular extension of $K^{h}$. Indeed, we may assume
that $K$ is Henselian. If $L$ is an abstract Galois extension of
$K$ then the valuation lifts uniquely from $K$ to $L$ by Henselianity
of $K$. The Lemma implies that the valuation lifts uniquely to $L\otimes_{K}M$;
so every automorphism of $L\otimes_{K}M$ over $M$ preserves the
valuation, hence Lemma \ref{hensel-disjoint} applies.
\end{rem}

\begin{defn}
Let $K$ be a model of $\VFA$, let $fx\in K\left[x\right]_{\sigma}$
be a difference polynomial and let $b_{\lambda_{0}}$ be a closed
ball of valuative radius $\lambda_{0}$. Let us define\footnote{Recall the definition of a transformally affine function on a model
of $\widetilde{\omega\OGA}$ given in \ref{transformallyaffine} and
the notion of the generic point of a closed ball given in \ref{genericofaball}}
\[
\Psi:\left(-\infty,\lambda_{0}\right)\cap\Gamma\otimes\mathbf{Q}\left(\sigma\right)\to\Gamma\otimes\mathbf{Q}\left(\sigma\right)
\]
by the rule $\Psi\lambda=vfb_{\lambda}$, where for $\lambda<\lambda_{0}$,
the closed ball $b_{\lambda}$ is the unique closed ball of valuative
radius $\lambda$ containing $b_{\lambda_{0}}$. We say that $\Psi$
is the \textit{transformal Herbrand function of $f$ above $b_{\lambda_{0}}$}.
\end{defn}

\begin{rem}
The transformal Herbrand function is analogous to the Hasse-Herbrand
transition function in local class field theory; see \cite{serre2013local}.
\end{rem}

The case $\lambda_{0}=\infty$ of a point is allowed. If $f$ is clear
from the context, we will just refer to the \textit{Herbrand function}
\textit{above $b_{\lambda_{0}}$.}
\begin{lem}
\label{lem-transformal-herbarnd-concave}
The function $\Psi$ is continuous, nondecreasing, concave and piecewise
transformally affine with slopes in $\mathbf{N}\left[\sigma\right]$.
\end{lem}

\begin{proof}
We may assume $\lambda_{0}=\infty$, i.e, it is enough to consider
the Herbrand function above a point $a\in K$. For $\mu\in\mathbf{N}\left[\sigma\right]$
let $c_{\mu}=f_{\mu}a$ and $\beta_{\mu}=vc_{\mu}$. Then $\Psi\lambda=\min\left\{ \beta_{\mu}+\mu\cdot\lambda\right\} $
is the pointwise minimum of transformally affine functions with slopes
in $\mathbf{N}\left[\sigma\right]$. This function is continuous,
nondecreasing and concave since the pointwise minimum of continuous,
nondecreasing and concave functions is again one. It is, moreover,
explicitly piecewise transformally affine.
\end{proof}
\begin{example}
\label{ex:point-count} Let $K$ be a model of $\VFA$ which
is algebraically closed and let $c\in K$ be such that $vc<0$. Let $a\in K$
solve the equation $a^{p}-a=c$ and $\Psi$ the transformal Herbrand
function of $x^{p}-x-c$ below $a$. Then $\Psi'\lambda=1$ when $\lambda>0$
and $\Psi'\lambda=p$ for $\lambda<0$ (here $\Psi'$ is the derivative of $\Psi$; see definition \ref{def:transformallyaff}). Indeed, since $f$ is additive we have $f_1 = -1$, $f_p = 1$ and $f_n = 0$ for all other values of $n$.

Observe that all solutions of this
polynomial lie in the same closed ball of valuative radius $0$, so
for $\lambda\neq0$ the quantity $\Psi'\lambda$ is the number of
roots of the polynomial in the closed ball of valuative radius $\lambda$
around the element $a$.
\end{example}

\begin{lem}
\label{valuation-group-increase}Let $K$ be a model of $\VFA$;
then there is a model $\widetilde{K}$ of $\VFA$ over $K$
with $\widetilde{\Gamma}=\Gamma\otimes\mathbf{Q}\left(\sigma\right)$.
\end{lem}

\begin{proof}
For every $\gamma\in\Gamma\otimes\mathbf{Q}\left(\sigma\right)$ let
$a_{\gamma}$ be generic in the closed ball of valuative radius $\gamma$
around $0$, and take $\widetilde{K}$ to be generated by mutually
independent realizations of the $a_{\gamma}$.
\end{proof}
\begin{prop}
\label{rootinaball}Let $K$ be a model of $\VFA$, let $fx\in K\left[x\right]_{\sigma}$
be a difference polynomial and let $b$ be a closed ball. Then the
following are equivalent:

\begin{enumerate}
    \item There is a model $\widetilde{K}$ of $\VFA$ over $K$ and
a $\widetilde{K}$-root of $f$ in the closed ball $b$
\item The transformal Herbrand function of $f$ above $b$ is strictly
increasing.
\end{enumerate}

\end{prop}

\begin{proof}
Suppose that there is to be a root $a$ of $fx$ in some extension
and let $\Psi$ be the Herbrand function of $f$ above $a$. We have
$f_{\mu}a\neq0$ for some $0<\mu\in\mathbf{N}\left[\sigma\right]$
whereas $fa=0$, so $\Psi\lambda=\min_{\mu\neq0}\left\{ \beta_{\mu}+\mu\cdot\lambda\right\} $
is strictly increasing as the pointwise minimum of finitely many strictly
increasing functions. It follows that $\left(1\right)\Rightarrow\left(2\right)$. 

For the converse, we may assume that $K$ is spherically complete,
strictly transformally Henselian and that $\Gamma=\Gamma\otimes\mathbf{Q}\left(\sigma\right)$
is transformally divisible; here we use Lemma \ref{valuation-group-increase},
Proposition \ref{spherically-complete-extension} and Theorem \ref{transformalhenselization}.

First we claim that there is a closed ball strictly contained in $b$
whose Herbrand function is strictly increasing as well. To see this,
let the valuative radius of the closed ball $b$ be equal to $\gamma$.
Let $b'=\gamma\mathcal{O}$ be the closed ball of valuative radius
$\gamma$ around the origin. The transformal Herbrand function of
$f\left(x\right)$ below $b$ then agrees with the transformal Herbrand
function of $f\left(x-a\right)$ below $\gamma\mathcal{O}$, for any
element $a$ in the ball $b$; so without loss we have $b=\gamma\mathcal{O}$.
If the constant term of $f$ is equal to zero, then zero is a root
inside $b$; so assume that the constant term is nonzero. Replacing
$fx$ by a scalar multiple shifts the transformal Herbrand function
by an additive constant and does not change the roots; thus we may
assume that the constant term of $fx$ is equal to $1$.

Let us write $fx=1+c_{1}x+c_{2}x^{2}+\ldots+c_{\sigma}x^{\sigma}+\ldots$
and for $0<\nu\in\mathbf{N}\left[\sigma\right]$ put $\beta_{\nu}=vc_{\nu}$.
The transformal Herbrand function of $f$ above $\gamma\mathcal{O}$
is then $\Psi\left(\lambda\right)=\min\left\{ 0,\beta_{\nu}+\nu\cdot\lambda\right\} $
for $\lambda<\gamma$. By assumption this function is strictly increasing
in the regime where $\lambda<\gamma$. On the other hand, we have
$\Psi\lambda=0$ for $\lambda$ sufficiently large. Since $\Gamma=\Gamma\otimes\mathbf{Q}\left(\sigma\right)$
and $\Psi$ is piecewise transformally affine and increasing, there
is a unique $\gamma_{0}\geq\gamma$ with $\Psi\left(\gamma_{0}\right)=0$.
Then the transformal Herbrand function of $f$ above $\gamma_{0}\mathcal{O}$
is strictly increasing; shrinking $b$ we me may assume that $\Psi\left(\gamma\right)=0$.
This implies that $\beta_{\nu}+\nu\cdot\gamma\geq0$ for all $0<\nu\in\mathbf{N}\left[\sigma\right]$
with equality for at least one choice of $0<\nu\in\mathbf{N}\left[\sigma\right]$. 

Now choose an element $t\in K$ with $vt=\gamma$ and let $g\left(x\right)=f\left(t^{-1}x\right)$.
Replacing $\gamma\mathcal{O}$ by $\mathcal{O}$ and $f\left(x\right)$
by $g\left(x\right)$ we are back in precisely the same situation;
so now $fx$ has coefficients in the valuation ring, the constant
term is equal to $1$, and there is some $0<\nu\in\mathbf{N}\left[\sigma\right]$
such that the valuation of the coefficient of $x^{\nu}$ is zero.
Since $k$ is a model of $\ACFA$ there is an element $a\in\mathcal{O}$
with $vfa>0$, so $vfa$ exceeds the generic value of $f$ on the
ball. If $fa=0$ then we are done; otherwise let $\gamma$ be the
maximal singular point of the transformal Herbrand function of $f$
above $a$. Then $\gamma>0$, and the closed ball of valuative radius
of $\gamma$ around $a$ will do.

In this way we find a sequence $b_{0}\supset b_{1}\ldots\supset b_{n}\supset\ldots$
of closed balls such that the transformal Herbrand function of $f$
above each of them is strictly increasing. Since $K$ is spherically
complete there is an element $a$ in the intersection. First assume
that the valuative radii of the $b_{n}$ are unbounded; then the transformal
Herbrand function of $f$ above the point $a$ agrees is of the form
$\beta_{\nu}+\nu\cdot\gamma$ for large enough $\gamma$ and some
$0<\nu\in\mathbf{N}\left[\sigma\right]$; so $vfb_{n}\to\infty$,
hence $fa=0$ by continuity. Otherwise find the maximal singular point \footnote{i.e the point where the slope changes}
$\lambda$ of the transformal Herbrand function of $f$ above $a$;
so $\lambda$ exceeds the valuative radius of each $b_{n}$, hence
we find a closed ball $b_{\omega}$ inside the intersection whose
transformal Herbrand function is strictly increasing. Continuing in
this manner into the transfinite, we conclude.
\end{proof}

\subsubsection{Generic Extensions}

The following definition is ad hoc but will prove useful later on;
it generalizes Definition \ref{genericofaball}.
\begin{defn}
\label{generic-definition}Let $K$ be a model of $\VFA$ and
let $a$ be an element within an extension. We say that $a$ is \textit{generic}
over $K$ if one of the following holds, possibly after an affine
change of variables:

\begin{enumerate}

\item The element $a$ is generic in $\mathcal{O}$ over $K$, that
is we have $va=0$ and the residue class is transformally transcendental
over $k$

\item The quantity $va$ is transformally transcendental over $\Gamma$,
i.e it does not lie in $\Gamma\otimes\mathbf{Q}\left(\sigma\right)$

\item There is a nested family $b_{0}\supset b_{1}\supset\ldots\supset b_{\alpha}\supset\ldots_{\alpha<\lambda}$
of closed balls in $K$ with the property that for every difference
polynomial $fx\in K\left[x\right]_{\sigma}$ there is some $\alpha$
such that the transformal Herbrand function of $f$ is constant above
$b_{\alpha}$, and the element $a$ lies in the intersection. Thus
$vfb_{\alpha}$ is independent of $\alpha$ for $\alpha$ large enough
and we have $vfa=vfb_{\alpha}$ for some (or any) such $\alpha$.
\end{enumerate}
\end{defn}

\begin{rem}
\label{rem-qf-generic}
In the settings of the third clause of Definition \ref{generic-definition},
using the concavity of the transformal Herbrand function, we deduce that if $a$ is any element of the intersection in some extension of $K$
and $fx\in K\left[x\right]_{\sigma}$ is a difference polynomial,
then there is a closed $K$-ball $b$ containing $a$ such that the
transformal Herbrand function of $f$ below $b$ is constant, and we must have $vfa=vfb$.
Thus $K\left(a\right)_{\sigma}$ is an immediate extension of transformal
valued fields. In the language of \cite{azgin2010valued}, the element
$a$ is a pseudo-limit of a pseudo-Cauchy sequence of transformally
transcendental type over $K$. In particular, the quantifier free
type of an element in the intersection of the balls is a complete quantifier free
type over $K$, a realization of which is transformally transcendental
over $K$.
\end{rem}

\begin{rem}
\label{rem-qf-complete-valgp-adjunction}
In the settings of the second clause of Definition \ref{generic-definition}, using $o$-minimality of $\widetilde{\omega\OGA}$, one verifies as in the proof of Lemma \ref{generic-qf} that the quantifier free type of $a$ is determined by the cut it determines over $\Gamma \otimes \mathbf{Q}\left(\sigma\right)$.
\end{rem}

\begin{prop}
\label{generic-closed-open}Let $K$ be an algebraically closed model
of $\VFA$. Let $a$ be generic over $K$ and $M$ the model
of $\VFA$ generated by $a$ over $K$; then $M^{h}$ has no
nontrivial finite $\sigma$-invariant Galois extensions.
\end{prop}

Recall (Fact \ref{fact-ttranscendental-nofing}) that if $K$ is an inversive algebraically closed difference field and $a$ is transformally transcendental over $K$, then the difference field $K\left(a\right)_\sigma$ has no nontrivial finite $\sigma$-invariant Galois extensions. Proposition \ref{generic-closed-open} is a generalization of this; the proof is similar, using a linear disjointness argument.

\begin{proof}
For $n\in\mathbf{Z}$ let us set $a_{n}=a^{\sigma^{n}}$. First observe
that the extension $M\left(a_{0},\ldots,a_{n}\right)\subset M\left(a_{0},\ldots,a_{n+1}\right)$
takes one of the forms described in Lemma \ref{lem:Let--be}. For
the case $a$ generic in $\mathcal{O}$ this is clear, and similarly
if the quantity $va$ is not in $\Gamma\otimes\mathbf{Q}\left(\sigma\right)$.
So assume that $M$ is an immediate extension of $K$. The element
$a=a_{0}$ lies in the intersection of a nested family $\mathcal{B}$
of closed balls in $K$, and for every difference polynomial $fx\in K\left[x\right]_{\sigma}$,
there is a small closed $b$ in $\mathcal{B}$ such that the transformal
Herbrand function of $f$ is constant above $b$. Since $K$ is inversive
it follows that $a_{n+1}$ lies in the intersection of the nested
family $\mathcal{B}^{\sigma^{n+1}}=\left\{ \sigma^{n+1}\left(b\right)\colon b\in\mathcal{B}\right\} $.
Now our assumption on $\mathcal{B}$ implies that for every algebraic
polynomial $g\left(x^{\sigma^{n+1}}\right)\in M\left[a_{0},\ldots,a_{n}\right]\left[x^{\sigma^{n+1}}\right]$
there is a small closed ball $b\in\mathcal{B}^{\sigma^{n+1}}$ such
that $vg\left(x^{\sigma^{n}}\right)$ is constant on $b$. The same is true for polynomials
over the field of fractions $M\left(a_{0},\ldots,a_{n}\right)$: a
polynomial and its nonzero scalar multiple are simultaneously constant
on a ball.

For a subset $I\subset\mathbf{Z}$ let us set $M_{I}=M\left(a_{n}\right)_{n\in I}^{h}$.
If $I\subset J$ are subsets of $\mathbf{Z}$ then we have the inclusion $M_{I} \subseteq{M_{J}}$; Henselization commutes with directed unions, so the field $M^{h}$
is the directed union of the $M_{I}$ for $I\subset\mathbf{Z}$ finite.
Now suppose that $E$ is a finite Galois extension of $M^{h}$ and
$\sigma\left(E\right)=E$; since $E$ is finite it descends to a finite
Galois extension $E'$ of $M_{I}$ for some finite $I\subset\mathbf{Z}$.
Let $\sigma$ act on $\mathbf{Z}$ by the rule $\sigma\left(n\right)=n+1$.
For $N$ a natural number large enough we have $\sigma^{N}\left(I\right)\cap I=\varnothing$;
put $J=\sigma^{N}\left(I\right)$ and $E''=\sigma^{N}\left(E'\right)$.

Now $M$ is a regular extension of both $M_{I}$ and $M_{J}$, using
Lemma \ref{lem:Let--be} and Remark \ref{linearly-disjoint-generic}.
Since $I\cap J=\varnothing$ and $K$ is algebraically closed, the
fields $M_{I}$ and $M_{J}$ are linearly disjoint over $K$. Since
$\sigma^{N}\left(E\right)=E$ we find that $E=E'\otimes_{M_{I}}M=E''\otimes_{M_{J}}M$;
so $E$ is linearly disjoint from $E$ over $M$, whence $E=M$.
\end{proof}
\begin{cor}
\label{trivially-valued}Let $k$ be an algebraically closed inversive
difference field, viewed as trivially valued and $\Gamma$ a free
finite ordered module over $\mathbf{Z}\left[\sigma^{\pm1},p^{\pm1}\right]$;
then $k\left(t^{\Gamma}\right)$ has no nontrivial h-finite $\sigma$-invariant
extensions.
\end{cor}

\begin{proof}
Let $K=k\left(t^{\Gamma}\right)$ and choose $t_{1},\ldots,t_{n}\in K$
such that $vt_{1},\ldots,vt_{n}$ is a basis of $\Gamma$. Let
$k_{0}=k$ and $k_{m}=k\left(t_{1},\ldots,t_{m}\right)$; then $k_{m}$
is generic over $k_{m-1}$, thus by induction $K=k_{n}$ has no nontrivial
$h$-finite $\sigma$-invariant extensions.
\end{proof}
\begin{rem}
\label{generic-ultrapower}Let $K$ be a model of $\VFA$ which
is algebraically closed and transformally Henselian. Suppose moreover
that the valuation group is a model of $\widetilde{\omega\OGA}$ and
the generic type of $\ACFA$ \footnote{i.e the type of a transformally transcendental element; this requirement holds unless the field is twisted periodic, i.e we have $\sigma^n = \phi^m$ for some $n \neq 0$} is finitely satisfiable over the residue
field. Let $a$ be generic over $K$. Then $K\left(a\right)_{\sigma}$
embeds over $K$ in an ultrapower. By \ref{generic-closed-open} and
Corollary \ref{transformalHenselizationalg} there is up to isomorphism
a unique algebraically closed and transformally Henselian hull $L$
of $K\left(a\right)_{\sigma}$; since $K$ is itself algebraically
closed and transformally Henselian we find that $L$ embeds over $K$
in an ultrapower of $K$ as well.
\end{rem}

\begin{cor}
\label{cor-generic-complete-type}
Let $K$ be a model of $\VFA$ which is algebraically closed. Then the quantifier free type of a generic element $x$ over $K$ in the sense of Definition \ref{generic-definition} is complete (namely in the third case there is a unique quantifier free type of an element in the intersection, and in the second case, the quantifier free type of $vx$ over $\Gamma$ determines the quantifier free type of $x$).

Furthermore, the field $K\left(x\right)_{\sigma}$ admits no nontrivial $h$-finite $\sigma$-invariant separable extensions, and hence it admits a unique algebraic closure in the category of models of $\VFA$, up to isomorphism.
\end{cor}
\begin{proof}
The uniqueness of the quantifier free type follows from the discussion of Remark \ref{rem-qf-generic}, Remark \ref{rem-qf-complete-valgp-adjunction}, and Lemma \ref{generic-qf}. The lack $\sigma$-invariant Galois extensions of the Henselian hull is proved in Proposition \ref{generic-closed-open}.
\end{proof}

\subsection{Scatteredness}

\begin{thm}
\label{Scatteredness}Let $K$ be a model of $\VFA$ and $fx\in K\left[x\right]_{\sigma}$
a nonzero difference polynomial.

(1) The set of roots of $fx$ in $K$ is scattered in the sense that
the valuative distance between any two of them can only assume finitely
many values.

(2) Let $\mathcal{B}$ be a nested family of closed balls and assume
that $fx$ admits a $K$-root in each of them. Then the intersection
of the nest admits a $K$-point; indeed, it even admits a $K$-root
of $fx$.
\end{thm}

\begin{proof}
(1) Let $\widetilde{K}$ be the relative transformal algebraic closure
of $K$ in an ultrapower of $K$. Then $\widetilde{K}$ is the directed union
of difference field extensions of $K$ of finite transcendence degree,
and it follows that $\widetilde{\Gamma}$ is the directed union of
ordered modules $\widetilde{\Gamma_{i}}$ over $\Gamma$ such that
$\nicefrac{\widetilde{\Gamma_{i}}}{\Gamma}$ is of finite rational
rank as an abstract abelian group. Since $\widetilde{\Gamma}$ embeds
over $\Gamma$ in an ultrapower of $\Gamma$ and since $\Gamma$ is torsion free
as a module over $\mathbf{Z}\left[\sigma\right]$, we find that $\Gamma=\widetilde{\Gamma}$.

Let $C\subset\Gamma$ be the set of those $\gamma$ with the property
that $v\left(a-b\right)=\gamma$ for a pair of roots $a,b$ of $fx$
in $K$ is definable with parameters. Then $C$ is definable and yet by the last paragraph
remains bounded in cardinality in an ultrapower of $K$; by compactness,
the set $C$ is finite.

(2) For simplicity of notation assume that the family $\mathcal{B}=\left(b_{n}\right)_{n=1}^{\infty}$
is countable. Pick a root $a_{n}$ of $fx$ in the closed ball $b_{n}$.
By (1) the quantity $v\left(a_{n}-a_{m}\right)$ can only assume finitely
many values for $n\neq m$; by passing to a subsequence we may assume
that it is outright constant, say equal to $\lambda$. The closed
ball of valuative radius $\lambda$ around some (any) of the $a_{n}$
is then contained in the intersection and contains a root of $fx$.
\end{proof}

\begin{defn}
Let $K$ be a model of $\VFA$. We say that $K$ is \textit{transformally algebraically maximal} if it admits no immediate transformally algebraic extensions: that is, whenever $\widetilde{K}$ is a model of $\VFA$ transformally algebraic over $K$ with $\widetilde{k} = k$ and $\widetilde{\Gamma} = \Gamma$, then we have $\widetilde{K} = K$.
\end{defn}

\begin{lem}
\label{talgmax-transitive}
Let $\left(K_{20},K_{21},K_{10}\right)$ be a triplet of models of
$\VFA$; then $K_{20}$ is transformally algebraically maximal
if and only if $K_{21}$ and $K_{10}$ are so.
\end{lem}

\begin{proof}
Let us assume that $K_{21}$ and $K_{10}$ are transformally algebraically
maximal. Let $L_{20}$ be an immediate transformally algebraic extension
of $K_{20}$. Then $L_{10}$ is an immediate transformally algebraic
extension of $K_{10}$, hence $L_{10}=K_{10}$. Thus $L_{21}$ is
an immediate transformally algebraic extension of $K_{21}$ which
gives $L_{21}=K_{21}$.

Let us assume conversely that $K_{20}$ is transformally algebraically
maximal. We will prove that $K_{21}$ and $K_{10}$ are transformally
algebraically maximal. Note $K_{21}$ and $K_{20}$ have the same
underlying difference field. Moreover, an immediate transformally
algebraic extension of $K_{21}$ can be regarded as an immediate transformally
algebraic extension of $K_{20}$. Thus if $K_{20}$ is transformally
algebraically maximal, then $K_{21}$ is transformally algebraically
maximal. Now let $L_{10}$ be an immediate transformally algebraic
extension of $K_{10}$. Let $L_{21}$ be the universal transformally
Henselian extension of $K_{21}$ reproducing the embedding $K_{1}\hookrightarrow L_{1}$
at the level of residue fields. Let $L_{20}$ be obtained by gluing
$L_{21}$ and $L_{10}$. Then by Theorem \ref{transformalhenselization}
the valuation groups of $L_{21}$ and $K_{21}$ coincide. By assumption
the valuation groups of $K_{10}$ and $L_{10}$ coincide; so the valuation
groups of $L_{20}$ and $K_{20}$ coincide. Moreover, there is no
residue field adjunction, since the residue field of $L_{20}$ identifies
with the residue field of $L_{10}$, which is in turn the same as
the residue field of $K_{10}$ and the same as the residue field of
$K_{20}$; so $L_{20}$ is an immediate transformally algebraic extension
of $K_{20}$. Since $K_{20}$ is transformally algebraically maximal
we have $K_{20}=L_{20}$ and hence $K_{10}=L_{10}$.
\end{proof}

The following characterization of the class of transformally algebraically maximal models of $\VFA$ will often be used without further mention:

\begin{lem}
\label{lem-talg-max-characterization}
Let $K$ be a model of $\VFA$. Then for $K$ to be transformally algebraically maximal, it is necessary and sufficient that the following condition is satisfied. Let $\left(b_\alpha\right)$ be a nested family of closed $K$-balls and $fx$ a nonzero difference polynomial over $K$. If the quantity $vfb_{\alpha}$ is strictly increasing in $\alpha$, then the intersection of the family admits a $K$-point.
\end{lem}
\begin{proof}
First assume that there is a difference polynomial $fx$ and a nested family of closed balls $\left(b_{\alpha}\right)$ as in the statement. We may choose $fx$ so that whenever $gx$ is a difference polynomial of strictly smaller degree, then the quantity $vgb_{\alpha}$ is constant for $\alpha$ sufficiently large. Using Proposition \ref{rootinaball} we find an extension $\widetilde{K}$ of $K$ and a $\widetilde{K}$-root $a$ of $fx$ in the intersection. One then checks as in \cite{azgin2010valued}, Lemma 5.3 that $K\left(a\right)_{\sigma}$ is an immediate extension of $K$ .

Assume conversely that $K$ admits an immediate transformally algebraic extension $\widetilde{K}$ and fix an element $a \in \widetilde{K}$ lying outside $K$. Since $\widetilde{K}$ is an immediate extension of $K$, the nested family $\left(b_\alpha\right)$ of closed $K$-balls containing $a$ has no $K$-rational points in the intersection. Moreover, since $\widetilde{K}$ is transformally algebraic over $K$, we have $fa=0$ for some difference polynomial; by Lemma \ref{lem-transformal-herbarnd-concave} the quantity $vfb_{\alpha}$ is strictly increasing in $\alpha$, contrary to our assumption.
\end{proof}

\textbf{Tropical Roots and the Newton Polygon.} Let $K$ be a model of $\VFA$ and let $fx = a_{\nu}x^{\nu} + \ldots + a_0$ be a difference polynomial over $K$. For each $i$ let us put $va_i = \beta_i$.

\begin{defn}
\label{defn-tropical-root}
Notation as above. We say that $\gamma \in \Gamma \otimes \mathbf{Q}\left(\sigma\right)$ is a \textit{tropical root} of $fx$ if there are distinct $i \neq j \in \mathbf{N}\left[\sigma\right]$ such that, for all $k \in \mathbf{N}\left[\sigma\right]$, we have that $\beta_i + i\gamma = \beta_j + j\gamma \leq \beta_k + k\gamma$. If the constant term of $fx$ is zero then we formally take $\infty$ to be a tropical root as well.
\end{defn}

The Newton polygon gives a geometric interpretation of the set of tropical roots as we now explain. For this, let us tacitly assume that $a_{\nu}a_0$ is nonzero and that $\nu \geq 1$. The \textit{Newton polygon} of $fx$ is defined in analogy with the usual Newton polygon of an algebraic polynomial over a valued field. Namely it is the lower boundary of the convex hull of the (finite) set of points $p_{\mu}$ of the form $\left(\beta_{\mu}, \mu\right)$ where $\beta_{\mu} = va_{\mu}$ with $a_{\mu} \neq 0$ inside the space $\Gamma \otimes \mathbf{Q}\left(\sigma\right) \times {\mathbf{N}\left[\sigma\right]}$.

The \textit{slopes} of the Newton polygon are defined and calculated in analogy with the ordinary settings of valued fields, and take values in the ordered group $\Gamma \otimes \mathbf{Q}\left(\sigma\right)$. Unwinding definitions we see that $\gamma$ is a tropical root of $fx$ precisely in the event that $-\gamma$ occurs as a slope in the Newton polygon of $fx$.

The following is an immediate consequence of the definitions and the ultrametric inequality:

\begin{lem}
\label{lem-Newton-Polygon-roots}
Let $K$ be a model of $\VFA$ and let $fx$ be a nonconstant difference polynomial over $K$ with a nonzero constant term. Let $c\in K$ be a root of $f$; then $vc$ is a tropical root of $fx$.
\end{lem}

This motivates the following definition:
\begin{defn}
\label{defn-t-newtonian}
Let $K$ be a model of $\VFA$.

We say that $K$ is \textit{transformally Newtonian} if whenever $fx$ is a difference polynomial and $\gamma \in \Gamma \otimes \mathbf{Q}\left(\sigma\right)$ is a tropical root of $fx$, then there is a $K$-root of $fx$ of valuation $\gamma$.
\end{defn}

Some remarks on Definition \ref{defn-t-newtonian}

\begin{rem}
\label{rem-newton-basic}
\begin{enumerate}
    \item Let $K$ be a model of $\VFA$ and $\widetilde{K}$ an extension of $K$. If $\widetilde{K}$ is transformally Newtonian and $K$ is transformally algebraically closed in $\widetilde{K}$, then $K$ is transformally Newtonian.
    \item Let $fx$ be a difference polynomial over $K$ and let $t \in K$ be nonzero, of      valuation $\alpha$. Let $T$ be the set of tropical roots of $fx$. Then the set of tropical roots of $f\left(tx\right)$ is the translate $-{\alpha} + T$. Moreover, the set of tropical roots of $f\left(x^{p^{n}\sigma^{m}}\right)$ is $p^{-n}\sigma^{-m}T$. Furthermore, the tropical roots of $fx$ over the Frobenius twist $\tau = \sigma \circ \phi^{m}$ coincide with $p^{-m}T$. Thus $K$ and its Frobenius twists are simultaneously transformally Newtonian.
    \item Let $K$ be a model of $\VFA$ which is transformally Newtonian. Then $K$ is transformally Henselian. One can deduce this from Proposition \ref{newton-max} below, but it is instructive to give a direct proof. Using (2), it is enough to show that $K$ obeys the transformal analogue of Hensel's lemma; as in the proof of Lemma \ref{newton}, we reduce to the case of a difference polynomial of the form $1 + x + t_2x^2 + \ldots + t_{\nu}x^{\mu}$ where the $t_{\mu}$ have coefficients in the maximal ideal. In this situation, the residue image of $fx$ is linear, so the Hensel axiom simply asks for an integral root. The Newton polygon of $fx$ indeed has a vertical segment, and all other slopes are positive; thus the Newton axiom guarantees the existence of an integral root.
\end{enumerate}
\end{rem}

\begin{prop}
\label{newton-max}
Let $K$ be a model of $\VFA$. Let us assume that the following conditions are satisfied:
\begin{enumerate}
    \item The field $K$ is transformally algebraically maximal
    \item Let $fx$ be a nonconstant difference polynomial over the residue field $k$ of $K$. If $fx$ has a nonzero root in some difference field extension of $k$, then there is already a nonzero root over $k$
    \item The valuation group $\Gamma = \Gamma \otimes \mathbf{Q}\left(\sigma\right)$ is transformally divisible
\end{enumerate}
Then $K$ is transformally Newtonian and itself obeys (2). Conversely, these properties characterize the transformally Newtonian models of $\VFA$.
\end{prop}

\begin{proof}
We begin by observing that if $K$ is transformally Newtonian then it obeys (2). Condition (2) can also be equivalently rephrased as the requirement that if $fx$ has at least two nonzero coefficients, then it admits a $K$-root \footnote{Use the equivalence between $\ACFA$ and ultraproduct of algebraically closed Frobenius difference fields.}. In this case, $fx$ also has a tropical root $\gamma \neq \infty$, hence a nonzero root in $K$ as $K$ is transformally Newtonian.

We now turn to the proof. Let us assume that $K$ is transformally Newtonian. First we show that $K$ is transformally algebraically maximal. To see this let $\left(b_{\alpha}\right)$ be a nested family of closed balls in $K$ and $fx$ a difference polynomial with $vfb_{\alpha}$ strictly increasing; we must show that the intersection admits a point over $K$. 

By scatteredness (Theorem \ref{Scatteredness}) it is enough to show that $fx$ admits a $K$-root in each of the balls $b_{\alpha}$. So fix some $\alpha_0$; normalizing, we may take $b_{\alpha_0} = \mathcal{O}$. By Proposition \ref{rootinaball} there is a transformally algebraic extension $\widetilde{K}$ of $K$ and a $\widetilde{K}$-root of $fx$ in $\widetilde{\mathcal{O}}$; using Lemma \ref{lem-Newton-Polygon-roots} this root is of nonnegative valuation $\gamma$. Since $K$ is transformally Newtonian, we can find a $K$-root of $fx$ of valuation $\gamma$, which must then lie in $\mathcal{O}$.

By considering difference polynomials of the form $x^{\nu} - t$ for $t \in K$ nonzero and $0 \neq \nu \in \mathbf{N}\left[\sigma\right]$ we find that $\Gamma = \Gamma \otimes \mathbf{Q}\left(\sigma\right)$ is transformally divisible. Since $K$ is transformally algebraically maximal, it is in particular transformally Henselian; thus $k$ identifies isomorphically with a difference subfield of $K$, relatively transformally algebraically closed in $K$. Since $K$ obeys (2), the same is true of $k$.

Assume conversely that (1), (2) and (3) hold; we show that $K$ is transformally Newtonian. Let $fx$ be a difference polynomial and $\gamma$ a tropical root; we seek a root of $fx$ of valuation $\gamma$. By a change of variables of the form $x=ty$, using the fact that $\Gamma$ is transformally divisible, we may assume that $\gamma = 0$. Multiplying all coefficients of $fx$ by a nonzero scalar won't change the set of tropical roots; so assume that $fx$ has coefficients in $\mathcal{O}$ with at least one of valuation zero. In this situation, there will in fact be at least \textit{two} coefficients of valuation zero, since the Newton polygon admits a vertical line segment. Using our assumption on $k$ we can find $a \in \mathcal{O}$ with $vfa >0$ and $va=0$. Since $K$ is transformally algebraically maximal, arguing as in the proof of Proposition \ref{rootinaball}, there is a root $c$ of $fx$ in $K$ whose residue class coincides with that of $a$; in particular $vc = 0$. Thus $K$ is transformally Newtonian.

\end{proof}
\begin{cor}
Let $K$ be a model of $\VFA$. Then there is a model $\widetilde{K}$ of $\VFA$ over $K$ which is transformally Newtonian and which is transformally algebraic over $K$.
\end{cor}
\begin{proof}
It is enough to find an arbitrary transformally Newtonian extension of $K$, since the relative transformal algebraic closure of $K$ inside such an extension is transformally Newtonian.
As in the proof of Proposotion \ref{rootinaball}, we may assume that $K$ is strictly transformally Henselian, spherically complete, and obeys $\Gamma = \Gamma \otimes \mathbf{Q}\left(\sigma\right)$, in which case Proposition \ref{newton-max} applies.
\end{proof}
\begin{cor}
\label{newton-generic}Let $K$ be a model of $\VFA$ which
is transformally Newtonian. Let $L$ be a transformally Henselian
model of $\VFA$ over $K$. Let us assume that no element of
$l$ is transformally algebraic over $k$ unless it already lies in
$k$; then every element of $L$ is generic over $K$. Moreover, assume
$K$ is nontrivially valued and $L$ is generated over $K$ as an
algebraically closed and transformally Henselian field by a single
element; then $L$ embeds over $K$ in an ultrapower of $K$.
\end{cor}

\begin{proof}
The field $K$ is algebraically closed as it is transformally Newtonian.
Pick $a\in L$. After an affine change of variables either $va=0$
and the residue class of $a$ lies outside $k$, or else $va$ lies
outside $\Gamma$, or else $a$ lies in an infinite intersection of
closed $K$-balls whose intersection does not admit any $K$-points.
In the first case, by Lemma \ref{transformally-transcendental-hens},
the element $a$ is generic in $\mathcal{O}$ over $K$; in the second
case, by Corollary \ref{newton-max} we have $\Gamma=\Gamma\otimes\mathbf{Q}\left(\sigma\right)$
so again the element $a$ is generic. We may therefore assume that
$a$ lies in an infinite intersection of closed $K$-balls whose intersection
does not admit any $K$-points; let $\left(b_{\alpha}\right)$ be
the family of all closed $K$-balls containing $a$. If $fx\in K\left[x\right]_{\sigma}$
is a difference polynomial then the quantity $vfb_{\alpha}$ must
eventually stabilize, for otherwise there would be a $K$-root of
$f$ in each ball by the Newton axiom, hence in the intersection using
scatteredness. It follows that $a$ is generic over $K$. For the
``moreover'' part use Remark \ref{generic-ultrapower}.
\end{proof}

\begin{rem}
\label{rem-ec-talg}
Heuristically, Corollary \ref{newton-generic} implies that if $K$ is a nontrivially valued model of $\VFA$, then for $K$ to existentially closed it is sufficient that $K$ is existentially closed for transformally algebraic extensions. This will be made precise in Section \ref{sec:The-Model-Companion}.
\end{rem}
\begin{lem}
The class of transformally Newtonian models of $\VFA$ is transitive
in triplets.
\end{lem}

\begin{proof}
Let us assume that $K_{20}$ is transformally Newtonian; we will prove that $K_{21}$ and $K_{10}$ are transformally Newtonian. The converse implications follows by inception of the argument in reverse. We use the criterion of Proposition \ref{newton-max}.
We have a short exact sequence $0\to\Gamma_{10}\to\Gamma_{20}\to\Gamma_{21}\to0$; using the fact that $\Gamma_{20}$ is transformally divisible we find that the same is true of $\Gamma_{21}$ and $\Gamma_{10}$. By Lemma \ref{talgmax-transitive} transitivity holds for the class of transformally algebraically maximal models, so $K_{21}$ and $K_{10}$ are transformally algebraically maximal. The residue field of $K_{10}$ coincides with the residue field $K_0$ of $K_{20}$, so it is transformally Newtonian; in particular it has enough roots of difference polynomials. Thus $K_{21}$ is transformally Newtonian too, using Proposition \ref{newton-max}.
\end{proof}

\begin{lem}
\label{t-Newtonian-spherically complete}
Let $K$ be a model of $\VFA$ which is transformally Newtonian. Then there exists a spherically complete immediate extension of $K$ which embeds over $K$ in an elementary extension.
\end{lem}
\begin{proof}
Let $L$ be a sufficiently saturated elementary extension of $K$. Let $E$ be a model of $\VFA$ over $K$ inside $L$ maximal with respect to the property that it is an immediate extension of $K$. Then $E$ is algebraically closed, and we claim that $E$ is spherically complete. For suppose that $\left(b_{\alpha}\right)$ be a nested family of closed balls in $E$ whose intersection admits no points over $E$. First assume that there is some difference polynomial $f$ with coefficients in $E$ such that the quantity $vfb_{\alpha}$ is strictly increasing in $\alpha$; take $f$ of minimal possible degree. Since $L$ is transformally Newtonian, there is an $L$-root of $f$ in each of the balls $b_{\alpha}$; by Theorem \ref{Scatteredness} there is an $L$-root of $f$ in the intersection. Arguing as in \cite{azgin2010valued} or \cite{kaplansky1942maximal} one shows that for $a$ a root of $f$ in the intersection, the extension $E\left(a\right)_{\sigma}$ of $E$ inside $L$ is immediate. By maximality of $E$ there can thus be no such difference polynomial $f$. So the intersection is generic over $E$; by saturation of $L$ (using the fact that the cardinality of $E$ is bounded by $2^{|K|}$), there is an element $a$ in the intersection; adjoining it gives again an immediate extension, contradicting maximality. Thus the field $E$ is spherically complete, as required.
\end{proof}
\newpage{}

\section{Rationality of the Cut\label{sec:Wild-Ramification}}

The purpose of this section is to prove the following Proposition:
\begin{prop}
\label{additivediff}Let $K$ be a model of $\VFA$ which is
algebraically closed, transformally Henselian, and transformally Archimedean.
Suppose further that $\Gamma=\Gamma\otimes\mathbf{Q}\left(\sigma\right)$
is transformally divisible. Let us assume that $K$ admits an immediate
transformally algebraic extension. Then there is an additive difference
polynomial $\tau x\in\mathcal{O}_{K}^{\times}\left[x\right]_{\sigma}$
with all coefficients of valuation zero and some $c\in K$ with $vc<0$
such that the following holds:

(1) For every $\gamma<0$ there is a unique closed ball $b_{\gamma}$
of valuative radius $\gamma$ in $K$ with the property that the transformal
Herbrand function of $\tau x-c$ above $b_{\gamma}$ is strictly increasing

(2) We have $v\left(\tau b_{\gamma}-c\right)\to0$ from below as $\gamma\to0$
from below.

(3) The equation $\tau x-c=0$ has no root in $K$, and the intersection
of the $b_{\gamma}$ has no $K$-points.

After twisting we may take $\tau'\neq0$. Moreover there is an immediate
transformally algebraic extension generated by the adjunction of a
root of $\tau x-c$ inside the intersection of the closed balls $b_{\gamma}$.
\end{prop}

Here, by an additive difference polynomial we mean a linear combination
of difference monomials of the form $x^{\sigma^{n}p^{m}}$ for $n,m\in\mathbf{N}$,
not both zero.

The following Corollary is an immediate consequence of Proposition \ref{additivediff}:

\begin{cor}
\label{cor-criterion-talg-max}
Let $K$ be a model of $\VFA$ which is algebraically closed, transformally Henselian, and transformally Archimedean. Suppose moreover that $\Gamma = \Gamma \otimes \mathbf{Q}\left(\sigma\right)$ is transformally divisible. Let us assume that all nonconstant additive operators with coefficients in $K$ are onto on $K$-points; then $K$ is transformally algebraically maximal.
\end{cor}

The proof of Proposition \ref{additivediff} also shows:
\begin{cor}
\label{cor:rat-cut}(Rationality of the Cut) Let $K$ be a model of
$\VFA$ which is algebraically closed, transformally Henselian,
and transformally Archimedean. Let us assume that $\Gamma=\Gamma\otimes\mathbf{Q}\left(\sigma\right)$
is transformally divisible. Let $L$ be an immediate transformally
algebraic extension of $K$ and let $a\in L$ be an element not in
$K$. Then the cut
\[
C\left(a\right)=\left\{ v\left(a-b\right)\colon b\in K\right\} 
\]
is rational, i.e, it is of the form $\left(-\infty,\gamma\right)$
for some $\gamma\in\Gamma$.
\end{cor}

\begin{proof}
    See Claim \ref{claim614}.
\end{proof}

\begin{rem}
There is a close relationship between rationality of cuts and additive
equations; see the computation in \ref{example-lemma}. Proposition
\ref{additivediff} is also related to an observation due to Kaplansky:
if $K$ is a complete valued field with an Archimedean value group, and $K$ admits
an immediate algebraic extension, it also admits one obtained by adjoining
the root of a $p$-polynomial.
\end{rem}

\begin{cor}
\label{no-immediate-countably}Let $K$ be a model of $\VFA$
which is algebraically closed and transformally Archimedean. Let us
assume that $\Gamma=\Gamma\otimes\mathbf{Q}\left(\sigma\right)$ is
transformally divisible and that every countable nested family of
closed balls in $K$ has a nonempty intersection; then $K$ is complete
and transformally algebraically maximal.
\end{cor}

\begin{proof}
Since $K$ is algebraically closed it is algebraically Henselian.
Since $K$ is transformally Archimedean, the valuation group of $K$
is of countable cofinality. By assumption every countable nested family
of closed balls in $K$ has a nonempty intersection; so $K$ is complete.
It follows from Proposition \ref{complete-t-henselian} that $K$
is transformally Henselian. 

Now assume towards contradiction that $K$ admits a proper immediate
transformally algebraic extension. Then we are in the situation of
Proposition \ref{additivediff}; let $c\in K$ and $\tau x\in\mathcal{O}_{K}^{\times}\left[x\right]_{\sigma}$
be as in the conclusion. Moreover let us set $\widetilde{\tau}x=\tau x-c$.
Since $K$ is transformally Archimedean the family $\left(b_{\gamma}\right)_{\gamma<0}$
has a countable cofinal subsequence; so our assumption implies that
there is a $K$-point in the intersection. Moreover there is a \textit{unique}
closed ball $b_{0}$ of valuative radius zero contained in each $b_{\gamma}$,
namely it is the closed ball of valuative radius $0$ around any element
of the intersection (which is nonempty by assumption). The transformal
Herbrand function of $\widetilde{\tau}x$ above $b_{0}$ is then strictly
increasing as it is strictly increasing above each $b_{\gamma}$;
moreover we have $v\widetilde{\tau}b_{0}=0$ by continuity. So after
an affine change of variables and twisting, we may assume that $vc=0$
and that $\tau'\neq0$. Since $K$ is transformally Henselian, an
extension generated by adjunction of a root of an equation of this form is not immediate.
\end{proof}

\subsubsection{\label{digitwisedom}Digitwise Domination and Binomial Coefficients}

Let $\nu\in\mathbf{N}\left[\sigma\right]$ be fixed for the moment.
By definition we may write $\nu=\sum\nu_{i}\cdot\sigma^{i}$ for certain
uniquely defined natural numbers $\nu_{i}$ which are almost all zero.
Each $\nu_{i}$ can then further be uniquely written in the form $\nu_{i}=\sum\nu_{ij}\cdot p^{j}$
where the $\nu_{ij}$ are natural numbers in the range $0\leq\nu_{ij}\leq p-1$;
this fact is of course otherwise familiar as the uniqueness of the expansion
of a natural number to base $p$.

If $\mu,\nu\in\mathbf{N}\left[\sigma\right]$ then we say that $\mu$
is \textit{digitwise dominated} by $\nu$ if $\mu_{ij}\leq\nu_{ij}$
for all $i,j$ and that it is \textit{strictly digitwise dominated}
by $\nu$ if, in addition, the inequality is strict for at least one
choice of $i,j$. Digitwise domination is thus a partial ordering
on $\mathbf{N}\left[\sigma\right]$ which refines the usual ordering;
the minimal elements above $0$ are of the form $\mu=\sigma^{n}\cdot p^{m}$
and we call them the \textit{transformal $p$-th powers}.

Next, we discuss binomial coefficients. Fix $0<\nu\in\mathbf{N}\left[\sigma\right]$.
We may write:
\[
\left(1+x\right)^{\nu}=\sum_{\mu}{\nu \choose \mu}\cdot x^{\mu}
\]
for certain uniquely defined elements ${\nu \choose \mu}$ of the
prime field. Using the additive nature of $\sigma$ and the Frobenius
and the fact that they commute with each other, in view of the expansion
$\nu=\sum\nu_{ij}\cdot\sigma^{i}\cdot p^{j}$, we may write:
\[
\left(1+x\right)^{\nu}=\prod_{i,j}\left(1+x^{\sigma^{i}\cdot p^{j}}\right)^{\nu_{ij}}
\]

and we thus find that
\[
{\nu \choose \mu}=\prod_{i,j}{\nu_{ij} \choose \mu_{ij}}
\]
the coefficient of $x^{\mu}$ in $\left(1+x\right)^{\nu}$ is therefore
nonzero precisely in the event that $\mu$ is digitwise dominated
by $\nu$. Here, each multiplicand on the right hand side has its
usual combinatorial interpretation, modulo the prime $p$.

\begin{rem}
    Recall (Remark \ref{rem-equichar}) that the models of $\VFA$ are of equal characteristic, so the transformal binomial coefficients defined above are of valuation zero if they do not vanish (they lie in the prime field).
\end{rem}

\subsubsection{~\label{6.4.1}}

Let $K$ be a model of $\VFA$ which is as in the statement
and suppose towards contradiction that $K$ admits an immediate transformally
algebraic extension. Then one can find a sequence $\left(b_{n}\right)_{n=1}^{\infty}$
of closed balls in $K$, whose intersection does not admit any points
over $K$, together with a nonzero difference polynomial $fx\in K\left[x\right]_{\sigma}$
with the property that the quantity $vfb_{n}$ fails to stabilize
as $n\to\infty$. Let the degree of $fx$ be chosen as small as possible
so as to witness this fact; since $K$ is algebraically closed, the
difference polynomial $fx$ cannot possibly be algebraic.

Introduce the following notation:
\begin{itemize}
\item Let $I\subset\mathbf{N}\left[\sigma\right]$ be the finite set consisting
of those $0<\mu\in\mathbf{N}\left[\sigma\right]$ with the property
that the formal transformal derivative $f_{\mu}$ is nonzero.
\item For every $0<\mu\in I$ the quantity $vf_{\mu}b_{n}$ stabilizes as
$n\to\infty$ by minimality of the degree of $f$; write $\beta_{\mu}$
for the eventual, common value.
\item Let the valuative radius of the closed ball $b_{n}$ be denoted by
$\gamma_{n}$.
\item The relation $\mu\prec\nu$ on the set $I$ defined by setting $\mu\prec\nu$
precisely in the event that $\beta_{\mu}+\mu\cdot\gamma_{n}<\beta_{\nu}+\nu\cdot\gamma_{n}$
for all $n\gg0$ is a total order on $I$; write $\rho$ for the minimal
element, and say that $\rho$ is the \textit{dominant exponent}.

\end{itemize}

\subsubsection{~}
\begin{lem}
\label{digitwise}The relation $\prec$ on the set $I$ refines digitwise
domination. That is, if $\mu$ is strictly digitwise dominated by
$\nu$, then we have $\beta_{\mu}+\mu\cdot\gamma_{n}<\beta_{\nu}+\nu\cdot\gamma_{n}$
for $n\gg0$. In particular, the dominant exponent is a transformal
$p$-th power.
\end{lem}

\begin{proof}
Let $a,c$ be generic in the closed ball $b_{n}$, so that in particular they are at  at valuative distance $\gamma_{n}$ from each other. For fixed
$0<\mu\in\mathbf{N}\left[\sigma\right]$, the Taylor expansion reads:
\[
f_{\mu}a-f_{\mu}c=\sum{\nu \choose \mu}\cdot f_{\nu}c\cdot\left(a-c\right)^{\nu-\mu}
\]
the summation taken over all those $\nu$ strictly digitwise dominating
$\mu$, per the analysis of the binomial coefficients. The valuation
of the summand corresponding to $\nu$ on the right is equal to $\beta_{\nu}+\gamma_{n}\cdot\left(\nu-\mu\right)$;
since $a$ was chosen generic relative to $c$ we find that the valuation
of the right hand side is equal to the minimal amongst these. This
must also be the valuation of the left hand side, which is, in turn,
the difference between a pair of elements of valuation $\beta_{\mu}$;
its valuation is therefore at least that. Combining these two facts
gives the inequality:
\[
\beta_{\mu}\leq\beta_{\nu}+\left(\nu-\mu\right)\cdot\gamma_{n}
\]
We therefore have a weak inequality for some $n$; since the sequence
of valuative radii is strictly increasing, a strict inequality must
ultimately hold, and we conclude.
\end{proof}

\subsubsection{~\label{703}}

The dominant exponent $\rho\in I$ is characterized by the property
that $\beta_{\rho}+\rho\cdot\gamma_{n}<\beta_{\mu}+\mu\cdot\gamma_{n}$
for any other exponent $\mu\in I$, at least as soon as $n$ is sufficiently
large. Given an exponent $\mu\in I$ other than $\rho$ it is possible
that a sharper statement is true, namely that we have $\beta_{\rho}+\rho\cdot\gamma_{n}<\beta_{\mu}+\mu\cdot\gamma_{m}$
for \textit{all} $n$ simultaneously, as soon as $m$ is chosen sufficiently
large, independently of $n$. Let us then say that $\mu$ is \textit{asymptotically
negligible}; otherwise, say that it is \textit{asymptotically dominant}.

\subsubsection{~}

For each $n$, let an element $a_{n}$ in the closed ball $b_{n}$
be chosen. Let also $N$ be so large so as to exceed each of the natural
numbers whose existence is hinted of in \ref{6.4.1} and in \ref{703}.
The Taylor expansion reads:
\[
fa_{n}=fa_{N}+\sum_{\mu\in I}f_{\mu}a_{N}\cdot\left(a_{n}-a_{N}\right)^{\mu}
\]
We have $vfa_{n}=\beta_{\rho}+\rho\cdot\gamma_{n}$, the quantity
$\rho\in I$ being here the dominant exponent; this must also be the
valuation of the right hand side. On the other hand, the valuation
of the summand corresponding to $\mu$ on the right is equal to $\beta_{\mu}+\mu\cdot\gamma_{N}$.
If $\mu$ is asymptotically negligible, then by our choice of $N$
and the definition of $\rho$ we have $\beta_{\rho}+\rho\cdot\gamma_{n}<\beta_{\mu}+\mu\cdot\gamma_{N}$
regardless of the choice of $n$. From the right hand side, the summand
corresponding to $\mu$ can be ejected, the valuation of the left
hand side remaining otherwise unaltered. In other words, letting

\begin{equation}
gx=fa_{N}+\sum_{\mu\in I'}f_{\mu}a_{N}\cdot\left(x-a_{N}\right)^{\mu}\label{truncateddiff}
\end{equation}
be the displayed truncation, with the summation taken only over the
subset $I'\subset I$ consisting of asymptotically dominant exponents,
we find that $vga_{n}=vfa_{n}$ for all $n>N$, so that $vga_{n}$
fails to stabilize as $n\to\infty$. Now we assert that the difference
polynomial $gx$ is, in fact, precisely of the form promised by the
statement of Proposition \ref{additivediff}, up to a scalar multiple.
This is a consequence of the Claims below.
\begin{claim}
Notation as above. Then the dominant exponent $\rho$ is \textit{not}
equal to $1$.
\end{claim}

\begin{proof}
Suppose to the contrary that $\rho=1$. We have $vfa_{n}=f_{1}a_{n}+\gamma_{n}$
for $n\gg0$. By transformal Henselianity and Lemma \ref{newton}
an element $t_{n}\in\mathcal{O}_{K}$ is found with $v\left(a_{n}-t_{n}\right)=\gamma_{n}$
and $ft_{n}=0$; in fact, these constraints uniquely determine $t_{n}$.
There is therefore a root of $fx$ in cofinally many of the balls
- this contradicts scatteredness!
\end{proof}
\begin{claim}
\label{transformal-pth-powers}Let $\mu_{0},\mu_{1}$ be the smallest
and largest transformal $p$-th powers in $I'$, respectively. Then
after twisting we have $\mu_{0}=1$ whereas $\mu_{1}$ is an infinite
transformal $p$-th power. In particular, they are distinct.
\end{claim}

\begin{proof}
First off, it is impossible that every exponent in $I'$ has a vanishing
constant term. For, using the fact that $K$ is inversive, we may
write:
\[
ga_{n}=\sum c_{\mu}a_{n}^{\mu}=\left(\sum c_{\mu}^{\frac{1}{\sigma}}a_{n}^{\frac{\mu}{\sigma}}\right)^{\sigma}
\]
so that the difference polynomial $hx=\sum c_{\mu}^{\frac{1}{\sigma}}x^{\frac{\mu}{\sigma}}$
has strictly smaller degree, and yet the quantity $vha_{n}$ fails
to stabilize as $n\to\infty$, contradicting the minimality of the degree of $fx$ (and hence of $gx$). It follows
that some exponent in $I'$ has a nonzero constant term, so it is
digitwise dominated by some natural number, namely the largest power
of $p$ dividing that term. By Lemma \ref{digitwise}, the set $I'$
is downwards closed with respect to digitwise domination, and it follows
that there is some $\mu_{0}\in I'$ which is an honest power of $p$.
By the same argument, after twisting, we may assume that $\mu_{0}=1$.

By assumption, the field $K$ is algebraically closed, so the difference
polynomial $gx$ cannot possibly be algebraic: there is some exponent
$\nu\in I'$ which is at least as large as $\sigma$, with respect
to the usual ordering on $\mathbf{N}\left[\sigma\right]$. It digitwise
dominates some infinite transformal $p$-th power, and this transformal
$p$-th power will do; here, again, we used the fact that $I'$ is
closed downwards with respect to digitwise domination.
\end{proof}
\begin{claim}
\label{claim614} Let $\gamma\in\Gamma$ be the unique element satisfying
$\beta_{\mu_{0}}+\mu_{0}\cdot\gamma=\beta_{\mu_{1}}+\mu_{1}\cdot\gamma$,
where $\mu_{0},\mu_{1}$ are given by Claim \ref{transformal-pth-powers}.
Then $\gamma_{n}\to\gamma$ as $n\to\infty$. Furthermore, every element
of $I'$ is a transformal $p$-th power.
\end{claim}

\begin{proof}
Let $X$ be the downwards closure of the sequence $\left(\gamma_{m}\right)_{m=0}^{\infty}$.
The fact that $\mu_{0},\mu_{1}$ are asymptotically dominant implies
that we have the equality $\beta_{\mu_{0}}+\mu_{0}\cdot X=\beta_{\mu_{1}}+\mu_{1}\cdot X$.
By Lemma \ref{t-Arch-cuts-rational} we have either $X=\Gamma$ or
$X=\left(-\infty,\gamma\right)$ or else $X=\left(-\infty,\gamma\right]$.
The third case is ruled out since $X$ has no maximal element. The
first case is ruled out since $K$ is transformally Henselian and
hence transformally algebraically closed in its completion (\ref{complete-t-henselian}).

The Claim now easily follows. Let $\gamma$ be the limit of the sequence
$\left(\gamma_{m}\right)_{m=0}^{\infty}$. Suppose that $\nu\in I$
is not a transformal $p$-th power, and let $\mu\in I$ be a transformal
$p$-th power strictly digitwise dominated by $\nu$. By Lemma \ref{digitwise}
we have $\beta_{\mu}+\mu\cdot\gamma_{m}<\beta_{\nu}+\nu\cdot\gamma_{m}$
for $m\gg0$. Since $\mu<\nu$, the right hand side increases faster;
hence the inequality $\beta_{\mu}+\mu\cdot\gamma<\beta_{\nu}+\nu\cdot\gamma$
is all the more true. By continuity, we have $\beta_{\mu}+\mu\cdot\gamma<\beta_{\nu}+\nu\cdot\gamma_{m}$
for $m\gg0$, so a fortiori $\beta_{\mu}+\mu\cdot\gamma_{n}<\beta_{\nu}+\nu\cdot\gamma_{m}$
for all $n$, which implies that $\nu$ is asymptotically negligible.

This shows that every member of $I'$ is a transformal $p$-th power,
and the fact that $\gamma$ is as promised follows by the same argument.
\end{proof}
By Claim \ref{claim614}, the difference polynomial $gx$ given by
\ref{truncateddiff} is an additive difference polynomial. Since $\Gamma=\Gamma\otimes\mathbf{Q}\left(\sigma\right)$,
replacing each $a_{n}$ by a scalar multiple, we may assume that $\gamma_{n}\to0$
as $n\to\infty$. The difference polynomial $gx$ is then precisely
as promised by the statement of \ref{additivediff}. 

\newpage{}

\section{\label{sec:Transformally-Archimedean-Ultrap}Transformally Archimedean
Ultrapowers}

The transformally Archimedean ultrapower, defined on the category
of transformally Archimedean ordered modules as in \ref{UltrapowerOGA},
can be lifted to the category of transformally Archimedean valued
difference fields, as follows. Fix a nonprincipal ultrafilter $\mathcal{F}$
on $\mathbf{N}$. Let $K$ be a model of $\VFA$ which is transformally
Archimedean and let $K^{\star}$ be the ultrapower of $K$ with respect
to $\mathcal{F}$. To the embedding of $K$ in $K^{\star}$, a pair
of transformally prime ideals of $\mathcal{O}^{\star}$ are attached
in an invariant manner:
\begin{itemize}
\item The collection of all those transformally prime ideals $I$ of $\mathcal{O}^{\star}$
with $I\cap\mathcal{O}=0$ contains the zero ideal and closed downwards
with respect to inclusion; let $\mathfrak{p}$ be the union.
\item Dually, the collection of all transformally prime ideals $J$ of $\mathcal{O}^{\star}$
with $J\cap\mathcal{O}=\mathcal{M}$ contains $\mathcal{M}^{\star}$
and is closed upwards with respect to inclusion; let $\mathfrak{q}$
be the intersection.
\end{itemize}
Let $\mathcal{O}^{\mathcal{F}}$ be obtained from $\mathcal{O}^{\star}$
by localizing away from $\mathfrak{q}$ and factoring out $\mathfrak{p}$.
Then the map $\mathcal{O}\to\mathcal{O}^{\mathcal{F}}$ is an injective
local homomorphism of difference rings; if $K^{\mathcal{F}}$ is the
field of fractions of $\mathcal{O}^{\mathcal{F}}$, then $K^{\mathcal{F}}$
is a model of $\VFA$ over $K$. Using the correspondence between
transformally prime ideals of $\mathcal{O}^{\star}$ and convex transformal
submodules of $\Gamma^{\star}$, one checks that the valuation group
of $\mathcal{O}^{\mathcal{F}}$ is the transformally Archimedean ultrapower
$\Gamma^{\mathcal{F}}$ of $\Gamma$. The fact that the process of
killing infinitesimals and passing to convex hulls commute is reflected
here at the level of rings by the fact that localizations and quotients
do. 

We therefore arrive at:
\begin{defn}
Let $K$ be a model of $\VFA$ which is transformally Archimedean.
The \textit{transformally Archimedean ultrapower} of $K$ with respect
to $\mathcal{F}$ is the field $K^{\mathcal{F}}$over $K$ defined
as above.
\end{defn}

\begin{lem}
\label{countably-sph-complete}Let $\mathcal{C}$ be the class of
valued fields with the following property: every countable nested
family of closed balls has a nonempty intersection. Then for a triplet $\left(K_{20}, K_{21}, K_{10}\right)$ of valued fields we have $K_{20} \in \mathcal{C} \Rightarrow K_{21}, K_{10} \in \mathcal{C}$
\end{lem}

\begin{rem}
    In fact, the class $\mathcal{C}$ is transitive in triplets, but we will not need this.
\end{rem}

\begin{proof}
Let $\left(K_{20},K_{21},K_{10}\right)$ be a triplet. We will prove
that $K_{20}\in\mathcal{C}\Rightarrow K_{21},K_{10}\in\mathcal{C}$. Every
ball $b$ contained in $\mathcal{O}_{10}$ is the homomorphic image
of a closed ball $b'$ in $\mathcal{O}_{20}$ under the residue map;
the valuative radii can moreover be taken equal if $\Gamma_{10}$
is identified with a subgroup of $\Gamma_{20}$. So $K_{20}\in\mathcal{C}\Rightarrow K_{10}\in\mathcal{C}$.
Similarly there is a map from closed balls $b$ in $K_{20}$ to closed
balls $b'$ in $K_{21}$: the center of $b'$ is the center of $b$
and the valuative radius $\gamma'$ of $b'$ is the image of $\gamma\in\Gamma_{20}$
in $\Gamma_{21}$. This preserves inclusions so $K_{21}\in\mathcal{C}$
provided that $K_{20}\in\mathcal{C}$.
\end{proof}
\begin{lem}
\label{EC-TA-Ultrapower}Let $K$ be a model of $\VFA$ which
is algebraically closed and transformally Archimedean and $K^{\mathcal{F}}$
the transformally Archimedean ultrapower of $K$.

(1) The field $K^{\mathcal{F}}$ is algebraically closed and transformally
Archimedean. Moreover it obeys $\Gamma^{\mathcal{F}}=\Gamma^{\mathcal{F}}\otimes\mathbf{Q}\left(\sigma\right)$.

(2) The field $K^{\mathcal{F}}$ has no immediate transformally algebraic
extensions; in particular it is transformally Henselian. Furthermore,
the field $K^{\mathcal{F}}$ is complete.
\end{lem}

\begin{proof}
(1) The transformally Archimedean ultrapower of $K$ is the residue
field of an algebraically closed field, namely the field underlying
the ordinary ultrapower, so it is algebraically closed. The valuation
group of $K^{\mathcal{F}}$ is $\Gamma^{\mathcal{F}}$ which is transformally
divisible and transformally Archimedean by Lemma \ref{transformarch}.

(2) Every countable nested family of closed balls in $K^{\mathcal{F}}$
has a nonempty intersection. This this property holds for the ordinary
ultrapower, hence holds for $K^{\mathcal{F}}$ by unwinding the definition
of $K^{\mathcal{F}}$ and using Lemma \ref{countably-sph-complete}.
So we conclude using Corollary \ref{no-immediate-countably}.
\end{proof}
\begin{lem}
\label{transformallyarchbasechange}
Let $\nicefrac{L}{K}$ be an extension of transformally Archimedean,
nontrivially valued models of $\VFA$, and let $K^{\mathcal{F}}$ be
the transformally Archimedean ultrapower of $K$. Let us assume that
$K$ is transformally Henselian and that $L$ is transformally algebraic
over $K$. Then $L$ and $K^{\mathcal{F}}$ are linearly disjoint
over $K$ in $L^{\mathcal{F}}$; moreover $l$ and $k^{\mathcal{F}}$
are linearly disjoint over $k$ in $l^{\mathcal{F}}$.
\end{lem}

In the next lemma, the transitivity of the class of transformally Henselian models in $\VFA$ will be used without further mention (see \ref{transitivetriplets})
\begin{proof}.
We may assume that $L$ is transformally Henselian. Let $K^{\ast}$ and $L^{\ast}$ denote the ordinary ultrapowers. Then $L$ and $K^{\ast}$ are linearly disjoint over $K$ in $L^{\ast}$, since $\left(L^{\ast}, K^{\ast}\right)$ is an elementary extension of $\left(L, K\right)$ (in the language of fields with a unary predicate for a subfield).

Let $\widetilde{L}$ and $\widetilde{K}$ denote the models of $\VFA$ with difference fields $L^{\ast}$ and $K^{\ast}$ and the coarsened valuation obtained by factoring out the $\Gamma_{K}$ infinitesimals, as in the construction of the transformally Archimedean ultrapower. Then $\widetilde{L}$ and $\widetilde{K}$ are transformally Henselian and linearly disjoint over $K$, since the underlying fields are the same; using Lemma \ref{linearly-disjoint-unramified} we find that the residue fields are linearly disjoint also. Note $L$ and $K$ both inherit from $\widetilde{L}$ the trivial valuation; by unwinding the definition of the transformally Archimedean ultrapower, we find that $L$ and $K^{\mathcal{F}}$ are linearly disjoint over $K$. The "moreover" part follows from a second application of Lemma \ref{linearly-disjoint-unramified}.
\end{proof}

\newpage{}

\section{The Model Companion\label{sec:The-Model-Companion}}

\subsection{The Amalgamation Property}

The purpose of this subsection is to prove the following Theorem:
\begin{thm}
\label{TheAmalgamationProperty}(The Amalgamation Property) Let $L\hookleftarrow K\hookrightarrow M$
be models of $\VFA$.

(1) Let us assume that $K$ has no nontrivial $\text{h}$-finite $\sigma$-invariant
extensions. Then $L$ and $M$ can be jointly embedded over $K$ in
a third model of $\VFA$.

(2) Let us assume that $K$ is algebraically closed and transformally
Henselian. Then there is a model $N$ of $\VFA$ together with
embeddings $L\hookrightarrow N\hookleftarrow M$ over $K$ so as to
render $L$ and $M$ linearly disjoint over $K$ in $N$, and so as
to render $k_{L}$ and $k_{M}$ linearly disjoint over $k_{K}$ in
$k_{N}$. Similarly, let $\widetilde{\Gamma_{K}},\widetilde{\Gamma_{L}}$,
$\widetilde{\Gamma_{M}}$ and $\widetilde{\Gamma_{N}}$ denote the
transformal divisible hulls of $\Gamma_{K},\Gamma_{L},\Gamma_{M}$
respectively; then we can arrange that $\widetilde{\Gamma_{M}}\cap\widetilde{\Gamma_{L}}=\widetilde{\Gamma_{K}}$
inside $\widetilde{\Gamma_{N}}$. 
\end{thm}

In the settings of Theorem \ref{TheAmalgamationProperty}, assume all fields are algebraically closed and transformally Henselian. If $L$ is in addition transformally algebraic over $K$, then the amalgamation is in fact outright unique, subject to the linear disjointness constraints:

\begin{cor}
\label{cor-talg-unique-basechange}
Let $K$ be a model of $\VFA$ which is algebraically closed and transformally Henselian. Let $L,M$ be transformally Henselian models of $\VFA$ over $K$, and assume that $L$ is transformally algebraic over $K$. Let $N = L \otimes_{K} M$ be the displayed abstract difference field. Then $N$ can be uniquely expanded to a model of $\VFA$. More precisely, there exists a unique $\sigma$-invariant and $\omega$-increasing valuation ring of $N$ dominating $\mathcal{O}_L$ and $\mathcal{O}_M$ under the natural inclusions.
\end{cor}

\begin{proof}
Let $N = L \otimes_{K} M$ be displayed difference field; we want to prove that $N$ can be uniquely expanded to a model of $\VFA$, subject to the evident compatibility constraints.

Let us assume first that $K$ is spherically complete and that $\Gamma$ is transformally divisible. Using Corollary \ref{cor-transformally-alg-no-valgp-adjunction} and Lemma \ref{unramified-linearly-disjoint amalgamation}, the valuation groups of $K$ and $L$ coincide and that the residue fields $l$ and $m$ are linearly disjoint over $k$ in the amalgamation; it follows from spherically complete domination (see \ref{Spherically Complete Domination}) that the valued field structure on the amalgamation is determined, as required.

For the general case let $\widetilde{K}$ be a spherically complete extension of $K$ with a transformally divisible valuation group as in the prevoius paragraph. Let $\widetilde{L} = L \otimes_{K}{\widetilde{K}}$ and similarly $\widetilde{M}$. By Corollary \ref{linearly-disjoint-unramified}, the residue fields $l$ and $m$ are linearly disjoint over $k$. Using Theorem \ref{TheAmalgamationProperty}, we can jointly amalgamate $\widetilde{L}$ and $\widetilde{M}$ so as to render the residue fields $\widetilde{l}$ and $\widetilde{m}$ linearly disjoint over $\widetilde{k}$. But now the argument of the previous paragraph applies to see that the valued field structure on the amalgamation is uniquely determined.

\end{proof}

\begin{rem}
In the settings of Corollary \ref{cor-talg-unique-basechange}, the condition that $L$ be transformally algebraic over $K$ is necessary for the uniqueness of the valued field amalgamation, as simple examples involving generic elements show. Furthermore, one should stress that that even if $K$ is algebraically closed and transformally Henselian, and $L$ is an abstract transformally algebraic difference field extension of $K$, then an expansion of $L$ to a model of $\VFA$ is not in general unique; once the valuation is fixed, however, it extends uniquely to any larger base.
\end{rem}

~
\begin{rem}
In the settings and notation of Theorem \ref{TheAmalgamationProperty},
let us assume given a saturated model $\widetilde{\Gamma_{N}}$ of
$\widetilde{\omega\OGA}$ together with embeddings $\widetilde{\Gamma_{L}}\hookrightarrow\widetilde{\Gamma_{N}}\hookleftarrow\widetilde{\Gamma_{M}}$
over $\widetilde{\Gamma_{K}}$ such that the intersection $\widetilde{\Gamma_{L}}\cap\widetilde{\Gamma_{M}}=\widetilde{\Gamma_{K}}$
inside $\widetilde{\Gamma_{N}}$; then the embedding $L\hookrightarrow N\hookleftarrow M$
can be taken to reproduce this embedding at the level of valuation
groups. If $K$ is spherically complete and obeys $\Gamma_{K}=\Gamma_{K}\otimes\mathbf{Q}\left(\sigma\right)$
then this follows from the machinery of stable domination over a spherically
complete base in the category of valued fields, see \cite{haskell2005stable};
and using the statement of Theorem \ref{TheAmalgamationProperty}
one reduces to this situation. For simplicity of the presentation,
we will not prove this slightly stronger statement.
\end{rem}

We will prove Theorem \ref{TheAmalgamationProperty} in several steps, but it is a matter of
patching various pieces together. We first describe the idea informally,
and then turn to the details.

\textbf{Informal Idea.}

The problem is naturally divided into the case of transformally transcendental
extensions and transformally algebraic extensions, and the case of
transformally algebraic extensions reduces by induction and a limit
argument to the transformally Archimedean settings. If all fields
are transformally Archimedean, then we can replace the base by its
transformally Archimedean ultrapower. This done, it has no immediate
transformally algebraic extensions, and indeed every transformally
algebraic extension must enlarge the residue field. Now the case of
residue field increase follows using the formalism of relative transformal
Henselizations.

It remains to deal with transformally transcendental extensions. We
can assume, using the results on transformal Herbrand functions, that
one always can find a root of a difference polynomial over $K$ in
a ball, unless there is an obvious obstruction to this. This implies,
by the principle of scatteredness, that every extension of $K$ not
enlarging the residue field by any new transformally algebraic elements
must contain an element which is generic in a ball, open or closed,
or an infinite nested family thereof. The case of valuation group
adjunction then follows from quantifier elimination and $o$-minimality
of $\widetilde{\omega\text{OGA}}$, and the case of immediate extensions
by compactness.

\textbf{Proof of Theorem \ref{TheAmalgamationProperty}.}

By Corollary \ref{transformalHenselizationalg}, in the settings of
$\left(1\right)$, there is an algebraically closed and transformally
Henselian extension of $K$ which can be embedded over $K$ in every
other extension with those properties; after increasing $L$ and $M$
and replacing $K$ by that, we are reduced to showing (2). From now
on, we thus assume that $K$ is algebraically closed and transformally
Henselian. Moreover, increasing $L$ and $M$ only makes the task
more difficult, so both can be taken algebraically closed and transformally
Henselian. In what follows by a span we mean a triple $L\hookleftarrow K\hookrightarrow M$
of algebraically closed and transformally Henselian models of $\VFA$;
say that a span is good if it obeys the conclusion of Theorem \ref{TheAmalgamationProperty}.
Our task is to prove that all spans is good; this is shown below.
\begin{lem}
\label{reduce-t-arch}(1) The directed union of good spans is good;
hence it is sufficient to consider spans where all fields are of finite
transformal height.

(2) The class of good spans is transitive in triplets: if the spans
$L_{21}\hookleftarrow K_{21}\hookrightarrow M_{21}$ and $L_{10}\hookleftarrow K_{10}\hookrightarrow M_{10}$
are good, then the span $L_{20}\hookleftarrow K_{20}\hookrightarrow M_{20}$
is good.

(3) It is enough to consider spans where $L$ is generated as an algebraically
closed and transformally Henselian extension of $K$ by a single element.
\end{lem}

\begin{proof}
By Lemma \ref{transitivetriplets}, the class of transformally Henselian
models is transitive in triplets, so the statement of $\left(2\right)$
is meaningful

(1) This is an easy application of the compactness theorem.

(2) Suppose given a triplet $\left(K_{21},K_{20},K_{10}\right)$ and
a pair of extensions $\left(L_{21},L_{20},L_{10}\right)$ and $\left(M_{21},M_{20},M_{10}\right)$.
Let $N_{2}=L_{2}\otimes_{K_{2}}M_{2}$ be the displayed difference
field. Since by assumption the span $L_{21}\hookleftarrow K_{21}\hookrightarrow M_{21}$
is good, there exists an $\omega$-increasing transformal valuation
ring $\mathcal{O}$ of $N_{2}$ dominating $\mathcal{O}_{L_{21}}$
and $\mathcal{O}_{M_{21}}$ and such that $L_{10}$ and $M_{10}$
are linearly disjoint over $K_{10}$ in the residue field of $\mathcal{O}$.
There is, then, a canonical embedding of $N_{1}=L_{1}\otimes_{K_{1}}M_{1}$
in the residue field of $\mathcal{O}$ over $K_{1}$. The difference
field $N_{1}$ can be expanded to a model $N_{10}=\left(N_{1},\mathcal{O}_{N_{10}}\right)$
of $\VFA$ over $L_{10}$ and $M_{10}$, rendering $L_{0}$
and $M_{0}$ linearly disjoint over $K_{0}$ in the residue field
$N_{0}$ of $N_{10}$, since by assumption the span $L_{10}\hookleftarrow K_{10}\hookrightarrow M_{10}$
is good. Let $\mathcal{O}_{N_{21}}$ be the pullback of $N_{1}$ under
the residue map of $\mathcal{O}$; then $\mathcal{O}_{N_{21}}$ is
an $\omega$-increasing transformal valuation ring of $N_{2}$ which
is dominated by $\mathcal{O}$ and which contains $\mathcal{O}_{L_{21}}$
and $\mathcal{O}_{M_{21}}$, and therefore dominates them, too. Letting
$\mathcal{O}_{N_{20}}=\mathcal{O}_{N_{21}}\times_{N_{1}}\mathcal{O}_{N_{10}}$
be the pullback, we see that the model $N_{20}=\left(N_{2},\mathcal{O}_{N_{20}}\right)$
is as desired. The additional statement regarding linear disjointness of the residue fields
holds since the residue field of $N_{20}$ is identified with the
residue field of $N_{10}$.

(3) We may assume that $L$ is finitely generated as an algebraically
closed and transformally Henselian extension of $K$, by compactness;
then by transitivity of linear disjointness we may assume inductively
that it is generated over $K$ by a single element.
\end{proof}
\begin{lem}
\label{talg-amalgamation}Every span in which both extensions are
transformally algebraic is a good span.
\end{lem}

\begin{proof}
Note for transformally algebraic extensions the rational disjointness requirement on the valuation groups is automatically satisfied, as they only increase the
valuation group inside the transformal divisible hull.
Let $L\hookleftarrow K\hookrightarrow M$
be a span in which both extensions are transformally algebraic, algebraically
closed and transformally Henselian. By Lemma \ref{reductionarch},
and using the transitivity property of transformally algebraic extensions
in triplets, we may assume that $K$ is transformally Archimedean.
By Lemma \ref{transformallyarchbasechange} and transitivity of linear
disjointness, we can pass to a base extension and replace $K$ with
the transformally Archimedean ultrapower, which is again algebraically
closed and transformally Henselian. Now that $K$ has been replaced
by that, using Lemma \ref{EC-TA-Ultrapower} we may assume that $K$
is complete, transformally Archimedean, transformally algebraically
maximal and obeys $\Gamma=\Gamma\otimes\mathbf{Q}\left(\sigma\right)$. 

By Lemma \ref{reduce-t-arch} again it is enough to find $K\subset E\subset L$
properly extending $K$, algebraically closed and transformally Henselian,
such that the span $E\hookleftarrow K\hookrightarrow M$ is good.
Since $K$ is transformally algebraically maximal and obeys $\Gamma=\Gamma\otimes\mathbf{Q}\left(\sigma\right)$
the extension $l$ of $k$ is a proper transformally algebraic extension;
so we can take $E$ to be the algebraically closed and transformally
Henselian hull of the universal transformally Henselian extension
associated with the embedding $k\hookrightarrow l$, in which case
Lemma \ref{unramified-linearly-disjoint amalgamation} applies.
\end{proof}
\begin{lem}
All spans are good.
\end{lem}

\begin{proof}
Let $L\hookleftarrow K\hookrightarrow M$ be a span, with all fields
algebraically closed and transformally Henselian. The proof is by
induction on the transformal transcendence degree of $L$ over $K$,
the case of degree zero given already by Lemma \ref{talg-amalgamation}. Using the same Lemma \ref{talg-amalgamation}
and Corollary \ref{newton-max} we may assume that $K$ is transformally
Newtonian and strictly transformally Henselian. First assume that
$K$ is nontrivially valued. By Lemma \ref{reduce-t-arch} we may
assume that $L$ is generated over $K$ as an algebraically closed
and transformally Henselian field by a single element. By Hensel lifting
we may assume no element of $l$ is transformally algebraic over $k$
unless it already lies in $k$; since $K$ is transformally Newtonian
it follows that every element $a\in L$ not in $K$ is generic over
$K$. In this situation, the field $L$ embeds over $K$ in an ultrapower
using Corollary \ref{newton-generic}. The same can be assumed for
$M$; so increasing $M$ and $L$ we may assume that both are elementary
extensions of $K$. Now the proof is standard; for example, let $c$
be an enumeration of $L$ and find an extension of $\tp\left(\nicefrac{c}{K}\right)$
to $M$ finitely satisfiable in $K$.

It remains to deal with the case where $K=k$ has the trivial valuation.
If one of $M$ or $L$ is also trivially valued then the proof is
clear; so assume both are nontrivially valued. Fix elements $a\in L$
and $b\in M$ of strictly positive valuation; we may assume that $L$
is generated as an algebraically closed and transformally Henselian
field over $k$ by the element $a$ and likewise for $b$ and $M$. By Corollary
\ref{trivially-valued} the field $L$ embeds in $k\left(\left(t^{\mathbf{Q}\left(\sigma\right)}\right)\right)$
over $k$ and likewise for $M$; so we may take $L=M=k\left(\left(t^{\mathbf{Q}\left(\sigma\right)}\right)\right)$.
Then the transformal valued field $k\left(\left(t^{\mathbf{Q}\left(\sigma\right)\oplus\mathbf{Q}\left(\sigma\right)}\right)\right)$
is easily seen to work.
\end{proof}

\subsection{\label{subsec:82}$\widetilde{\VFA}$~}

 Careful inspection of the proof of Theorem \ref{TheAmalgamationProperty} will show that if
 $K$ is a model of $\VFA$ which is nontrivially valued, transformally
 Newtonian and strictly transformally Henselian then $K$ is existentially
 closed.   We will proceed in a different way, invoking the amalgamation property in order
 to reduce to extensions that are generated by a single element in an appropriate sense, and obtaining
 a more minimalistic set of axioms for the  class of existentially closed models of $\VFA$.

\begin{defn}
Let $\widetilde{\VFA}$ be the following first order theory
of models $K$ of $\VFA$:

(1) The field $K$ is strictly transformally Henselian: it is transformally
Henselian and its residue field is a model of $\ACFA$

(2) The valuation group $\Gamma$ of $K$ is nonzero and transformally
divisible, that is it is a model of $\widetilde{\omega\OGA}$

(3) Every nonconstant additive operator given by a linear combination
of difference monomials of the form $x^{\sigma^{n}\cdot p^{m}}$ for
$n,m\in\mathbf{N}$ is onto.
\end{defn}

\begin{prop}
\label{wvfa-newtonian}Let $K$ be a model of $\widetilde{\VFA}$.
Then $K$ is transformally Newtonian.
\end{prop}

\begin{proof}
First note that $K$ is algebraically closed as a field. Indeed $K$
is algebraically strictly Henselian, as it is strictly transformally
Henselian; it is moreover algebraically tamely closed, since $\Gamma=\Gamma\otimes\mathbf{Q}\left(\sigma\right)$
and hence $\Gamma=\Gamma\otimes\mathbf{Q}$. The operator $x^{p}-x$
is onto on $K$ by the axioms; as $K$ is perfect, by the discussion of \ref{subsec-ramificationtheory} we learn that $K$ is algebraically closed. 

The field $K$ is the directed union of models of $\widetilde{\VFA}$
of finite transformal height; the directed union of transformally
Newtonian models is again one; so we may assume that $K$ is of finite
transformal height. Since $k$ is a model of $\ACFA$ and $\Gamma=\Gamma\otimes\mathbf{Q}\left(\sigma\right)$
is transformally divisible, using the criterion of Proposition \ref{newton-max}, it is enough to prove that $K$ is transformally
algebraically maximal.

First assume that $K$ is transformally Archimedean. In this situation, it follows from Corollary \ref{cor-criterion-talg-max} that $K$ is transformally algebraically maximal. In general, suppose that $\left(K_{20},K_{21},K_{10}\right)$
is a triplet with $K_{20}$ a model of $\widetilde{\VFA}$ of
finite transformal height. Let us assume that $K_{21}$ and $K_{10}$
both nontrivially valued and that $K_{21}$ is transformally Archimedean.
Then $\Gamma_{20}$ is transformally divisible, hence so are $\Gamma_{21}$
and $\Gamma_{10}$. By Lemma \ref{transitivetriplets} the field $K_{10}$
is strictly transformally Henselian. Moreover, every additive difference
polynomial over $K_{21}$ has a root; since $K_{10}$ is the homomorphic
image of $\mathcal{O}_{K_{21}}$ it follows that every additive difference
operator is onto on $K_{10}$. So $K_{10}$ is a model of $\widetilde{\VFA}$;
by induction on the transformal height it is transformally algebraically maximal.
Next, the field $K_{21}$ is transformally Henselian, algebraically
closed and transformally Archimedean with a transformally divisible
valuation group; since $K_{21}$ is closed under additive operators, it is transformally algebraically maximal, using Corollary \ref{cor-criterion-talg-max}. Thus $v_{20}$ is glued from
the transformally algebraically maximal valuations $v_{21}$ and $v_{10}$;
by Lemma \ref{talgmax-transitive} we learn that $v_{20}$ is transformally
algebraically maximal.
\end{proof}

\begin{thm}
\label{model-complete}Every model of $\VFA$ embeds in a model
of $\widetilde{\VFA}$; the theory $\widetilde{\VFA}$
is model complete. If $K$ is a model of $\VFA$ with no nontrivial
$h$-finite $\sigma$-invariant extensions, then the theory $\widetilde{\VFA}_{K}$
of models of $\widetilde{\VFA}$ over $K$ is complete. In particular,
the theory $\widetilde{\VFA}$ eliminates quantifiers to the
level of the field theoretic algebraic closure. 
\end{thm}

In what follows, if $F$ is a model of $\VFA$ then we write
$\acl_{\omega}F$ for the choice of an algebraically closed and transformally
Henselian hull of $F$.

\textbf{Informal idea.} We first describe the idea informally. It
is not difficult to prove that every model of $\VFA$ embeds
in a model of $\widetilde{\VFA}$, so one must show that the
models of $\widetilde{\VFA}$ are existentially closed. Thus
let $M$ be a saturated model of $\widetilde{\VFA}$. Let $A\subseteq M$
be countable with $A=\acl_{\omega}A$ and let $B=\acl_{\omega}B$
be a countable extension of $A$ which properly extends $A$; we must
prove that $B$ embeds in $M$ over $A$. One reduces to the case
where $B$ is transformally algebraic and $\acl_{\omega}$-finitely
generated over $A$. In the case where $A$ is transformally Newtonian,
using Lemma \ref{lem:partially-EC} below, one can find $A\subseteq C\subseteq B$
with $\acl_{\omega}C=C$ which properly extends $A$ and such that
$C$ embeds in $M$ over $A$. Now the field $C$ is not in general
transformally Newtonian; in this case we replace $C$ with a transformally
Newtonian model $C^{\star}$ inside $M$, replace $B$ with $B^{\star}=\acl_{\omega}\left(B\otimes_{C}C^{\star}\right)$
and repeat the process. 

One can show that in retrospect that this
process terminates in finitely many stages, for example by induction
on $\text{SU}$ rank or the inertial dimension of $B$ over $A$;
in what follows we give a soft proof in which the process is repeated
transfinitely, exploiting the fact that the valuation group does not
increase in transformally algebraic extensions. An alternative approach directly using scatteredness is given in Remark \ref{scatteredness}

Before we turn to the proof of Theorem \ref{model-complete} we state and prove the Lemma
promised in the informal description:
\begin{lem}
\label{lem:partially-EC}Let $M$ be a saturated model of $\widetilde{\VFA}$.
Let $A\subseteq M$ be countable and transformally Newtonian. Let
$B=\acl_{\omega}\left(Ab\right)$ be a transformally algebraic extension
of $A$ which is algebraically closed and transformally Henselian,
where $b\in B$ is a single element lying outside $A$. Then there
is an element $c\in B$ with the following properties:

(1) We have the inequality $v\left(b-c\right)>v\left(a-c\right)$
for all $a\in A$

(2) The field $\acl_{\omega}\left(Ac\right)$ embeds in $M$ over
$A$
\end{lem}

\begin{proof}
We first reduce ourselves to the case where $vb=0$ and the residue
class of $b$ is not in $k_{A}$. Since $A$ is transformally Newtonian,
it is transformally algebraically maximal and we have $\Gamma_{A}=\Gamma_{A}\otimes\mathbf{Q}\left(\sigma\right)$.
So as $B$ is transformally algebraic over $A$, there is an element
$a^{\star}\in A$ maximizing the quantity $v\left(b-a^{\star}\right)$,
i.e, we have $v\left(b-a^{\star}\right)\geq v\left(b-a\right)$ for
all $a\in A$. Let $b^{\star}=b-a^{\star}$. Replacing $b$ by $b^{\star}$
is harmless: if $c^{\star}$ witnesses the truth of the statement
for $b^{\star}$, then $c=c^{\star}+a^{\star}$ witnesses the truth
of the original statement. We may therefore assume $a^{\star}=0$
maximizes the valuative distance, i.e we have the inequality $v\left(b\right)\geq v\left(b-a\right)$
for all $a\in A$. In a similar fashion, using $\Gamma_{A}=\Gamma_{A}\otimes\mathbf{Q}\left(\sigma\right)=\Gamma_{B}$,
after replacing $b$ by a scalar multiple of an element of $A$, we
may assume that $vb=0$; so the residue class of $b$ is not in $k_{A}$.

Let $A^{\text{tunr}}$ be the maximal transformally unramified extension
of $A$ in $B$, so $k_{A^{\text{tunr}}}=k_{B}$; so the field $A^{\text{tunr}}$
is the universal transformally Henselian extension of $A$ reproducing
the embedding $k_{A}\hookrightarrow k_{B}$ at the level of residue
fields. Since $M$ is a model of $\widetilde{\VFA}$, it is
strictly transformally Henselian, so the field $A^{\text{tunr}}$
embeds in $M$ over $A$. Since the field $A$ and the residue field
$k_{A^{\text{tunr}}}$ are both algebraically closed, it follows from \ref{transformalhenselization}
that $A^{\text{tunr}}$ admits no nontrivial finite $\sigma$-invariant
Galois extensions; using $M=\acl_{\omega}M$ and Corollary \ref{transformalHenselizationalg},
we can further find an embedding of $\acl_{\omega}\left(A^{\text{tunr}}\right)$
in $M$ over $A$. Finally, since $k_{A^{\text{tunr}}}=k_{B}$ and
the residue class of $b$ is not in $k_{A}$, we can find an element
$c\in\acl_{\omega}\left(A^{\text{tunr}}\right)$ which is as in the
lemma (namely any element whose residue class coincides with that
of $b$).
\end{proof}
\textbf{Proof of Theorem \ref{model-complete}} We now turn to the details.

\textbf{Step 1.} Every model of $\VFA$ embeds in a model of
$\widetilde{\VFA}$. Let $K$ be a model of $\VFA$. If
$K$ is trivially valued then we may embed it in a field of Hahn series
so assume that $K$ is nontrivially valued. Using Lemma \ref{valuation-group-increase}
we may assume that $\Gamma=\Gamma\otimes\mathbf{Q}\left(\sigma\right)$
and using Proposition \ref{strict-t-henselization} we may assume
that $K$ is strictly transformally Henselian. By Corollary \ref{newton-max}
we may assume that $K$ is transformally Newtonian; this implies that
every nonconstant difference polynomial with coefficients in $K$
has a $K$-root; as this holds in particular for additive operators,
we find that $K$ is a model of $\widetilde{\VFA}$

Next we show that the models of $\widetilde{\VFA}$ are existentially
closed. Fix a saturated model $M$ of $\widetilde{\VFA}$. Let
$A\subseteq M$ be countable with $\acl_{\omega}A=A$. Let $B=\acl_{\omega}B$
be a countable extension of $A$; we want to prove that $B$ embeds
in $M$ over $A$.

\textbf{Step 2.} We first deal with the case where $B$ is transformally
algebraic over $A$. By compactness, we may assume that $B$ is $\acl_{\omega}$-finitely
generated over $A$, that is, we have $B=\acl_{\omega}\left(Ab\right)$
for a finite tuple $b\in B$. By induction on the number of generators,
we may moreover assume that $b$ is a single element.

We shall construct continuous transfinite sequences
\[
A=A_{0}\subseteq A_{1}\subseteq A_{2}\subseteq\ldots\subseteq A_{\alpha}\subseteq\ldots
\]
and
\[
B=B_{0}\subseteq B_{1}\subseteq\ldots\subseteq B_{\alpha}\subseteq\ldots
\]
of models of $\VFA$ with $B_{\alpha}=\acl_{\omega}B_{\alpha}$
and $A_{\alpha}=\acl_{\omega}A_{\alpha}$, for each countable ordinal
$\alpha$. Each $A_{\alpha}$ will be equipped with an embedding $i_{\alpha}\colon A_{\alpha}\hookrightarrow M$
and an embedding $j_{\alpha}\colon A_{\alpha}\hookrightarrow B_{\alpha}$
over $A=A_{0}$. These embeddings will be coherent in the sense that
for $\alpha<\beta$ we have that $i_{\alpha}$ is the restriction
of $i_{\beta}$ to $A_{\alpha}$ and likewise for $j_{\alpha}$ and
$j_{\beta}$. Furthermore, the fields $B_{\alpha}$ will all be transformally
algebraic over $B=B_{0}$. Finally, for each $\alpha$, if $b\notin A_{\alpha}$,
then there will be some $\beta>\alpha$ and some $c_{\beta}\in A_{\beta}$
with the property that $v\left(b-c_{\beta}\right)>v\left(b-a\right)$
for all $a\in A_{\alpha}$. Since $\Gamma_{B_{\alpha}}=\Gamma_{A}\otimes\mathbf{Q}\left(\sigma\right)$
such a sequence must eventually terminate; the embedding of $A_{\alpha}$
in $M$ over $A$ for $\alpha$ large will thus restrict to an embedding
of $\acl_{\omega}\left(Ab\right)$ in $M$ over $A$, as wanted.

The construction is as follows. For $\alpha$ a limit ordinal take
the union. If $\alpha$ is even, let $A_{\alpha+1}$ be a countable
transformally Newtonian model of $\VFA$ over $A_{\alpha}$
inside $M$, and let $B_{\alpha+1}=\acl_{\omega}\left(B_{\alpha}\otimes_{A_{\alpha}}A_{\alpha+1}\right)$; thus $B_{\alpha+1}$ is an algebraically closed and transformally Henselian extension of $A_{\alpha}$ containing copies of $B_{\alpha}$ and $A_{\alpha + 1}$, and generated as an algebraically closed and transformally Henselian field by the images of those fields; the choice of $B_{\alpha+1}$ is not in general unique, even up to isomorphism, but its existence is guaranteed by Theorem \ref{TheAmalgamationProperty}. If $\alpha$ is
odd, we let $B_{\alpha+1}=B_{\alpha}$ and $A_{\alpha+1}=\acl_{\omega}\left(A_{\alpha}c_{\alpha+1}\right)$,
where $c_{\alpha+1}\in B_{\alpha+1}$ is the element provided by Lemma
\ref{lem:partially-EC}.

\textbf{Step 3.} For the general case, increasing $B$ (while keeping it countable) we may assume
that it is itself a model of $\widetilde{\VFA}$ and in particular transformally Newtonian. Let $C$ be a model of $\VFA$ with $A\subseteq C\subseteq B$
and with $C=\acl_{\omega}C$ maximal with
respect to the property that it embeds in $M$ over $A$. By Step
2, we find that no element of $B$ is transformally algebraic over
$C$ unless it already lies in $C$. Since $B$ is a model of $\widetilde{\VFA}$,
it is transformally Newtonian, so as $C$ is transformally algebraically
closed in $B$, the field $C$ is transformally Newtonian also. By
Corollary \ref{newton-generic} we find every $a\in B$ not in $C$
is generic over $C$, thus $\acl_{\omega}\left(Ca\right)$ embeds
in $M$ over $C$; using maximality of $C$ we thus find that $C=B$,
and the proof is finished.

\begin{rem}
\label{scatteredness}
We now give an alternative approach to the proof of Theorem \ref{model-complete}, avoiding transfinite towers. In fact, the argument shows that the transfinite tower constructed above terminates in finitely many steps, with an explicit bound.

Let $K$ be a countable model of $\VFA$ which is algebraically closed , transformally Henselian and nontrivially valued. Let $M$ be a saturated model of $\widetilde{\VFA}$ over $K$, and $L=K\left(a\right)_\sigma$ a transformally algebraic extension of $K$, with $a$ a single element; we want to embed $L$ in $M$ over $K$. Using Example \ref{example-monogonic-t-henselian} and compactness, this guarantees that every countable transformally algebraic extension of $K$ embeds in $M$ over $K$.

The element $a$ is a root of some difference polynomial $fx$ with coefficients in $K$. After twisting, we may assume that $f'a \neq 0$. By scatteredness, only finitely many valuative distances are attained between distinct roots of $fx$, say $\gamma_{1} < \ldots < \gamma_{N}$. Note that this set of valuative distances is fixed, and does not change upon base extension. To facilitate the induction below, we also put formally $\gamma_0 = -\infty$.

We will find inductively a countable $K_n = \acl_{\omega}K_n$ over $K$ contained in $M$, an element $a_n \in K_n$ and common amalgamation of $L$ and $K_n$ over $K$, such that the valuative distance $v\left(a-a_n\right) > \gamma_n$ inside the amalgam; i.e, the element $a$ lies in an open ball of valuative radius $\gamma_n$ around an element of $K_n$. For $n = 0$, we can simply put $K_0 = K$ and $a_0 = 0$.

Let us assume for that this has been achieved. Relabelling, we may write $K=K_n$. Then $a$ is the unique root of of $fx$ inside some open ball $K$-ball; since $K$ is transformally Henselian, the element $a$ lies in $K$, and we are done.

It remains to find $K_{n+1}$ given $K_n$. So assume that $K_n$ and $c_n$ were found; for simplicity of notation, relabelling, we may write $K = K_n$, $c_n = c$, $\alpha = \gamma_{n}$ and $\beta = \gamma_{n+1}$. So $a$ lies an open ball of valuative radius $\alpha$ around an element $K$; we seek to increase $K$ inside $M$ and amalgamate, so that $a$ will lie in an open ball of valuative radius $\beta$ around an element of $K$. Since $M$ is transformally Newtonian, there is an $M$-root $d$ of $fx$ in the open ball of valuative radius $\alpha$ around the element $a$. Let us amalgamate $\acl_{\omega}Kc$ and $L$ together over $K$; replacing $K$ by $\acl_{\omega}Kd$ we may therefore assume that $a$ lies in a closed ball of valuative radius $\beta$ around an element of $K$. Increasing $K$ again inside $M$ we may assume that $\Gamma$ is transformally divisible; after an affine change of variables, we may thus assume that $\beta = 0$.

To finish it is enough to find some $d \in \acl_{\omega}Ka$, whose residue class coincides with that $a$, and such that $\acl_{\omega}Kd$ embeds in $M$ over $K$. If the residue class of $a$ has a representative in $K$, then this is trivial. But in any case, the argument of Lemma \ref{lem:partially-EC} applies, and the induction step is complete.

\end{rem}

\begin{prop}
\label{saturation}Let $K$ be a model of $\widetilde{\VFA}$
and $\kappa$ an infinite cardinal. Then $K$ is $\kappa^{+}$-saturated
as a model of $\widetilde{\VFA}$ if and only if $k$ and $\Gamma$
$\kappa^{+}$-saturated as models of $\ACFA$ and $\widetilde{\omega\OGA}$
respectively, and moreover every nested family of balls in $K$ of
cardinality $\kappa$ has a nonempty intersection.
\end{prop}

\begin{proof}
As in the proof of Theorem \ref{model-complete}, it is enough to
establish the following. Let $A$ be a small model contained in $K$
and $B$ a small algebraically closed and transformally Henselian
extension of $A$ properly extending $A$; then there is some $A\subset B'\subset B$
with $B'$ transformally Henselian, without any finite $\sigma$-invariant
Galois extensions, properly extending $A$, and so that $B'$ embeds
in $K$ over $A$.

If $B$ is transformally algebraic over $A$ then $k_{B}$ properly
extends $k_{A}$, as $A$ is transformally Newtonian. By saturation
of $k$ we can embed $k_{B}$ in $k$ over $k_{A}$; this lifts uniquely
to an embedding of $B^{\text{tunr}}$ in $K$ over $A$. The field
$B^{\text{tunr}}$ has no nontrivial finite $\sigma$-invariant Galois
extensions, since $k_{B}$ is algebraically closed. So we may assume
that no element of $B$ is transformally algebraic over $A$ unless
it already lies in $A$; thus every element of $B$ not in $A$ is
generic over $A$. Saturation of $k$ implies that it is of large
transformal transcendence degree over the prime field, so the generic
of $\mathcal{O}$ over $A$ is realized in $K$. Similarly saturation
of $\Gamma$ and the fact that every small nested family of closed
balls in $K$ has a nonempty intersection implies that the other two
kinds of generic types over $A$ are realized in $K$.
\end{proof}
\begin{thm}\label{thm:stably-embedded}
Let $K$ be a model of $\widetilde{\VFA}$. Then $k$ and $\Gamma$
are embedded, stably embedded and fully orthogonal pure models of
$\ACFA$ and $\widetilde{\omega\OGA}$ respectively. In other words,
let $D\subset k^{n}\times\Gamma^{m}$ be a definable set which is
definable with parameters in $K$ and in the language of $\widetilde{\VFA}$.
Then $D$ is a boolean combination of definable sets of the form $X_{i}\times Y_{j}$
where $X_{i}\subset k^{n}$ and $Y_{j}\subset\Gamma^{m}$ are definable in the language of difference fields (resp. ordered modules) and with parameters in $k$ (resp. $\Gamma$). Furthermore, the canonical parameter of every definable subset of $k^d$ in $\widetilde{\VFA}^{eq}$ is equidefinable in $\widetilde{\VFA}$ with a tuple from $k$, and likewise for $\Gamma$.
\end{thm}

\begin{rem}
    Let $T$ be a theory with elimination of imaginaries in a language $\mathcal{L}$ and let $\widetilde{T}$ be an expansion of $T$ in a language $\widetilde{\mathcal{L}}$ and possibly additional sorts (so every model of $T$ can be extended to a model of $\widetilde{T}$). It follows from \cite{chatzidakis1999model}, Appendix $A$ that $T$ is stably embedded in $\widetilde{T}$ if and only if the following condition is satisfied: in a saturated model of $\widetilde{T}$, every automorphism on the sorts of $T$ (i.e a partial elementary map which is bijective on the sorts of $T$) extends to a global automorphism. This is equivalent to the requirement that every definable subset of the sorts of $T$ is definable with parameters on the sorts of $T$. Moreover, say that $T$ is \emph{fully embedded} in $\widetilde{T}$ if it is stably embedded in $\widetilde{T}$ and every definable subset of the sorts of $T$ which is definable without parameters and in the language $\widetilde{\mathcal{L}}$ is already definable without parameters and in the language $\mathcal{L}$; then $T$ is fully embedded in $\widetilde{T}$ if and only if it is stably embedded in $T$ and every partial map on the sorts of $T$ which is elementary with respect to $\mathcal{L}$-formulas is also elementary with respect to $\widetilde{\mathcal{L}}$ formulas. In this language, Theorem \ref{thm:stably-embedded} says that $\ACFA$ and $\widetilde{\omega\OGA}$ are fully embedded in $\widetilde{\VFA}$. This will be used below without further mention.
\end{rem}

\begin{proof}
We may take $K=k\left(\left(t^{\Gamma}\right)\right)$ where $k,\Gamma$
are saturated; so by Proposition \ref{saturation} the field $K$
is saturated as a model of $\widetilde{\VFA}$. The restriction
map $\text{Aut}\left(K\right)\to\text{Aut}\left(k\right)\times\text{Aut}\left(\Gamma\right)$
is evidently surjective, indeed has a section. So $k$ and $\Gamma$
are stably embedded and fully orthogonal. To finish we must show that the induced structure on $k$ and $\Gamma$ is the
pure difference field and pure ordered transformal module structure.
By quantifier elimination in $\ACFA$ and $\widetilde{\omega\OGA}$,
it is enough to prove: if $k_{1}\cong k_{2}$ is an isomorphism between
small algebraically closed difference subfields of $k$ and $\Gamma_{1}\cong\Gamma_{2}$
is an isomorphism between finitely generated transformal submodules
of $\Gamma$ then it lifts to a global automorphism of $K$. The valuation
map and the residue map both have a section; so we get an isomorphism
$k_{1}\left(t^{\Gamma_{1}}\right)\cong k_{2}\left(t^{\Gamma_{2}}\right)$
of transformal valued subfields of $K$. We then finish using quantifier
elimination in $\widetilde{\VFA}$, saturation and Corollary
\ref{trivially-valued}.
\end{proof}
\begin{prop}
The theories $\widetilde{\VFA},\widetilde{\VFA_{\mathbf{Q}}}$
and $\widetilde{\VFA_{\mathbf{F}_{p}}}$ are recursive.
\end{prop}

\begin{proof}
Let $\phi$ be a sentence of the language of $\VFA$. If $\neg\phi$
holds in every model of $\widetilde{\VFA}$ then by the completeness
theorem there is a syntactic proof of $\neg\phi$ derived from the
$\widetilde{\VFA}$ axioms; this is then witnessed computably
as $\widetilde{\VFA}$ is recursively axiomatized. Suppose conversely
that $M$ is a model of $\widetilde{\VFA}$ and $M\models\phi$.
The elementary theory of $M$ as a model of $\VFA$ is determined
by the elementary theory of $k_{M}$ as a pure difference field and
the fact that $M$ is a model of $\widetilde{\VFA}$. By compactness
there is therefore a sentence $\psi$ consistent with $\ACFA$ such
that $\phi$ continues to hold in every model of $\widetilde{\VFA}$,
provided only that the residue field obeys $\psi$; since $\ACFA$
is recursive, systematically enumerating all such sentences and searching
for a syntactic proof, we find that if the theory of models of $\widetilde{\VFA}$
where $\phi$ holds is consistent then this can be witnessed computably.
It follows that $\widetilde{\VFA}$ is recursive; the proof
of the other cases goes through without change since the theories
$\ACFA_{\mathbf{Q}}$ and $\ACFA_{\mathbf{F}_{p}}$ are recursive.
\end{proof}
\begin{prop}
\label{dcl}Let $M$ be a model of $\widetilde{\VFA}$ and let
$K$ be a transformal valued subfield. Then $K$ is model theoretically
algebraically closed in $M$ if and only if it is field theoretically
algebraically closed, inversive, and transformally Henselian.
\end{prop}

\begin{proof}
By uniqueness of Hensel lifts, the transformal Henselization is contained
in the definable closure. The inversive hull and the field theoretic
algebraic closure are moreover evidently contained in the model theoretic
algebraic closure. For the converse, let $E$ be an algebraically
closed, transformally Henselian extension of $K$ in $M$. Then $E\otimes_{K}E$
can be equipped with the structure of a model of $\VFA$. By
model completeness of $\widetilde{\VFA}$ and quantifier elimination,
we find that $\text{tp}\left(\nicefrac{E}{K}\right)$ is realized
anew over $E$, so that no element of $E$ is model theoretically
algebraic over $K$ unless it already lies in $K$.
\end{proof}

\subsection{Quantifier Elimination}

We give now give a precise version of the quantifier elimination statement.

\subsubsection{~}

Let $\FA$ be the theory of perfect, inversive difference fields (with no valuation).
The language $\mathcal{L}_{\FA}$ expands the language of fields by
unary function symbols $\sigma^{\pm1}$ and $\phi^{\pm1}$ with the
obvious interpretations. Let $\widetilde{\mathcal{L}_{\FA}}$ be the
expansion of $\mathcal{L}_{\FA}$ where basic formulas are of the
form:
\[
\exists y\theta\left(x_{1},\ldots,x_{n},y\right)
\]
where $\theta\left(x_{1},\ldots,x_{n},y\right)$ is a quantifier free
formula in the language $\mathcal{L}_{\FA}$; and moreover for some
$m\in\mathbf{N}$ and difference polynomials $f_{0},\ldots,f_{m-1},g_{0},\ldots,g_{m-1}\in\mathbf{Z}\left[x_{1},\ldots,x_{n}\right]_{\sigma}$
the formula $\theta\left(x_{1},\ldots,x_{n},y\right)$ implies that:
\[
y^{m}=f_{0}\left(x_{1},\ldots,x_{n}\right)+\ldots+f_{m-1}\left(x_{1},\ldots,x_{n}\right)\cdot y^{m-1}
\]
\[
y^{\sigma}=g_{0}\left(x_{1},\ldots,x_{n}\right)+g_{m-1}\left(x_{1},\ldots,x_{n}\right)\cdot y^{m-1}
\]

This means that quantification is only allowed over algebraic difference field extensions which are primitive already as abstract field extensions (and in particular finite).

\begin{prop}
\label{qe-acfa}The theory $\ACFA$ admits elimination of quantifiers
in $\widetilde{\mathcal{L}_{\FA}}$.
\end{prop}

\begin{proof}
Let $M$ be a saturated model of $\ACFA$ and $\varphi\colon K\cong L$
an isomorphism between models of $\FA$ contained in $M$ preserving
$\widetilde{\mathcal{L}_{\FA}}$ formulas; we must prove that $\varphi$
lifts to a global automorphism. Let $\widetilde{K}$ be a finite Galois
extension of $K$ left invariant under $\sigma$ and $\widetilde{L}$
the corresponding $\sigma$-invariant Galois extension of $L$ under
the isomorphism $\varphi$. By the primitive element theorem, we have
$\widetilde{K}=K\left[a\right]=K\left[a\right]_{\sigma}$ for a single
element $a\in\widetilde{K}$. Now $\nicefrac{\widetilde{K}}{K}$ is
finitely presented as a difference algebra over $K$, so as $\varphi$
preserves $\widetilde{\mathcal{L}_{\FA}}$ formulas there is some
$b\in\widetilde{L}$ such that $\widetilde{K}$ and $\widetilde{L}$
are isomorphic as difference fields by an isomorphism extending $\varphi\colon K\cong L$
and which carries $a$ to $b$. So the isomorphism extends to any
finite $\sigma$-invariant Galois extension.

Let $\overline{K}$ and $\overline{L}$ denote the union of all finite
$\sigma$-invariant Galois extensions of $K$ and $L$ respectively.
By the previous paragraph we can extend $\varphi$ to an isomorphism
of any finite $\sigma$-invariant extension of $K$ contained in $\overline{K}$;
by compactness of the Galois group we can extend $\varphi$ to an
isomorphism of difference fields between $\overline{K}$ and $\overline{L}$.
Changing the notation assume that $K$ has no nontrivial finite $\sigma$-invariant
Galois extensions. In this case there is up to isomorphism a unique
field theoretic algebraic closure of $K$. Thus we may assume that $K$ is algebraically closed. In this case the category of difference fields over $K$ admits the amalgamation property, hence $\ACFA_{K}$ is complete, and we can lift the isomorphism $K \cong L$ to an automorphism of $M$.
\end{proof}

\subsubsection{~}

Let $\VFA^{h}$ be the theory of algebraically Henselian models
of $\VFA$. The language $\mathcal{L}_{\VFA^{h}}$ can
be described as follows. It admits a single sort $K$ together with
unary predicates $K^{\times},\mathcal{O}$ and $\mathcal{M}$ as well
as constant symbols $0,1$. We have binary function symbols $+,-,\cdot\colon K^{2}\to K$
and finally unary function symbols $\sigma^{\pm1}$ and $\phi^{\pm1}$.
Finally, for each $n\geq2$ we have a function $h_{n-2}\colon\mathcal{M}^{n-2}\to\mathcal{O}$.
The axioms assert that $K$ is a model of $\VFA$ and for all
$t_{2},\ldots,t_{n}\in\mathcal{M}^{n-2}$, if $x=h_{n-2}\left(t_{2},\ldots,t_{n}\right)$
then we have
\[
1+x+t_{2}x^{2}+\ldots+t_{n}x^{n}=0
\]
we call the $h_{n}$ the Hensel lifting functions of $K$. If $K$
is an algebraically Henselian model of $\VFA$, then $K$ can
be uniquely expanded to a model of $\VFA^{h}$ and conversely.

Now consider the language $\widetilde{\mathcal{L}_{\VFA}}$
where basic formulas are of the form:

\[
\phi\left(x_{1},\ldots,x_{n}\right)\wedge\psi\left(x_{1},\ldots,x_{n}\right)
\]
where $\phi$ is quantifier free in the language of $\VFA^{h}$
and $\psi\left(x_{1},\ldots,x_{n}\right)$ is a formula in the language
$\widetilde{\mathcal{L}_{\FA}}$.
\begin{prop}
\label{qe-wvfa}Notation as above. Then $\widetilde{\VFA}$
eliminates quantifiers in $\widetilde{\mathcal{L}_{\VFA}}$.
\end{prop}

\begin{proof}
Let $M$ be a saturated model of $\widetilde{\VFA}$. Let $E,F$
be models of $\VFA^{h}$ which are contained in $M$ and $\varphi\colon E\cong F$
an isomorphism of transformal valued fields which preserves $\widetilde{\mathcal{L}_{\VFA}}$-formulas;
we must show that $\varphi$ lifts to a global automorphism. The proof
is the same as Lemma \ref{qe-acfa}, using the criterion of \ref{criterion-unique-algclosure}.
\end{proof}

\subsection{The Elementary Theory of the Frobenius Automorphism}

Recall that by a \textit{Frobenius transformal valued field} we mean
a transformal valued field of $K$ of positive characteristic $p>0$
whose distinguished endomorphism coincides with a power of the Frobenius
endomorphism $\phi x=x^{p}$. We shall regard perfect Frobenius transformal valued
fields as structures for the language $\VFA$ of perfect, inversive
transformal valued fields.
\begin{fact}
\label{effective-frob}Let $\varphi$ be a sentence in the language
of difference fields and assume that $\ACFA\models\varphi$. Then
there is an effectively computable $N=N\left(\varphi\right)$ such
that whenever $\left(K,\sigma\right)$ is an algebraically closed
Frobenius difference field and $x^{\sigma}=x^{q}$ where $q>N$, then
$\left(K,\sigma\right)\models\varphi$.
\end{fact}

\begin{proof}
This follows from the fact that the constant implicit in the geometric
estimates of \cite{hrushovski2004elementary}
is effectively computable; see the discussion on pp. 125.
\end{proof}
\begin{thm}
\label{frobarevfa} Let $T$ be the theory of algebraically closed
and nontrivially valued Frobenius transformal valued fields. Then
$T$ is decidable, and the theory $\widetilde{\VFA}$ is the
theory of $\omega$-increasing models of $T$.
\end{thm}

\begin{proof}
First we show that $\widetilde{\VFA}$ is the theory of $\omega$-increasing
and nontrivially valued models of $T$. Let $K$ be an $\omega$-increasing
and nontrivially valued model of $T$. Then $K$ is a model of $\widetilde{\VFA}$;
it is transformally Henselian by \ref{lem:(Basic-properties)}; the
valuation group is nonzero and transformally divisible by Proposition
\ref{woga-frobenius}; the residue field is a model of $\ACFA$ by
the main theorem of \cite{hrushovski2004elementary}; and finally,
it is easy to check that the additive difference operators are onto.
Conversely, we must show that every model of $\widetilde{\VFA}$
arises in this way. Now in $\widetilde{\VFA}$ the residue field
is stably embedded and the induced structure is precisely the pure
difference field structure. Moreover, the completions of $\widetilde{\VFA}$
are determined by the elementary theory of the residue field as a
pure difference field. Since $K$ is elementarily equivalent to its
residue field as a difference field, the result follows from \ref{frobacfa}.

Now we show that $T$ is decidable. Since $\widetilde{\VFA}$
is decidable, this boils down to the following: given a sentence $\phi$
such that $\widetilde{\VFA}\models\phi$ find effectively a
natural number $N=N\left(\varphi\right)$ such that $K\models\varphi$,
whenever $K$ is an algebraically closed and nontrivially valued Frobenius
transformal valued field with $x^{\sigma}=x^{q}$ and $q>N$. It is
enough to do this for a sentence $\varphi$ which is part of our axiomatization.
For the $\ACFA$ axioms over the residue field we can use Fact \ref{effective-frob}.
Now consider for instance a part of the axiom scheme asserting that
the field is transformally Henselian. Every finite fragment asserts that the
Hensel lifting lemma holds for difference polynomials supported on a fixed
finite subset $I\subset\mathbf{N}\left[\frac{\sigma}{p^{\infty}}\right]$.
If $p=1$, it is sufficient to choose $q$ large enough so that the
substitution $\sigma\mapsto q$ from $\mathbf{N}\left[\sigma\right]$
to $\mathbf{N}$ is injective when restricted to $I$; simply take
$q$ large enough so it exceeds the coefficients of all the elements
of $I$. In finite characteristic $p$, let $m\in\mathbf{N}$ be the
largest power of $p$ which appears in the denominator of the various
$\nu\in I$ one should choose $q$ large so that $\frac{q}{p^{m}}$
is a natural number, and then apply the same reasoning. 
\end{proof}

\subsection{Various Results}
\begin{prop}
Let $K$ be a saturated model of $\widetilde{\VFA}$ and let
$F$ be the fixed field of $K$. Then $F$ is a pseudofinite field
which inherits from $K$ the trivial valuation. Furthermore, every
definable subset of $F^{n}$ which is definable with parameters in
$K$ and in the language of $\VFA$ is already definable with
parameters in $F$ and in the pure field language.
\end{prop}

\begin{proof}
By Hensel lifting the residue map identifies $F$ with the fixed field
of $k$, so the Theorem follows from stable embededness of the residue
field in $\widetilde{\VFA}$ and the corresponding statement
for $\ACFA$; see \cite{chatzidakis1999model}.
\end{proof}
\begin{prop}
Let $K$ be a model of $\widetilde{\VFA}$ and $k_{0}$ the
algebraic closure of the prime field.

(1) Let $0<n\in\mathbf{N}$ and $m\in\mathbf{Z}$; then $\left(K,\sigma^{n}\circ\phi^{m}\right)$
is a model of $\widetilde{\VFA}$, which is saturated provided
that $\left(K,\sigma\right)$ is.

(2) Let us assume that $K$ is saturated. Fix $n>0$ and let $\tau_{0}$
be an automorphism of $k_{0}$ such that $\tau_{0}^{n}$ and $\sigma$
agree on $k_{0}$; then there is an $\omega$-increasing automorphism
$\tau$ of $K$ extending $\tau_{0}$, which renders $\left(K,\tau\right)$
a saturated model of $\widetilde{\VFA}$, and such that $\tau^{n}=\sigma$
on $K$.

(3) The elementary theory of $K$ is determined by the isomorphism
type of the restriction of $\sigma$ to $k_{0}$.
\end{prop}

\begin{proof}
For (3) use Theorem \ref{model-complete}. For (1), let $\widetilde{K}$ denote
the model of $\VFA$ whose underlying valued field is the same
and whose automorphism is given by $\tau = \sigma^n \circ \phi^m$. Then $\widetilde{k}$
is a model of $\ACFA$ which is saturated provided that $k$ is; see
\cite{chatzidakis1999model}, 1.12. One then easily checks that $\widetilde{K}$ is a model of $\widetilde{\VFA}$; saturation follows at once from Proposition \ref{saturation} (namely $\widetilde{\Gamma}$ is still saturated by $o$-minimality and the condition on intersections of nested families of balls is still satisfied, as it depends only on the underlying valued field). Finally, for (2), using Hensel's lemma we may identify $k_0$ with the algebraic closure of the prime field in $k$; the claim follows from full embeddedness of the residue field in $\widetilde{\VFA}$ and \cite{chatzidakis1999model}, 1.12.
\end{proof}
\begin{prop}
\label{ACFA-t-hens}Let $K$ be a model of $\VFA$ which is
nontrivially valued. Then $K$ is a model of $\widetilde{\VFA}$
if and only if is transformally Henselian and its underlying difference
field is a model of $\ACFA$.
\end{prop}

\begin{proof}

By definition models of $\widetilde{\VFA}$ are transformally Henselian. Furthermore, 
if $K$ is a model of $\widetilde{\VFA}$ then $K$ is elementarily equivalent to an ultraproduct of algebraically closed Frobenius difference fields, using Theorem \ref{frobarevfa}; thus
the underlying difference field of a model of $\widetilde{\VFA}$
is a model of $\ACFA$, using \ref{frobacfa}. 

For the converse, assume $K\models\ACFA$
is nontrivially valued and transformally Henselian. We want to prove that $K \models \widetilde{\VFA}$.
Since $K$ is transformally Henselian, the residue field lifts, and we can identify
$k$ with a relatively transformally algebraically closed difference
subfield of $K$; a difference field transformally algebraically closed
inside a model of $\ACFA$ is a model of $\ACFA$, so $K$ is strictly
transformally Henselian. Moreover, every nonconstant difference polynomial
has a root; so $\left(K^{\times}\right)^{\nu}=K^{\times}$ for all
$0\neq\nu\in\mathbf{Z}\left[\sigma\right]$. Finally, all additive
difference operators are onto, thus $K\models\widetilde{\VFA}$.
\end{proof}

\begin{rem}
Let $K$ be a valued field. Then $K$ is EC if and only if it is EC as a field (that is, algebraically closed) and nontrivially valued. But this is no longer true in the transformal settings: the condition that $K$ is transformally Henselian is necessary. See Example \ref{ex:optimal-axiom}.
\end{rem}

\begin{prop}
Let $\left(K_{20},K_{21},K_{10}\right)$ be a triplet of models of
$\VFA$ with $K_{21}$ and $K_{10}$ nontrivially valued. Then
$K_{20}$ is a model of $\widetilde{\VFA}$ if and only if $K_{21}$
and $K_{10}$ are so. In particular, the transformally Archimedean
ultrapower of a transformally Archimedean model of $\widetilde{\VFA}$
is a model of $\widetilde{\VFA}$.
\end{prop}

\begin{proof}
Let us assume that $K_{20}\models\widetilde{\VFA}$. We must prove that $K_{21}$ and $K_{10}$ are models of $\widetilde{\VFA}$.

The abstract difference field $K_{2}\models\ACFA$ and $K_{21}$ is transformally
Henselian by Lemma \ref{transitivetriplets}; as $K_{21}$ is by assumption
nontrivially valued, it is a model of $\widetilde{\VFA}$ using
Proposition \ref{ACFA-t-hens}. Moreover $K_{1}$ is isomorphic to
a difference subfield of $K_{2}$ which is transformally algebraically
closed in $K_{2}$; so $K_{1}$ is a model of $\ACFA$. Since $K_{10}$
is transformally Henselian, it follows that $K_{10}$ is a model of
$\widetilde{\VFA}$. For the converse, we know that $K_{20}$ is transformally
Henselian so it is enough to prove that $K_{2}\models\ACFA$; but
this follows from the fact that $K_{21}\models\widetilde{\VFA}$.

The last part follows immediately from the explicit construction of the transformally Archimedean ultrapower (and the fact that $\widetilde{\VFA}$ is preserved under ultrapowers)
\end{proof}
\begin{prop}
Let $K$ be a model of $\widetilde{\VFA}$. Then the completion
$\widehat{K}$ of $K$ is a model of $\widetilde{\VFA}$.
\end{prop}

\begin{proof}
By Proposition \ref{complete-t-henselian} the field $\widehat{K}$
is transformally Henselian; the residue field and the valuation groups
of $K$ and $\widehat{K}$ coincide so it is enough to prove that
the additive difference operators are onto on $\widehat{K}$. Since
$\widehat{K}$ is perfect and inversive, after twisting, we may assume
that $\tau'\neq0$. By scatteredness and rescaling we may assume that
all the roots of $\tau x$ are integral. Since $\mathcal{O}$ lies
dense in $\widehat{\mathcal{O}}$, approximating the coefficients
we can find $a\in\mathcal{O}$ with $v\tau a>\beta$ where $\beta$
is the valuation of the coefficient of the constant term; since $\widehat{K}$
is transformally Henselian, by Lemma \ref{newton} we find that $\tau$
admits a root in $\widehat{\mathcal{O}}$.
\end{proof}
\begin{cor}
Let $k$ be a model of $\ACFA$ and $t$ an element transformally
transcendental over $k$. Then there is a model $K\models\ACFA$ over
$k\left(t\right)_{\sigma}$ and such that no element of $K$ is transformally
algebraic over $k$, unless it already lies in $k$.
\end{cor}

\begin{proof}
Let us regard $k$ as a trivially valued model of $\VFA$ and
set $K=k\left(\left(t^{\mathbf{Q}\left(\sigma\right)}\right)\right)$;
then $K\models\widetilde{\VFA}$ and hence $K\models\ACFA$
by Proposition \ref{ACFA-t-hens}.
\end{proof}
\begin{example}
\label{ex:optimal-axiom}If $K$ is an algebraically closed and nontrivially
valued field, then $K$ is existentially closed as a valued field
and in particular Henselian. This is no longer true in the presence
of an automorphism; the requirement that $K$ be transformally Henselian
in Proposition \ref{ACFA-t-hens} is necessary. For example, let $K\models\widetilde{\VFA}$.
Let $x=a+t$ where $a^{\sigma}=a$ and $t$ is generic in $\mathcal{M}$
over $K$; assume moreover the residue class of $a$ is algebraically
transcendental over $k$. Let $L$ be a model of $\ACFA$ over $K\left(x\right)_{\sigma}$
with no element of $L$ transformally algebraic over $K$, unless
it already lies in $K$. The fixed field of $L$ must then coincide
with the fixed field of $K$. On the other hand, the fixed field of
$l$ properly extends the fixed field of $k$. So the fixed field
of $L$ does not lift to a trivially valued field; it follows that
$L$ cannot be transformally Henselian.

\end{example}
\newpage{}
\bibliographystyle{plain}

\end{document}